\documentclass[12pt]{amsart}
\usepackage[english]{babel}
\usepackage[utf8x]{inputenc}
\usepackage[T1]{fontenc}
\usepackage[]{float}
\usepackage[margin=2cm]{geometry}
\usepackage{empheq}
\usepackage{amsmath}
\usepackage{amsthm}
\usepackage{amssymb}
\usepackage{graphicx}
\usepackage[colorinlistoftodos]{todonotes}
\usepackage[colorlinks=true, allcolors=blue]{hyperref}

\newtheorem{theorem}{Theorem}[section]
\newtheorem{proposition}[theorem]{Proposition}
\newtheorem{corollary}{Corollary}[theorem]
\newtheorem{lemma}[theorem]{Lemma}
\newtheorem{remark}{Remark} [section]
\newtheorem{definition}{Definition} [section]

\DeclareMathSymbol{,}{\mathpunct}{operators}{"2C}
\DeclareMathSymbol{.}{\mathpunct}{operators}{"2E}
\numberwithin{equation}{section}

\newcommand{\ep}{\epsilon}
\newcommand{\be}{\begin{equation}}
\newcommand{\ee}{\end{equation}}
\newcommand{\p}{\partial}
\newcommand{\RP}{\text{Re}}
\newcommand{\IP}{\text{Im}}
\newcommand{\R}{\mathbb{R}}

\newcommand{\C}{\mathbb{C}}
\newcommand{\CI}{\mathcal{I}}
\newcommand{\CH}{\mathcal{H}}
\newcommand{\CR}{\mathcal{R}}
\newcommand{\CC}{\mathcal{C}}

\newcommand{\BF}{\mathbf{F}}

\newcommand{\N}{\mathbb{N}}

\newcommand{\CU}{\mathcal{U}}

\newcommand{\CF}{\mathcal{F}}
\newcommand{\BBU}{\mathbb{U}}
\newcommand{\BBY}{\mathbb{Y}}
\newcommand{\BBF}{\mathbb{F}}
\newcommand{\CK}{\mathcal{K}}

\newcommand{\CS}{\mathcal{S}}
\newcommand{\wt}{\widetilde}

\begin{document}
\begingroup
\allowdisplaybreaks

\title[Water waves linearized at shear flows]{On the spectra of the gravity water waves linearized at monotone shear flows}

\author[X. Liu]{Xiao Liu}
\address[X. Liu]{Department of Mathematics,  University of Illinois, 605 E Springfield Ave., Champaign, IL 61820}
\email{liu328@illinois.edu}

\author[C. Zeng]{Chongchun Zeng}
\address[C. Zeng]{School of Mathematics, Georgia Institute of Technology, Atlanta, GA 30332}
\email{zengch@math.gatech.edu}
\thanks{CZ is supported in part by the National Science Foundation DMS-1900083.}

\begin{abstract}
We consider the spectra of the 2-dim gravity  waves of finite depth linearized at a uniform monotonic shear flow $U(x_2)$, $x_2 \in (-h, 0)$, where the wave numbers $k$ of the horizontal variable $x_1$ is treated as a parameter. Our main results include a.) a complete branch of non-singular neutral modes $c^+(k)$ strictly decreasing in $k\ge 0$ and converging to $U(0)$ as $k \to \infty$; b.) another branch of non-singular neutral modes $c_-(k)$, $k \in (-k_-, k_-)$ for some $k_->0$, with $c_-(\pm k_-) = U(-h)$; c.) the non-degeneracy and the bifurcation at $(k_-, c=U(-h))$; d.) the existence and non-existence of unstable modes for $c$ near $U(0)$, $U(-h)$, and interior inflection values of $U$; e.) the complete spectral distribution in the case where $U''$ does not change sign or changes sign exactly once non-degenerately. In particular, $U$ is spectrally stable if $U'U''\le 0$ and unstable if $U$ has a non-degenerate interior inflection value or $\{U'U''>0\}$ accumulate at $x_2=-h$ or $0$. Moreover, if $U$ is an unstable shear flow of the fixed boundary problem in a channel, then strong gravity could cause instability of the linearized gravity waves in all long waves (i.e. $|k|\ll1$).
\end{abstract}
\maketitle

\section{Introduction}

Consider the two dimensional gravity water waves in the moving domain of finite depth 
\[
\CU_t=\{(x_1,x_2)\in \mathbb{T}_L \times \mathbb{R} \mid -h< x_2 < \eta(t,x)\}, \quad  \mathbb{T}_L:=\mathbb{R}/  L\mathbb{Z}, \; L>0,
\]
or 
\[
\CU_t=\{(x_1,x_2)\in \R \times \mathbb{R} \mid -h< x_2 < \eta(t,x)\}.
\]
%be the fluid domain at time t. 
The free surface is given by $S_t=\{(t,x) \mid x_2=\eta(t,x_1)\}$. 
For $x \in \CU_t$, let $v=(v_1(t,x), v_2(t, x)) \in \mathbb{R}^2$ denote the fluid velocity and $p=p(t,x)\in \mathbb{R}$ the pressure. They satisfy the free boundary problem of the incompressible Euler equation:
\begin{subequations} \label{E:Euler}
 \begin{empheq} 
 %[left={ }\empheqlbrace]
 {align}
  & \partial_t v+(v\cdot \nabla)v+ \nabla p+g\vec{e}_2=0,  & x\in \CU_t, \label{E:Euler-1} \\
  & \nabla \cdot v=0, &x\in   \CU_t,  \label{E:Euler-2}\\
  & \p_t \eta (t, x_1)=v(t,x)\cdot(-\p_{x_1} \eta (t, x_1),1)^T, & x\in  S_t, \label{E:Euler-3} \\
  & p(t, x)=0, & x\in S_t, \label{E:Euler-4}\\
  & v_2 (x_1, -h)=0, &  x_2=-h, \label{E:Euler-5}
\end{empheq}
\end{subequations}
where $g>0$ is the gravitational acceleration and the constant density is normalized to be 1.
The kinematic boundary condition \eqref{E:Euler-3} is equivalent to that $\partial_t+v\cdot \nabla$ 
% \parallel T\{(t,x)\mid 
is tangent to $\{(t, x) \mid x\in S_t\}$ at any $x\in S_t$, which means that $S_t$ moves at the velocity restricted to $S_t$. 

Shear flows are a fundamental class of stationary solutions of laminar flows 
\be \label{E:shear}
v_*:= \big(U(x_2), 0\big)^T, \quad S_*:=\{(t,x)|x_2=\eta_* (x_1) \equiv 0\}, \quad \nabla p_* = -g \vec{e}_2.
\ee 
Our goal in this paper is to analyze the spectral distribution (and thus the spectral stability and the linear instability) of the gravity water wave system linearized at a monotonic shear flow satisfying 
\be \tag{{\bf H}} 
U \in C^{l_0}([-h, 0]), \; l_0\ge 6; \quad \; U'(x_2) > 0, \ \forall x_2 \in [-h, 0].
\ee 

\begin{remark} \label{R:sign}
Due to the symmetry of horizontal reflection
\[
(x_1, x_2)\to (-x_1, x_2), \quad (v_1, v_2, \eta, p) \to (-v_1, v_2, \eta, p),  
\]
the case of $U'<0$ is completely identical except the sign assumptions on $U''$ and $U'''$ in the theorems 
%\ref{T:e-values} 
should be reversed.
\end{remark}

\subsection{Linearization and the spectral problem} \label{SS:Linearization and eigenvalues}

We first derive the linearized system of \eqref{E:Euler} at the shear flow $(v_*=(U(x_2), 0)^T, \eta_* =0)$ given in \eqref{E:shear}  satisfied and we denote the linearized solutions  by $(v, \eta, p)$.  
Let $(S_t^\epsilon, v^\epsilon (t, x), p^\ep (t, x))$ be a one-parameter family of solutions of \eqref{E:Euler} with $(S_t^0, v^0(t,x), p^0(t, x))=(S_*, v_*, p_*)$. 
Following the same procedure as in \cite{LiuZ21}, differentiating  the Euler equation system \eqref{E:Euler} with respect to $\ep$ and then evaluating it at $\ep=0$ yield 
\begin{subequations} \label{E:LEuler}
\be\label{E:LEuler-1}
    \partial_t v+U(x_2) \p_{x_1} v+(U'(x_2) v_2 , 0)^T +\nabla p=0, \quad  \;\; \nabla\cdot v=0, \quad \; x_2 \in (-h, 0), 
\ee
\be\label{E:LEuler-2}
 -\triangle p =2U'(x_2)\partial_{x_1}v_2, \quad x_2 \in (-h, 0),  \; \text{ and }  \;  \partial_{x_2}p|_{x_2=-h} =0. 
\ee
\be\label{E:LEuler-4} 
    \partial_t \eta= v_2|_{x_2=0} -U(0) \partial_{x_1}\eta.
\ee
\be\label{E:LEuler-5}
    p=g\eta, \quad \text{ at } \ x_2=0. 
\ee
\end{subequations}

Observe that the variable coefficients in the linear system \eqref{E:LEuler} depend only on $x_2$ and thus the Fourier modes of $x_1$ evolve in $t$ independently of each others. The zeroth Fourier mode corresponds to the perturbed shear flow component and can be handled easily, so the main interest lies in the $k$-th Fourier modes with $k\ne 0$. To find eigenvalues and eigenfunctions, we consider linearized solutions in the form of 
\be \label{E:e-func} 
v(t, x) = e^{ik(x_1 - c t)} \big(\hat v_{1}(x_2), \hat v_{2} (x_2) \big), \;\; \eta(t, x_1) = e^{ik(x_1 - c t)} \hat \eta, \;\; p(t,x) = e^{ik(x_1 - c t)} \hat p(x)
\ee
where apparently the eigenvalues take the form $\lambda =-ikc$ with the wave speed $c= c_R + i c_I \in \C$. In seeking solutions in the form of \eqref{E:e-func}, the wave number $k \in \R$ is often treated as a parameter. Substituting \eqref{E:e-func} into the linearization \eqref{E:LEuler}, through straight forward calculations (see \cite{Yih72, HL08} as well as \cite{LiuZ21}), one obtains the boundary value problem of the standard Rayleigh equation 
\begin{subequations} \label{E:Ray}
\be \label{E:Ray-1}
- \hat v_2'' + \Big( k^2 + \frac {U''}{U-c} \Big) \hat v_2 = 0,
\ee
with the boundary conditions 
\be \label{E:Ray-2}
\hat v_2 (-h)= 0, 
\ee
\be \label{E:Ray-3} 
\big((U-c)^2\hat v_2'-(U'(U-c)+g)\hat v_2\big)\big|_{x_2=0} =0.
\ee
\end{subequations} 
Here and throughout the paper $'$ denotes the derivative with respect to $x_2$. 
Apparently the last boundary condition at $x_2=0$ differs it from that of the linearized channel flow for $x_2 \in (-h, 0)$ with the slip boundary conditions at $x_2 =-h, 0$. It actually makes a substantial difference as we shall see below. For $k \ne 0$, $\hat v_1$ and $\hat \eta$ can be determined by $\hat v_2$ using the divergence free and the kinematic condition, respectively, 
\[
-k \hat v_1 + \hat v_2'=0, \qquad \hat \eta = \frac {\hat v_2(0)}{ik (U(0)-c))}. 
\]
So we shall mainly focus on $\hat v_2$. 

The linear system  \eqref{E:LEuler} is linearly unstable if there exist solutions $(k, c, \hat v_2)$ to \eqref{E:Ray} with $c_I>0$, which appear in conjugate pairs. Obviously the Rayleigh equation becomes singular when $c \in U([-h, 0])$, otherwise its initial value problem depends on $c$ and $k^2$ analytically and thus also even in $k$. We recall the following standard terminology. 

\begin{definition} \label{D:modes}
$(k, c)$ is a non-singular mode if $c \in \C \setminus U([-h, 0])$ and there exists a non-trivial solution $y(x_2)$ to \eqref{E:Ray} (thus also yields a solution to \eqref{E:LEuler} in the form of \eqref{E:e-func}). $(k, c)$  is a singular mode if $c \in U([-h, 0])$ and 
\[
(U-c) (-y''+ k^2 y) + U'' y=0
\]
has a non-trivial $H_{x_2}^2$ solution $y(x_2)$ satisfying the boundary conditions \eqref{E:Ray-2} and \eqref{E:Ray-3}. 
A non-singular mode is a stable (or unstable) mode if $c_I\le 0$ (or $c_I >0$). $(k, c)$ is a neutral mode if it is a singular or non-singular mode and $c \in \R$.
%\eqref{E:Ray} such that $\IP \, c_n >0$ and $k_n \to k$ and $c_n \to c$, as $n \to \infty$.  
%\[
%c_n \to c, \quad V_{2, n}(-h) \to 0, \quad \big((U-c)^2V_{2, n'}-(U'(U-c)+(g+\sigma k^2))V_{2, n}\big)\big|_{x_2=0} \to 0. 
%\] 
\end{definition}

In all above cases, \eqref{E:LEuler} has a solution in the form of \eqref{E:e-func}. 

We first present our main results and then comment on them along with the literatures.

\subsection{Main results} \label{SS:main}

In this paper we focus on the spectral distribution of the gravity waves linearized at a uniformly monotonic shear flow. It serves as the first step for us to understand the linearized flow at monotonic shear flows including the possible linear inviscid damping and the dispersive character (see \cite{LiuZ21} for capillary gravity waves). This would lay the foundation for the study of the local nonlinear dynamics. 
The following is a non-technical summary (assuming $U'>0$) of the main results mostly given in the following three main theorems. 
\begin{itemize}
\item There is a complete upper branch of neutral modes $c^+(k)> U(0)$ decreasing in $k>0$ and converging to $U(0)$ in the same fashion as the dispersion relation of the free gravity wave as $k \to +\infty$. 
\item Possibly two more branches of eigen-modes for $|k|\lesssim 1 $ and $|k|\gg1$, respectively.
\begin{itemize}
\item There is a branch of neutral modes $c_-(k)< U(-h)$ increasing in $k\in [0, k_-]$ and reaches $U(-h)$ at some $k_->0$ and then, near $c=U(-h)$ and for $0< k-k_- \ll1$, it either bifurcates into a (possibly broken) branch of unstable modes whose real parts are given by $\{ U(x_2) \mid U'(x_2) U''(x_2) >0\}$ near $U(-h)$ or the branch disappears completely if $U'U''<0$ near $x_2=-h$. 
\item Unless $U'U''<0$ near $x_2 =0$, there is a (possibly broken) branch $c^-(k)$ of "extremely weakly" unstable modes converging to $U(0)$ as $k\to +\infty$, whose real parts are given by $\{ U(x_2) \mid U'(x_2) U''(x_2) >0\}$ near $U(0)$. 
\end{itemize}
\item Inflection values of $U$ in $U\big((-h, 0)\big)$ always yield one or two singular neutral modes. Bifurcation to unstable modes occurs with $c$ on one side of any non-degenerate interior inflection value. 
\item $\{U''U'>0\}\ne \emptyset$ is necessary and also "almost" sufficient for linear instability. 
\item If $U''U'> 0$ on $[-h, 0]$, then $c^-(k)$ connects to $c_-(k)$ and, along with $c^+(k)$ and $\overline{c^-(k)}$, give all the eigen-modes (singular and non-singular modes) and \eqref{E:LEuler} is spectrally unstable for all $|k| > k_-$. 
\item If $U''U'\le 0$ on $[-h, 0]$, then  \eqref{E:LEuler} is spectrally stable and $c_-(k)$ for $|k| \le k_-$, and $c^+(k)$, are all the eigen-modes besides those at interior inflection values. 
\item If $U''$ changes from negative to positive exactly once non-degenerately, then $c^-(k)$ connects to the inflection value at $|k|=k_0>k_-$, and, along with $c^+(k)$ and $\overline{c^-(k)}$, give all the eigen-modes. Hence \eqref{E:LEuler} is spectrally unstable for all $|k| \in ( k_-, k_0)$.
\item If $U''$ changes from positive to negative instead, the eigenvalue distribution is also obtained. In particular, in some cases \eqref{E:LEuler}  with sufficiently strong gravity is spectrally unstable for all long waves (i.e. $|k| \ll 1$).  
\item Some improvements to those results in \cite{LiuZ21} on the spectral distribution of the capillary gravity water waves linearized at monotonic shear flows (Proposition \ref{P:CGWW}). 
\end{itemize}

The following first main theorem is on the above branches $c^\pm (k)$ and $c_-(k)$, the singular neutral modes, as well as the bifurcations near singular neutral modes. 

\begin{theorem} \label{T:e-values}
%({\bf Eigenvalues.}) 
Assume $U \in C^6$ and $U'>0$ on $[-h, 0]$, then singular neutral modes $(k, c)$ occur only if $c= U(-h)$ or $c$ is an inflection value of $U$. Moreover the following hold for a constant $C>0$ determined only by $U$. 
\begin{enumerate} 
\item There exists $k_0>0$ such that there exist an analytic function $c^+(k) \in (U(0), +\infty)$ and a $C^4$ function $c^-(k) = c_R^- (k) + ic_I^- (k)$ both defined for $|k|\ge k_0$ and even in $k$ and the following hold for any $|k| \ge k_0$. 
\begin{enumerate}
\item $c_I^-(k)$ and $U''(U^{-1} \big(c_R^-(k) \big)$ have the same signs ($+, -$, or $0$) and  
\begin{align*}
&\lim_{|k| \to \infty}  \big(c^\pm (k) - U(0)\big) /\sqrt{g/ |k|} =\pm 1, \\
%\quad |c_I^- (k)| \le C |k|^{-\frac 32} e^{ -\frac {2\sqrt{g|k|}}{U'(0)}} \big|U''(U^{-1} (c_R^-(k) \big)|.
& C^{-1} \big|U''(U^{-1} (c_R^-(k) \big)\big|\le |k|^{\frac 32} e^{ \frac {2\sqrt{g|k|}}{U'(0)}} |c_I^- (k)| \le C\big|U''(U^{-1} (c_R^-(k) \big)\big|.
\end{align*}
\item $(k, c)$ is a non-singular or singular mode iff i.) $c= c^+(k)$, ii.) $c = c^-(k)$ and $c_I \ge 0$, or iii.) $c = \overline{c^-(k)}$ and $c_I \le 0$. 
%all of which correspond to semi-simple eigenvalues $-ikc^\pm (k)$. 
\item $c^+(k)$ can be extended to an even and analytic function of $k \in \R$ with $c^+(k) \in \big( U(0), c^+ (0) \big]$ and $(c^+)'(k)<0$ for $k > 0$. Moreover, for any $k \in \R$, $c^+(k)$ is the only non-singular mode in $(U(0), +\infty)$ and it corresponds to  an 
%semi-simple 
eigenvalue $-ikc^+ (k)$ of \eqref{E:LEuler}. 
\end{enumerate}
\item There exist a unique $k_->0$ and a $C^{1, \alpha}$ (for any $\alpha \in [0, 1)$) even function $c_-(k) \in (-\infty, U(-h)]$ defined for $|k| \le k_-$ such that the following hold.
\begin{enumerate}
\item $c_-(k)<U(-h)$ is analytic in $k \in (-k_-, k_-)$, $c_- (\pm k_-) =U(-h)$, and $(c_-)'(k) >0$ for all $k \in (0, k_-]$.
\item  $(k, c)$ is a non-singular mode with $c \ne c^+(k)$ outside the disk \eqref{E:semi-circle} iff $|k| < k_-$ and $c = c_-(k)$. 
\item There exist $\ep, \rho>0$ such that $c_-(k)$ can be extended to $[-k_- -\ep, k_-+\ep]$ as a complex valued $C^{1, \alpha}$ even function  satisfying the following.
\begin{enumerate}
\item For $|k| \in [k_-, k_-+\ep]$, $\RP\, c_-' (k)>0$, $\IP\, c_- (k)$ and $U''\big(U^{-1} (\RP\, c_-(k))\big)$ have the same signs ($+, -$, or $0$), and 
%for $0\le \pm k -k_- \le \ep$,
\[
C^{-1} \big|U''(U^{-1} (\RP\, c_-(k) \big)\big| (|k| - k_-)^{2} \le   |\IP \, c- (k)| \le C\big|U''(U^{-1} (\RP \, c_-(k) \big)\big| (|k| - k_-)^{2}.
\]
\item $(k, c= c_R+ic_I)$ is a singular or non-singular mode of  \eqref{E:LEuler} with $|k-k_-| \le \ep$, $|c-U(-h)|\le \rho$, and $c_I\ge 0$ iff $c = c_-(k)$ ($c_I\le 0$ iff $c=\overline{c_-(k)}$). 
\end{enumerate}\end{enumerate}
\item Suppose $c_0=U(x_{20})$ with $x_{20} \in (-h, 0)$  and $U''(x_{20})=0$,  then the following hold. 
\begin{enumerate}
\item There exists $S \subset [0, +\infty)$ such that $|S|=1$ or $2$, 
%$S \cap (0, +\infty) \ne \emptyset$, 
$\max S>0$, and $(k, c_0)$ is a singular neutral mode of  \eqref{E:LEuler} iff $k \in S$.
\item 
%Let $k_*=\max S>0$. 
If, in addition, $U'''(x_{20})\ne 0$, then, for any $k_*\in S\setminus \{0\}$, there exist $\ep, \rho>0$ and a $C^4$  complex valued function $\CC(k)$ defined on $[-k_* -\ep, k_*+\ep]$ satisfying the following. 
\begin{enumerate} 
\item $\CC_I(k)$ has the same sign as $(k-k_*) U'''(x_{20})$ and 
\[
C^{-1} |k-k_*| \le |\CC_I(k)| \le C |k-k_*|.
\] 
\item $(k, c)$ is a singular or non-singular mode of  \eqref{E:LEuler} with $|k-k_*| \le \ep$, $|c-c_0|\le \rho$, and $c_I\ge 0$ iff $(k-k_*) U'''(x_{20}) \ge 0$ and $c = \CC(k)$ ($c_I\le 0$ iff $c=\overline{\CC(k)}$ and $(k-k_*) U'''(x_{20}) \ge 0$). 
\end{enumerate} \end{enumerate}
\item The linearized gravity water wave system \eqref{E:LEuler} is linearly unstable for some wave number $k \in \R$ if $U''$ has a non-degenerate zero point in $(-h, 0)$ or there exists a sequence $x_{2, n} \in (-h, 0)$ converging to $-h$ or $0$ as $n \to +\infty$ such that $U'' (x_{2,n})>0$ for all $n$.
\end{enumerate}
\end{theorem}

Clearly the above $c^-(k)$, $c_-(k)$, and $\CC(k)$ are relevant only when their imaginary parts are non-negative, so the subsets of these branches corresponding to eigenvalues are possibly broken. They correspond to unstable modes iff the imaginary parts are positive. From the estimates of the imaginary parts, the strength of the instability (the exponential growth rates) is the strongest near non-degenerate inflection values, and the weakest near $U(0)$.  

\begin{remark} 
a.) Due to the symmetry, the case of $U'<0$ (and also in the following theorems) is completely identical except the signs of $U''$ and $U'''$ should be reversed. \\
b.) More detailed asymptotics of $c^\pm (k)$ can be found in Lemma \ref{L:e-v-asymp-2}. \\
c.) See Lemma \ref{L:neutral-M-3} for more details on when $|S|=1$ or $2$ in statement (3) and its relationship to the singular neutral mode of the linearized channel flow. The statement $S \ne \emptyset$ had also been proved in \cite{Yih72, HL13} and we gave a different proof here.  It actually holds for a wider class of shear flows not necessarily monotonic, see Remark \ref{R:inflectionV}.  When $|S|=2$, %i.e. there are two wave numbers making the inflection value $c_0$ a singular neutral mode, 
see Lemma \ref{L:bifurcation-P} for  how a closed branch of unstable modes emerges from the smaller wave number and makes the linearized system unstable for $|k|\ll 1$. \\
d.) While $U'''(x_{20})\ne 0$ is the most easily checked non-degeneracy condition in statement (3b), a more precise condition and more details of the bifurcation (including the case of $k_*=0$) can be found in Lemma \ref{L:bifurcation}.  \\
e.) Local bifurcation in statement (4) from an interior inflection value, similar to \cite{LiuZ21}, can be compared to that in \cite{HL08} for a different class of shear flows. Also see comment below. 
 \end{remark} 

In the next theorem we consider the case where $U''$ does not change sign. 

\begin{theorem} \label{T:e-values-1}
Assume $U \in C^6$ and $U'>0$ on $[-h, 0]$, then the following hold. 
%for a constant $C>0$ determined only by $U$. 
\begin{enumerate} 
\item Suppose $U''(x_2) >0$ for all $x_2 \in (-h, 0)$, then the above $c^- (k)$ (given in Theorem \ref{T:e-values}(1)) and $c_-(k)$ (Theorem \ref{T:e-values}(2)) can be extended for all $k \in \R$ and they coincide, with $c_I^- (k) = 0$ for all $|k| \le k_-$ and $c_I^- (k) >0$ for all $|k| > k_-$. Moreover, all singular and non-singular modes of  \eqref{E:LEuler} are given by $(k, c^\pm (k))$ and $(k, \overline {c^-(k)})$ for $k \in \R$.
\item The linearized gravity water wave system \eqref{E:LEuler} is spectrally stable for all wave number $k\in \R$ if $U''\le 0$ on $[-h, 0]$. Moreover $(k, c)$ is a neutral mode iff $c= c^+(k)$, $c = c_-(k)$ and $|k|\le k_-$, or $c \in U((-h, 0))$ is an inflection value of $U$ and $|k| \in S$ given in Theorem \ref{T:e-values}(3). 
\end{enumerate}
\end{theorem}

\begin{remark} 
Under the assumption $U'' < 0$  (and $U'(0)\ge 0$), the spectral stability had been established in \cite{Yih72} (see also \cite{HL08}). The above spectral stability statement (2) can also be proved under the relaxed assumptions $U''\le 0$ on $[-h, 0]$ and $U'(0)>0$. See Remark \ref{R:stability}.
\end{remark}

Finally, the last theorem is for the case where $U$ has exactly one non-degenerate inflection point in $(-h, 0)$.  

\begin{theorem} \label{T:e-values-2}
Assume 
\be \label{E:K+}
\{x_2 \in (-h, 0) \mid U''(x_2)=0\} = \{x_{20}\}, \quad c_0= U(x_{20}), 
\ee
then the following hold.  
\begin{enumerate} 
\item If $U''' (x_{20}) >0$, then  $c^-(k)$ can be extended evenly for all $|k| \ge k_0$, where $S=\{k_0\}$ (see Theorem \ref{T:e-values}(1)(3)), such that $c^-(k_0) = c_0$ and $c_I^- (k) >0$ for all $|k|> k_0$. Moreover all singular and non-singular modes of  \eqref{E:LEuler} are given by $(k, c^+ (k))$ for $k \in \R$, $(k, c^-(k))$ and $(k, \overline {c^-(k)})$ for $|k| \ge k_0$, and $(k, c_-(k))$ for $|k| \le k_-$. 
\item If $U'''(x_{20})<0$ and $S=\{k_0\}$, then $k_0 > k_-$ and $c_-(k)$ (given in Theorem \ref{T:e-values}(2)) can be extended evenly for  $|k| \le k_0$ such that $c_-(k_0)=c_0$, $\IP\, c_-(k) >0$ for $|k| \in (k_-, k_0)$, and all singular and non-singular modes of  \eqref{E:LEuler} are given by $(k, c^+ (k))$ for $k \in \R$ and $(k, c_- (k))$ and $(k, \overline {c_-(k)})$ for $|k| \le k_0$.
\item If $U'''(x_{20})<0$, $S=\{k_0>k_1\}$, and $k_1 \le k_-$, then $k_0 > k_-$ and we have the following. 
\begin{enumerate} 
\item $c_-(k)$ can be extended evenly for $|k| \le k_0$ such that $c_-(k_0)=c_0$ and $\IP\, c_-(k) >0$ for $|k| \in (k_-, k_0)$.
\item  There exists a $C^4$ even complex valued function $\CC(k)$ defined for $|k| \in [0, k_1]$ such that $\CC(\pm k_1) =c_0$ and $\CC_I(k) >0$ for all $|k| \in [0, k_1)$. 
\item All singular and non-singular modes of  \eqref{E:LEuler} are given by $(k, c^+ (k))$ for $k \in \R$, $(k, c_- (k))$ and $(k, \overline {c_-(k)})$ for $|k| \le k_0$, and $(k, \CC(k))$ and $(k, \overline {\CC(k)})$  for $|k| \le k_1$.
%\item \eqref{E:LEuler} is unstable iff $|k| < k_1$ and $|k| \in (k_-, k_0)$.
\end{enumerate}
\item If $U'''(x_{20})<0$, $S=\{k_0>k_1\}$, and $k_1 > k_-$, then $k_0 > k_-$ and \eqref{E:LEuler} is unstable iff $|k| < k_0$. Moreover there exist connected components $\CS_j$, $j=-1, 0, 1$, of 
\[
\CS= \{ \text{non-singular modes } (k, c) \in \R \times \C \mid c_I>0 \} \cup \{(\pm k_-, U(-h)),\, (\pm k_0, c_0),\, (\pm k_1, c_0)\},   
\]
such that 
%$\CS_0$ is symmetric in $k$ and 
\[
\CS=\CS_{-1} \cup \CS_0 \cup \CS_1, \quad \CS_1 = \{ (k, c) \mid (-k, c) \in \CS_{-1}\}, \quad  \{ (k, c) \mid (-k, c) \in \CS_{0}\}= \CS_0,
\]
and either 
\begin{enumerate} 
\item $(k_-, U(-h)), (k_0, c_0) \in \CS_1$ and $(\pm k_1, c_0) \in \CS_0$, or 
\item $(k_-, U(-h)), (k_1, c_0) \in \CS_1$ and $(\pm k_0, c_0) \in \CS_0$. 
\end{enumerate}\end{enumerate}
\end{theorem}

\begin{remark} 
a.) An equivalent condition to the existence of $k_1$ is given in Lemma \ref{L:neutral-M-3}. \\
b.)  In this theorem, the case of $U'''(x_{20})<0$ corresponds to the intersection of increasing and the so-called class $\CK^+$ (see \cite{HL08}) shear flows. \\
c.) In particular, \eqref{E:LEuler} becomes unstable for all long waves (i.e. $|k|\ll 1$) if $k_1 >0$. \\
d.) Assuming $U'''(x_{20})<0$ and $U$ is an unstable shear flow of the channel flow with fixed boundaries. Letting $g$ increase, the system could deform from case (2) to (3) and then possibly to (4), see Lemma \ref{L:neutral-M-3}. In Lemma \ref{L:bifurcation-P} which holds under a weaker assumption, one can see a closed  branch of unstable modes bifurcates from $c_0$ as $1\gg k_1>0$. As this branch grows, it may or may not intersect the branch $c_-(k)$ in case (2), resulting in the two possibilities in case (4) where the three subsets may even coincide and be equal to $\CS$.  
\end{remark}

\subsection{Backgrounds and discussions}

Due to its strong physical and mathematical relevance there have been extensive studies of the Euler equation   linearized at shear currents. In particular the linear instability is often viewed as the first step in understanding the transition of the fluid motion from the laminar flows to turbulent ones. Much of the existing analysis was on the fluid in a fixed channel with slip boundary conditions 
\be \label{E:E-channel}
\text{\eqref{E:Euler-1}--\eqref{E:Euler-2} with } g=0, \;\; \; x_2 \in (-h, 0), \quad v_2 (x_1, 0)= v_2(x_1, -h)=0, 
\ee
and some of the results have been extended to free boundary problems. 
%such as the gravity waves. The spectral analysis is naturally a crucial part of such linear systems. Eigenvalues yield linear solutions exponential in time, while the continuous spectra often lead to algebraic decay of solutions, the so-called inviscid damping due to the lack of a priori dissipation mechanism of the Euler equation.   

%\noindent 
$\bullet$ {\it Linearized channel flows.} Classical results on the spectra of the Euler equation \eqref{E:E-channel} in a channel  linearized at a shear flow
% in a fixed channel 
include:
\begin{itemize}
\item Unstable eigenvalues  are isolated for each $k \in \R$ and do not exist for $|k|\gg 1$.
\item Rayleigh's necessary condition of instability \cite{Ray1880}: unstable eigenvalues do not exist for any $k$ if $U'' \ne 0$ on $[-h, 0]$ (see also \cite{Fj50}).
\item Howard's Semicircle Theorem \cite{How61}: eigenvalues exist only with $c$ in the disk 
\be \label{E:semi-circle}
\big(c_R-\tfrac 12 (U_{max}+U_{min})\big)^2+c_I^2\leq \tfrac 14 (U_{max}- U_{min})^2. 
\ee
\item Unstable eigenvalues may exist with $c$ near inflection values of $U$ (Tollmien \cite{To35} formally, also \cite{LinC55}). 
\end{itemize}
Many classical results can be found in books such as \cite{DR04, MP94} {\it etc.} For a class of shear flows, the rigorous bifurcation of unstable eigenvalues was proved, e.g., in \cite{FH98, Lin03}. In particular, continuation of branches of unstable eigenvalues were obtained by Lin in the latter. 

$\bullet$ {\it Linearized gravity water waves.} It has been extended to the linearized free boundary problem of the gravity  waves at certain classes of shear flows (see \cite{Burns53, Hunt55, Ben62, Yih72, Shr93, LH98, HL08, KR11, RR13, BR13} {\it etc.}) that: a.) the semicircle theorem still holds for unstable modes; 
%singular neutral modes exist in $U\big((-h, 0)\big)$ iff $U$ has inflection values; 
b.) the bifurcation and continuation of branches of unstable modes starting from limiting singular neutral modes; c.) the stability if $U''<0$ and $U'(0)\ge0$, {\it etc.} Compared to channel flows with fixed boundaries, new phenomena of the linearized gravity waves include: 
\begin{itemize} 
\item In addition to inflection values, critical values of $U$ where $U'=0$, and $c=U(-h)$ may be limiting singular neutral modes. 
\item $U'' \ne 0$ may not ensure the spectral stability. 
\item There are non-singular neutral modes with $c$ in both $(-\infty, U(-h))$ and $(U(0), +\infty)$ which are outside the circle \eqref{E:semi-circle}. 
\item $U(-h)$ can be a singular neutral mode for certain wave number. 
\end{itemize} 
However, the rigorous bifurcation analysis at $U(-h)$ or critical values of $U$ was still missing. This holds the key how Rayleigh's stability condition may fail in the gravity wave case. Moreover we have not been aware of rigorous studies of the branch of unstable modes near $c=U(0)$ for $|k| \gg 1$. These have been addressed in the above Theorem \ref{T:e-values}. 
%Along with the analytic continuation argument, the complete distribution of eigenvalues are obtained in Theorems \ref{T:e-values-1} and \ref{T:e-values-2} when $U''$ changes the sign at most once.  

In particular, we would like discuss some of our results (assuming $U'>0$) in relation to those in \cite{Yih72, HL08, HL13, RR13}. 

* {\it Singular neutral modes and branches from bifurcations.} Firstly we do not only prove that $U(-h)$ and internal inflection values are always singular neutral modes for some wave numbers as in \cite{Yih72, HL08, HL13}, but also obtain the exact numbers of those wave numbers (see Theorem \ref{T:e-values} and Subsection \ref{SS:int-N-M}). For an internal inflection value $c_0=U(x_{20})$, it turns out that whether it becomes a singular neutral mode at one or two wave numbers depends on i.) whether it is a singular neutral mode for the linearized channel flow and also on ii.) whether $g$ is greater than some threshold. Secondly our bifurcation analysis is carried out in whole neighborhoods of singular neutral modes and allows broken branches, instead of in  cones of neighborhoods as in \cite{Lin03, HL08}. This is more than a technical improvement, as it gives both the existence and the number of nearby unstable modes which is crucial for the index counting used in the analytic continuation argument. In particular, when  the above i) holds, strong gravity could creates a closed branch of unstable modes with $|k| \le k_1$ for some $k_1>0$ (Lemma \ref{L:bifurcation-P} and Theorem \ref{T:e-values-2}). This is contrary to the common expectation that gravity stabilizes long waves.  Finally, when $U''$ changes the sign at most once, these ingredients allow us to obtain the complete eigenvalue distribution by identifying how different branches of unstable modes connect to each other. Some of these branches had been observed in numerics \cite{RR13}. 

* {\it Comparison to numerics.} Three examples of shear flows were computed in \cite{RR13} which we shall adapt to our notations to avoid unnecessary confusions. In the last two examples, $U'>0$ for $x_2 \in (-2, 0)$, $U''(x_2)=0$ iff $x_2=-1$, and $c_0= U(-1)=0$ is  the only inflection value.  Their second example is $U(x_2) = \tanh a(x_2+1)$ and $a=1/2$ or $2$, where clearly $U'''(-1) <0$. The former is stable in the channel flow case, while the latter is unstable. For $a=1/2$, a branch of unstable modes corresponding to $c_-(k)$ in  Theorem \ref{T:e-values-2}(2) was found. For $a=2$, two branches of unstable modes corresponding to those in either (3) or (4a) of Theorem \ref{T:e-values-2} were observed. In their third example, $U(x_2) = 1+ x_2 + (1+x_2)^3/2$ where $U'''(-1) >0$, and the branch of unstable modes corresponds to that in Theorem \ref{T:e-values-2}(1).
 
%\noindent
$\bullet$ {\it Comparison with the linearized capillary gravity waves at monotonic shear flows.} 
In this paper we follow a strategy to study the spectra of the linearized gravity waves \eqref{E:LEuler} similar to that of the linearized capillary gravity waves   in \cite{LiuZ21}. The analysis is based on a detailed understanding of the Rayleigh equation \eqref{E:Ray-1}, particularly near its singularity where $U(x_2) -c \sim 0$. Subsequently, we consider the solution $y_-(k, c, x_2)$ to \eqref{E:Ray-1} satisfying the normalized initial condition $y_-(-h)=0$ and $y_-'(-h)=1$, which are analytic in $k \in \R$ and  $c \in \C \setminus U([-h, 0])$. Such a solution gives rises to an eigenfunction iff the boundary condition \eqref{E:Ray-3} is satisfied. This is formulated into an equation $F(k, c)=0$ where $F$ is defined in \eqref{E:dispersion} whose regularity is carefully studied. 
%It is  and $C^{1, \alpha}$ when restricted in $c_I \ge 0$. 
In Subsections \ref{SS:Ray} and \ref{SS:e-values-0}, we outline those basic analysis developed in \cite{LiuZ21} needed in this paper. To identify the roots of $F$, the key ingredients are a.) the asymptotic properties of $F$ as $|k|, |c| \to +\infty$ and $c_I \to 0$; b.) the bifurcations for $c$ near $U([-h, 0])$; and c.) the analytic continuation of the roots of $F$. Here $k$ is treated as a parameter in both b.) and c.). 

As revealed in Theorems \ref{T:e-values} and \ref{T:e-values-1}, the spectral distributions are significantly different   when the surface tension is not present. Moreover, the analysis in the current paper also substantially improves some results for capillary gravity waves (see Proposition \ref{P:CGWW}).

On the one hand, while one might expect that the water waves tend to be more unstable when there is no surface tension, particularly in high wave numbers, the main differences in the spectral distributions at monotonic shear flows include the following. 
\begin{itemize}
\item [(i)]  For $|k| \gg 1$, capillary gravity waves are linearly stable, while the gravity waves are unstable if the set $\{U'U''>0\}$ accumulates at $x_2=0$. 
\item [(ii)] While $c=U(-h)$ and interior inflection values of $U$ may or may not be a singular neutral mode of the linearized capillary gravity waves depending on the combined strength of the gravity and surface tension, they are {\it always} singular neutral modes of the linearized gravity waves for one or two $k\ge 0$ (for a unique $k>0$ in the case of $c= U(-h)$).  
%In contrast to the linearized capillary gravity waves, the system \eqref{E:Ray} for the eigenfunction at an interior inflection value is a Sturm-Liouville problem. 
\item [(iii)] As a consequence, $\{ U'U'' >0\} \ne \emptyset$ is "almost" sufficient for the instability of the linearized gravity waves (Theorem \ref{T:e-values}(4)), which is not always the case for the capillary gravity waves.  
\item [(iv)] Contrary to the natural expectation, a.) when the surface tension is very weak with the coefficient $0< \sigma\ll 1$, it actually creates instability near non-degenerate inflection values in relatively large wave  numbers $k \sim \sigma^{-1}$ (Proposition 4.11 in \cite{LiuZ21}), and b.) strong gravity may also cause instability in all small wave numbers. 
\end{itemize}
Related to (i), one is reminded that, for shear flows, the Rayleigh-Taylor sign $-\frac {\p p}{\p n} \big|_{x_2=0} = g>0$ at any shear flow implies that the nonlinear gravity waves are locally well-posed for initial data near shear flows and there is no unbounded exponential growth rates in high waves numbers. However, when $U'U''>0$ near $x_2=0$, the linearization at the shear flow is still unstable for all $|k|\gg1$, but the exponential growth rates $O\big(|U''| |k|^{-\frac 12} e^{-\frac {2\sqrt{g|k|}}{U'(0)}}\big)$ of those unstable linear solutions (Theorem \ref{T:e-values}(1a)) are extremely weak. 

%It would be interesting to see the effect of these weak instabilities in the nonlinear local dynamics. 

When the combination of the surface tension and the gravity is sufficiently strong, the linearized capillary gravity wave at a monotonic shear flow is essentially the superposition of a linearized channel flow with fixed boundaries (corresponding to the continuous spectra) and a dispersive system (corresponding to the non-singular modes and the surface motions) with two branches of the dispersion relation given by $c^\pm(k) \notin U([-h, 0])$ (see \cite{LiuZ21}). Without surface tension, the complete branches of the non-singular modes alway reach and also possibly converge fast to $U([-h, 0])$. This interaction between the continuous spectra and eigenvalues would make the the linear inviscid damping (see \cite{LiuZ21}) a very subtle issue except in the simplest case where $U'U'' < 0$ near $x_2=0$ and the period in $x_1$ eliminates all the singular neutral modes. We would address the inviscid damping in a future work. 

These differences in the spectra could potential result in substantial differences in the nonlinear local dynamics.  

On the other hand, the bifurcation analysis of unstable eigenvalues requires certain regularity and non-degeneracy of $F(k, c)$, particular in the most subtle case near $c=U(-h)$. The regularity and non-degeneracy at $c=U(-h)$ in \cite{LiuZ21} were verified under the additional assumption $U'' \ne 0$ on $[-h, 0]$. In Subsections \ref{SS:Y} and \ref{SS:non-deg}, we prove these regularity and non-degeneracy without any assumptions additional to the monotonicity of $U$. Therefore the work in this paper also substantially improves the bifurcation analysis in the capillary gravity wave case, see Proposition \ref{P:CGWW} in Subsection \ref{SS:CGWW}.

Compared to \cite{LiuZ21}, in this paper we also thoroughly analyzed the case where $U''$ changes sign exactly once. 

$\bullet$ {\it The spectra of the linearized fluid interface problem.} 
Another related problem is the two phase fluid problem linearized at a shear flow $U(x_2) = U^+ (x_2) \chi_{x_2 \in (0, h_+)} + U^- (x_2) \chi_{x_2 \in (-h_-, 0)}$ where the slip boundary condition is assumed at $x_2 = \pm h_\pm$ . The famous Kelvin Helmholtz instability was first identified in the simplest setting where $U^\pm = consts$ \cite{kelvin1871hydro, DR04, BR13}. In the case where the upper fluid is much lighter than the bottom one (the air-water interface, for example), Miles \cite{Mi57, Mi59a, Mi59b, Jan04} proposed the critical layer theory to model how the wind (shear flow) in the air above a stationary water generates waves through linear instability due to the resonance between  the shear flow in the air and the temporal frequency of the linear irrotational capillary gravity water waves. This was later rigorously proved in \cite{BSWZ16}. More recently, the spectra of the linearized fluid interface problem at monotonic shear flows  had been studied more comprehensively and in more details in \cite{Liu22}.      \\

\noindent {\bf Notations.} Throughout the paper, 
%we use the Japanese bracket $\langle x \rangle = (1+x^2)^{\frac 12}$. 
wave numbers are always denoted by "$k$" and $K=k^2$ is also used. For complex quantities, the subscripts "R" and "I" denote their real and imaginary parts, respectively. Sometimes "$\RP$" and "$\IP$" are also used when the subscripts make the notations too cumbersome.

\section{Preliminary analysis} \label{S:Pre}

In this section we review some basic results on the homogeneous Rayleigh equation and the equation satisfied by the eigenmodes $(k, c, y(x_2))$. Most of these results have already been obtained in \cite{LiuZ21} and the readers are referred there for more details.

\subsection{Results on the homogeneous Rayleigh equation 
%(mostly from  \cite{LiuZ21})
} \label{SS:Ray}

%The homogeneous Rayleigh equation
%\begin{subequations} \label{E:Ray-H1}
\be \label{E:Ray-H1-1}
-y''(x_2)+ \big(k^2 + \tfrac{U''(x_2)}{U(x_2)-c}\big)y(x_2)=0, \quad x_2\in [-h,0],
\ee
where
\[
'= \p_{x_2}, \quad k\in \R, \quad c=c_R + i c_I \in \C,
\]
had been thoroughly analyzed in Section 3 of \cite{LiuZ21}. In this subsection, we summarize those  results which we shall use here. When necessary some modifications will be outlined to address the different needs in this paper.
Throughout this subsection we assume 
\be \label{E:U'}
U'(x_2)>0, \quad \forall x_2 \in [-h, 0]. 
\ee
As pointed out in Remark \ref{R:sign}, the case of $U'<0$ can be reduced to the above one. Hence all results under \eqref{E:U'} hold for all uniformly monotonic $U(x_2)$, namely those $U$ satisfying $U'\ne 0$ on $[-h, 0]$. The solutions to the Rayleigh equation \eqref{E:Ray-H1-1} 
are obviously even in $k$ and thus $k\ge 0$ will be assumed mostly. Similarly complex conjugate of solutions also solve \eqref{E:Ray-H1-1}  with $c$  replaced by $\bar c$, hence we shall consider the case of $c_I > 0$. Most of the estimates are uniform in $c$ with $c_I>0$ and thus hold in the limiting case as $c_I \to 0+$. 
%We shall first consider the case of $c \in \C \setminus U([-h, 0])$ and then let $c \to U([-h, 0])$ with 
The dependence on $k\gg1$ will also be carefully tracked. 

Recall $U \in C^{l_0}$. For convenience we extend $U$ to be a $C^{l_0}$ function defined in a neighborhood $[-h_0-h, h_0]$ of $[-h, 0]$, where 
\be \label{E:h0}
h_0 = \min\Big\{\frac h2, \,  \frac {\inf_{[-h, 0]} U'}{4 |U''|_{C^0([-h, 0])}}
%, \, \frac {\inf_{[-h, 0]} U''}{4 |U'''|_{C^0([-h, 0])}} 
\Big\} >0,  
\ee
%$h_0>0$ can be chosen depending only on $|U''|_{C^1([-h, 0])}$, $\inf U'$, and $\inf U''$ on $[-h, 0]$, 
so that  
\be \label{E:U-ext}
%U'' \ge \tfrac 12 \inf_{[-h, 0]} U''(x_2), \;\; 
|U'|_{C^{l_0-1}([-h_0-h, h_0])} \le 2|U'|_{C^{l_0-1} ([-h, 0])}, \quad U' \ge \tfrac 12 \inf_{[-h, 0]} U'(x_2), \quad \text{ on } \;  [-h_0-h, h_0].
\ee
In the analysis of the most singular case of $c$ near $U([-h, 0])$, we let $x_2^c$ be such that 
\be \label{E:x2c}
c_R = U(x_{2}^c), \; \text{ if } c_R \in U([-h_0-h, h_0]).
\ee

Consider the solution  $y_- (k, c, x_2)$ to \eqref{E:Ray-H1-1} satisfying the initial conditions 
\be \label{E:y-pm} \begin{split}
y_- (-h) =0, \quad y_-'(-h) =1, 
\end{split} \ee
where, as often in the rest of the paper, the dependence on $k$ and $c$ are skipped to make the notations simpler when there is no confusion in the context. Apparently this is a non-trivial solution and the following lemma holds. 

\begin{lemma} \label{L:basic-y_pm}
For any $k\in \R$ and $c \in \C \setminus U([-h, 0])$, $y_- (k, c, x_2) = y_- (-k, c, x_2) = \overline {y_- (k, \bar c, x_2)}$ is analytic in $c$ and $k^2$. 
\end{lemma} 

%Those above exceptional $c$ corresponds to the neutral or unstable/stable modes of the Rayleigh equation \eqref{E:Ray}, which are isolated in $\C \setminus U([-h, 0])$ due to the analyticity of \eqref{E:Ray-H1-1} in $c  \in \C \setminus U([-h, 0])$. 

To give some basic results on $y_- (k, c, x_2)$, we first introduce some subintervals of $[-h, 0]$, 
\be \label{E:CIs} \begin{split} 
&\CI_2 := (x_{2l}, x_{2r}) = \left\{ x_2 \in [-h, 0]\, : \,   \frac 
%{|U''|_{C^0([-h_0-h, h_0])}}
1{|U(x_2)-c|} >  
%\min \{k^2, k^{\frac 32}\} = \rho k^{2} 
\rho_0 \mu^{-\frac 32} \right\}, 
\quad \CI_1 = [-h, x_{2l}), \\
& \CI_3=(x_{2r}, 0], \quad \mu= \langle k \rangle^{-1} = (1+k^2)^{-\frac 12}, \quad 
\rho_0 = \frac 
%{4 |U''|_{C^0([-h_0-h, h_0])}}
4{h_0 \inf_{[-h_0-h, h_0]} U'}.
\end{split} \ee
Clearly $[-h, 0]= \CI_1 \cup \CI_2 \cup \CI_3$ where any of the three subintervals may be empty. In particular, if $\CI_2 =\emptyset$, then we consider $[-h, 0]$ as $\CI_1$ 
in the statement of the following lemma. The choice of the above constant $\rho_0$ and the fact $0\le \mu \le 1$ ensure 
\be \label{E:rho_0}
c_R \in U([-\tfrac 14 h_0-h, \tfrac 14 h_0]) \; \text{ if } \; \CI_2 \ne \emptyset. 
\ee
The following is a part of Lemma 3.9 in \cite{LiuZ21}. 

\begin{lemma} \label{L:y_pm-1}
For any $\alpha  \in (0, \frac 12)$, there exists $C >0$ depending only on $\alpha$, $|U'|_{C^2}$, and $|(U')^{-1}|_{C^0}$, such that, for any $c \in \C \backslash U([-h, 0])$, it holds
\be \label{E:y-_1}
|\mu^{-1} y_- (x_2) - \sinh (\mu^{-1}(x_2+h))|   \le C \mu^{\alpha}\sinh (\mu^{-1}(x_2+h)),
\ee
%\be \label{E:y+_1}
%|\mu^{-1} y_+ (x_2) - \sinh (\mu^{-1}x_2)|   \le C\big(\mu^{\alpha} +\mu |c|^2 \big)\cosh (\mu^{-1}x_2), 
%\ee
for all $x_2 \in [-h, 0]$. In addition, if $\CI_2 =\emptyset$, 
%or $x_{2l} =x_{2r}$,
%\be \label{E:kc-away-3}
%|c_I|  \ge |U''|_{C^0} \max\{1, k^{-\frac 12}\} k^{-\frac 32} \; \text{ or } \;   
%%x_2 \in \CI_1,
%%0 \le x_2^c - x_2 \le |(U')^{-1}|_{C^0} |U''|_{C^0} \max\{1, k^{-\frac 12}\} k^{-\frac 32}, 
%\ee
then for all $x_2\in [-h, 0]$, 
\be \label{E:y-_2}
|y_-'(x_2) - \cosh (\mu^{-1}(x_2+h))|  \le C \mu^{\alpha}\sinh (\mu^{-1}(x_2+h)),  
%\quad x_2 \in [-h, 0],   
\ee
%\be \label{E:y+_2}
%|y_+'(x_2) - \cosh (\mu^{-1}x_2)|  \le  C \big(\mu^{\alpha} + \mu |c|^2 \big)\cosh (\mu^{-1}x_2).
%%\mu^{\alpha}|\sinh (\mu^{-1}x_2)|.  
%%\quad x_2 \in [-h, 0],   
%\ee
Otherwise if $\CI_2 \ne \emptyset$, 
%and $x_{2r} > x_{2l}$, 
%\be \label{E:kc-not-away}
%|c_I|  \le |U''|_{C^0} \max\{1, k^{-\frac 12}\} k^{-\frac 32} \; \text{ and } \;  ,
%%x_2 \in \CI_1,
%%0 \le x_2^c - x_2 \le |(U')^{-1}|_{C^0} |U''|_{C^0} \max\{1, k^{-\frac 12}\} k^{-\frac 32}, 
%\ee
then 
\be \label{E:y-_3}
|y_-'(x_2) - \cosh (\mu^{-1}(x_2+h))|  \le
\begin{cases}  
C \mu^{\alpha}\sinh (\mu^{-1}(x_2+h)),  \qquad & x_2 \in \CI_1\\
C\mu^{\alpha}\cosh (\mu^{-1}(x_2+h)),  & x_2 \in \CI_3
 %\quad x_2 \in [-h, 0],   
\end{cases} \ee
%\be \label{E:y+_3}
%|y_+'(x_2) - \cosh (\mu^{-1}x_2)|  \le 
%%\begin{cases} 
%C\mu^{\alpha}\cosh (\mu^{-1}x_2), \qquad  x_2 \in \CI_1 \cup \CI_3, 
%%\\
%%\quad x_2 \in [-h, 0],   
%%C \mu^{\alpha}\cosh (\mu^{-1}x_2)
%%C \mu^{\alpha}|\sinh (\mu^{-1}x_2)|,  
%%& x_2 \in \CI_3, 
%%\end{cases} 
%\ee
and for $x_2 \in \CI_2$, 
\be \label{E:y-_4} \begin{split}
\big|y_-'(x_2) - \cosh (\mu^{-1}(x_2+h)) &- \frac {U''(x_2^c)}{U'(x_2^c)} y_-(x_{2l})\log |U(x_2) - c|  \big| \\
& \le C \mu^{\alpha} \big( 1 + \mu \big| \log |U(x_2) - c|\big| \big)\cosh (\mu^{-1} (x_2+h)).   
\end{split} \ee
%\be \label{E:y+_4} \begin{split}
%\big|y_+'(x_2) - \cosh (\mu^{-1} x_2) &- \frac {U''(x_2^c)}{U'(x_2^c)} y_+ (x_{2r})\log |U(x_2) - c|  \big| \\
%& \le C \mu^{\alpha} \big( 1  + \mu  \big| \log |U(x_2) - c|\big| \big)\cosh (\mu^{-1} x_2).   
%\end{split} \ee
\end{lemma} 

\begin{remark} \label{R:y-pm}
Even though $c \notin U([-h, 0])$ is assumed in the lemma, the estimates are uniform in $c$ and hence they also hold for the limits of solutions as $c_I \to 0+$, while the limits as $c_I \to 0-$ are the conjugates of those as $ c_I \to 0+$. 
%This remark also applies to the case $k \le k^*$ considered in Lemma \ref{L:y-pm-0}. Moreover, the constant $C$ for $y_-$ does not depend on $g$.  
\end{remark}

To study the limits as $c$ converges to the singular set $U([-h, 0])$, let 
\be \label{E:y0}
y_{0-} (k, c, x_2) = \lim_{c_I \to 0+} y_- (k, c+ic_I, x_2), \quad c \in U([-\tfrac {h_0}2-h, \tfrac {h_0}2]).
%x_2 \in [-h, 0].   
\ee
The following lemma (a combination of parts of Lemmas 3.10, 3.12, 3.13, 3.14, 3.15, 3.16, and 3.19 in \cite{LiuZ21}) gives some basic properties of $y_{0-}(k, c, x_2)$ including their dependence in $c\in U([-\tfrac {h_0}2-h, \tfrac {h_0}2])$. 

\begin{lemma} \label{L:y0}
Assume $U \in C^{l_0}$, $l_0\ge 3$. For $c = U(x_2^c) \in U([-\tfrac {h_0}2-h, \tfrac {h_0}2])$ and $x_2 \in [-h, 0]$, the following hold, where $C>0$ is determined only by $U$.
\begin{enumerate} 
\item As $c_I \to 0+$, $y_- (k, c + ic_I, x_2) \to y_{0-} (k, c, x_2)$ uniformly in $x_2$ and $c$, $y_-' \to y_{0-}'$ locally uniformly in $\{U(x_2) \ne c\}$.
\item For each $c$, $y_{0-} (x_2) \in \R$ if $x_2 \le x_2^c$. Moreover  $y_{0-} \in C^\alpha ([-\tfrac {h_0}2-h, \tfrac {h_0}2 ])$ for any $\alpha \in [0, 1)$ and $y_{0-}$ is $C^{l_0}$ in $x_2\ne x_2^c$. 
%and they  
%%$C^{l_0-2}$ in $k\in \R$, 
%satisfy \eqref{E:Ray-H0-3} with $\phi=0$. 
\item There exists a unique continuous-in-$\tau$ real $2\times 2$ matrix valued $B(\mu, c, \tau)$ determined only by $U$ such that 
\[
\begin{pmatrix} \tfrac 1\mu y_{0-} (x_2)  \\ y_{0-}' (x_2) \end{pmatrix}= B \big(\mu, c, \tfrac 1\mu(x_2-x_2^c)\big) \begin{pmatrix} 1 & 0 \\ \Gamma_0^\#\big(\mu, c, \tfrac 1\mu(x_2-x_2^c)\big) & 1\end{pmatrix}\begin{pmatrix} \tfrac 1\mu y_{0-} (x_2^c) \\ b_{2-} \end{pmatrix}. 
\]
Here, 
\[
\Gamma_0^\# (\mu, c, \tau) = \tau + \frac {\mu U''(x_2^c)}{U'(x_2^c)} \big( \log |\tau| + \frac {i\pi}2 (sgn(\tau)+1) \big). 
\]
%\[
%\gamma_0(\mu, c_R, \tau) = \int_0^\tau \frac{\mu^2 \big(U'(x_2^c) U'' (x_2^c + \mu \tau') - U''(x_2^c) U' (x_2^c + \mu \tau')\big) }{U'(x_2^c)\big(U(x_2^c + \mu \tau') - U(x_2^c) \big) } d\tau',
%\]
\[
b_{2-} = \lim_{x_2 \to x_2^c} \Big( y_{0-}' (x_2) - \tfrac {U''(x_2^c)}{U'(x_2^c)} y_{0-} (x_2^c) \big( \log \big(\tfrac{U'(x_2^c)}\mu  |x_2 - x_2^c|\big) + \tfrac {i\pi}2 (sgn(x_2-x_2^c)+1)  \big) \Big)  \text{ exists},
\]
and $B(\mu, c, \tau)$ is  
%at $x_2 \ne x_2^c$, and 
$C^{l_0-2}$ in $c\in U\big([-\tfrac 12 h_0-h, \tfrac 12 h_0]\big)$, $\tau$, and $\mu$ and 
\[
\det B=1, \quad B(\mu, c, 0)=I_{2\times 2}, \quad B(0, c, \tau) =  \begin{pmatrix} \cosh \tau -\tau \sinh \tau & \sinh \tau \\ \sinh \tau - \tau \cosh \tau  & \cosh \tau \end{pmatrix}. 
%=  \begin{pmatrix} \cosh \tau & \sinh \tau \\ \sinh \tau  & \cosh \tau \end{pmatrix}\begin{pmatrix} 1 & 0 \\ -\tau & 1\end{pmatrix}. 
\]
\item For any $k \in \R$, $y_{0-} (k, c, x_2)>0$ for any $x_2\in (-h, 0]$ and $c\in \R \setminus U\big((-h, 0)\big)$.
\item  For any $k , c\in \R$, it holds 
\be \label{E:y-lower-b-1} 
(Ck)^{-1} \sinh k(x_2+h) \le y_{0-} (k, c, x_2) \le C k^{-1} \sinh k(x_2+h)
%\; if \; \big(U(-h)-c\big)\big(U(x_2)-c\big)\geq 0,
\ee
%and $(k, c, x_2) >0$ for any $c \in U\big((-h, 0]\big)$ and $x_2 \in (-h, x_2^c]$.   
%\item There exists $C>0$ determined only by $U$ such that, if $x_2^c \in [-h, 0)$, then  it holds,  for any $k \in \R$, 
for $(x_2, c)$ satisfying $\big(U(-h)-c\big)\big(U(x_2)-c\big)\geq 0$ and moreover, for $c \in U\big((-h, 0) \big)$, 
\be \label{E:y-lower-b-2} 
C^{-1}|U''(x_2^c)|
%\mu^2 |U''(x_2^c)| \sinh \mu^{-1}(x_2^c+h) \sinh \mu^{-1} |x_2^c| 
\le \frac {| \IP\, y_{0-} (k, c, 0)|}{\mu^2  \sinh \mu^{-1}(x_2^c+h) \sinh \mu^{-1} |x_2^c|} \le C |U''(x_2^c)|. 
% \mu^2\sinh \mu^{-1}(x_2^c+h) \sinh \mu^{-1} |x_2^c|.
\ee
\item $y_{0-}$ is locally $C^\alpha$ in both $k$ and $c$ for any $\alpha \in [0, 1)$ and $(y_{0-}, y_{0-}')$ are $C^\infty$ in $k$ at any $(k, c, x_2)$ except for $y_{0-}'$ at $c = U(x_2)$.
\item $(y_{0-}, y_{0-}')$ are locally $C^\alpha$ in both $k$ and $c$ for any $\alpha \in [0, 1)$ at any $(k, c, x_2)$ satisfying $U(x_2) \ne c$.  
\item $(y_{0-}, y_{0-}')$
are $C^{l_0-2}$ in both $k$ and $c$ at any $(k, c, x_2)$ satisfying  $U(x_2) \ne c$ and $c \ne U(-h)$ . Moreover, the following estimates hold for $c$ in a neighborhood of $U([-h, 0])$ (uniform in $k$), where $\rho_0$ is given in \eqref{E:CIs}.
\begin{enumerate} 
\item For any $j_1, j_2\ge 0$ with $j_1+j_2>0$ and $x_2 \in \CI$, where $\CI \subset [-h, 0]$ is any interval satisfying $ -h \in \CI \subset  \{ |U-c| \ge \rho_0^{-1} \mu\}$, 
%$x_2 \in [-h, 0]$ with $c- U(x_2) \ge \rho_0^{-1} \mu$ and , 
we have 
\be \label{E:pcy1} \begin{split} 
\mu^{-1} |\p_k^{j_1} \p_c^{j_2} y_{0-} (x_2)| + |\p_k^{j_1} \p_c^{j_2} y_{0-}'(x_2)| \le& C_{j_1, j_2} \mu  \big(|U(x_2) -c|^{-j_2} + |U(-h) -c|^{-j_2}\big) \\
&\times \big( 1+ \mu^{-j_1} (x_2+h)^{j_1} \big) \sinh (\mu^{-1} (x_2+ h)),
%\le & C \mu^{1-j_1-j_2} \sinh (\mu^{-1} (x_2+ h)).
\end{split}\ee
where $C_{j_1, j_2}>0$ depends only on $j_1$, $j_2$,  $|U'|_{C^2}$, and $|(U')^{-1}|_{C^0}$. 
\item Suppose $[-h, 0] \not\subset \{ |U-c| \ge \rho_0^{-1} \mu\}$ and $l_0\ge 5$, then, for $x_2 \in [-h, 0]$ satisfying $U(x_2) -c \ge \rho_0^{-1} \mu$,   
\be \label{E:pcy2} 
%\begin{split}
%\begin{align*}
\mu^{-1} |\p_c y_{0-} (x_{2})| + |\p_c y_{0-}' (x_{2})| \le C \Big(1 +  \log \frac \mu{\min\{\mu,  |U(-h)-c|\}} \Big) \cosh (\mu^{-1}(x_{2}+h)), 
%|\mu^{-1} \p_c^j y_{0-} (x_2)| + |\p_c^j y_{0-}' (x_2)| \le C\mu^{1-j} \cosh (\mu^{-1}(x_{2}+h)), \quad 1\le j\le l_0-2, \; \text{ if } \; x_{2l} >-h,
%\begin{cases}    
\ee
and for $2\le j \le l_0-4$, 
\[
%\be \label{E:pcy2-1}
\tfrac 1\mu |\p_c^j y_{0-} (x_{2})| + |\p_c^j y_{0-}' (x_{2})| \le C\big( \mu^{1-j} +  |U(-h)-c|^{1-j}  \big) \cosh (\mu^{-1}(x_{2}+h)).
%\mu^{-1} |\p_c y_{0-} (x_2)| \le C \big( 1+ \big|\log \big( \mu^{-1} |U(-h)-c|\big)\big|\big) \sinh (\mu^{-1} (x_{2}+h)), 
%\quad  x_{2l}=-h,   
%\end{cases} \end{split}
\]
\end{enumerate}
%d.) $y_{0-}$ and $y_{0+}$ are $C^{l_0-2}$ in $k$, while $y_{0-}'$ and $y_{0+}'$ are $C^{l_0-2}$ in $k$ if $U(x_2) \ne c$ and 
%\be \label{E:pky}
%\p_k y_{0\pm}'(k, c, x_2) = O\big(1 + \big|\log |U(x_2) -c|\big| \big), \; \text{ as } \, U(x_2)-c \to 0;
%\ee
\item $y_{0-} (k, c, x_2^c)$ is $C^{l_0-2}$ in $c$ and $k$ for $c \in U([-h, 0])$. In addition,  for $1\le j \le l_0-2$, it holds 
\be \label{E:pcy4}
| \p_c^j  \big(y_{0-} (k, c, x_2^c)\big) |  \le  C \mu^{1-j} \cosh (\mu^{-1}(x_2^c +h)).
% \; \text{ if } x_2^c \in [-h, 0]. 
\ee
\end{enumerate}
\end{lemma}

\begin{remark} 
In many cases, to simplify the notation, we denote $y_- (k, c, x_2) = y_{0-} (k, c, x_2)$ for $c\in \R$.  In \eqref{E:y-lower-b-1}, $k^{-1} \sinh kx$ is understood as $x$ when $k=0$. The regularity of $y_{0-}$, the analyticity of $y_-$ and $y_-'$ in $k$ and $c$ with $c_I>0$, and its convergence as $c_I \to 0+$ imply $y_- (k, c, x_2)$ is also $C^{l_0-2}$ in $k$ and $c \notin \{U(x_2), U(-h)\}$  when restricted to $c_I\ge 0$.  
\end{remark}

%The analyticity of $y_-$ in $c$ and their asymptotic behaviors as $|c|\to \infty$ allow us to represent $y_-(k, c, x_2)$ using their values with $c \in \R$. The following is a part of Lemma 3.12 in \cite{LiuZ21}.  

%\begin{lemma} \label{L:pcy-complex}
%%It holds that $\p_c y_-$ is analytic in 
%Assume $U\in C^6$. As $c_I \to 0+$, $y_- (k, c + ic_I, x_2) \to y_{0-} (k, c, x_2)$ uniformly in $x_2$ and $c$, $y_-' \to y_{0-}'$ locally uniformly in $\{U(x_2) \ne c\}$.
%%with a singularity of $O\big(\big| \log |U(x_2) -c_R| \big|\big)$ near $U(x_2) = c_R$. 
%\end{lemma}

%\begin{remark} \label{R:l_0-1}
%It suffices to assume $U\in C^3$ in statements (1) and (3). The regularity of $y_{0-}$, the analyticity of $y_-$ and $y_-'$ in $k$ and $c$ with $c_I>0$, and its convergence as $c_I \to 0+$ imply $y_- (k, c, x_2)$ is also $C^{l_0-3}$ in $k$ and $c \notin \{U(x_2), U(-h)\}$  when restricted to $c_I\ge 0$.  
%\end{remark}

Finally,  the following quantity related to the Reynolds stress is crucial for the linearized water wave problem: 
\be \label{E:Y} \begin{split}
&Y(k, c) = Y_R (k, c) + iY_I (k, c) := \frac {y_{-}' (k, c, 0)}{y_{-} (k, c, 0)},  \quad  c = c_R + i c_I \in \C \setminus U([-h, 0]), \\
%\in \mathbb{C}\backslash U\big((-h, 0]\big);  \\ 
&Y(k, c) = \lim_{\ep \to 0+} Y(k, c+i\ep)= \frac {y_{0-}' (k, c, 0)}{y_{0-} (k, c, 0)}, \quad c\in U\big([-h, 0)\big),
\end{split} \ee
with the domain 
\[
D(Y) = \{ (k, c) \in \R \times \C \mid c\ne U(0), \ y_- (k, c, 0)\ne 0\}.
\] 
Those excluded points (except $c= U(0)$) exactly correspond to the eigenvalues of the linearized Euler equation in the fixed channel $x_2 \in (-h, 0)$ at the shear flow $U(x_2)$. The end point $c = U(0)$ is also excluded due to the singularity of $y_-'(k, c=U(0), x_2)$ at $x_2 =0$. The following lemma (Lemmas 3.20, and 3.22 and parts -- some in the proofs -- of Lemmas 3.23, 3.24, and 4.5 in \cite{LiuZ21}) summarizes some basic properties of $Y(k, c)$. 

\begin{lemma} \label{L:Y-def}
Assume $U\in C^{l_0}$, $l_0\ge 3$. It holds that $Y(k, \bar c) = Y(-k, c)= \overline {Y(k, c)}$ and  $Y$ is a.) analytic in $(k, c) \in D(Y)\setminus (\R \times U([-h, 0]))$; and, when restricted to $c_I\ge 0$, b.) $C^{l_0-2}$ in $(k, c) \in D(Y) \setminus (\R \times \{U(-h)\})$, and c.) $C^{\infty}$ in $k$ and locally $C^\alpha$ in $(k, c) \in D(Y)$  for any $\alpha \in [0, 1)$.  Moreover, 
\begin{enumerate}
\item $Y(k, c) \in \R$ for all $c \in \R \setminus U\big((-h, 0]\big)$ and $Y(0, U(-h)) =\frac {U'(0)}{U(0)-U(-h)}$.
%for any $k_* >0$, 
\item There exists $C, \rho>0$ depending only on $U$ such that 
%for any $k \in \R$ and $|c- U(0)|\le \rho$, 
\[
|Y(k, c)| \le C \big(\mu^{-1} + \big|\log  \min \big\{1, |U(0) -c|\big\}\big| \big), \; \;  \forall k\in \R, \; |c - U(0)| \le \rho. 
%; \quad |Y(k, c)| \le C \big(k + \big|\log |U(0) -c|\big|\big), \;\; \forall |c - U(0)| \le C k^{-1}.  
\]
\item For any $\alpha \in (0, \frac 12)$, there exist $k_0>0$ and $C>0$ depending only on  $\alpha$, $|U'|_{C^2}$, and $|(U')^{-1}|_{C^0}$ such that,  
\[
|Y(k, c) - k \coth kh| \le C (\mu^{\alpha-1} + |\log \min\{1, \, |U(0)-c|\}|), \quad \forall |k| \ge k_0, \; c \ne U(0).  
\]
\item For any $M>0$ and $k_*>0$, there exists $C>0$ depending on $k_*$ and $M$ such that 
\[
|Y(k, c) - k \coth kh| \le \frac C{ dist(c, U([-h, 0]))}, 
\]
if $k$ and $c$ satisfy 
\[ 
|k| \le k_* \; \text{ and } \; \Big|c - \frac {U(-h) + U(0)}2\Big|\ge M + \frac {U(0)-U(-h)}2. 
\]
\item $Y_I (k, c)=0$ for $c\in \R\backslash U\big((-h, 0]\big)$ and 
%%Suppose $k\in \R$ such that \eqref{E:no-emb-ev-CF} holds, 
%Assume $c = U(x_2^c) \in U\big((-h, 0)\big)$ and $y_-(k, c, 0) \ne 0$, then  
%%$Y(k, c)$ is $C^{l_0-2}$ in $c$ and $k$ for $c \notin \{U(-h), U(0)\}$, $C^\alpha$ for any $\alpha \in [0, 1)$ at any $c \ne U(0)$, 
\[
Y_I (k, c) =\frac {\pi U''(x_2^c) y_{-} (k, c, x_2^c)^2}{U'(x_2^c) |y_{-} (k, c, 0)|^2}, \quad c \in \CI \triangleq \{ c\in U\big( (-h, 0)\big) \mid y_-(k, c, 0) \ne 0\}.   
\]  
Moreover, the above formula implies $Y_I(k, c)$ is $C^\infty$ in $k$ and $C^{l_0-2}$ in $(k, c) \in D(Y) \cap \big(\R \times U\big((-h, 0)\big)\big)$.  
\item Assume $l_0\ge 4$, then for any $q\in [1, \infty)$, $j_1, j_2\ge 0$, $j_2 \le 2$, and $j_1+j_2\le l_0-4$, $\p_k^{j_1} \p_{c_R}^{j_2} Y_I$ is $L_k^\infty W_{c_R}^{1, q}$ locally in $(k, c) \in D(Y) \cap \big(\R \times U\big([-h, 0)\big)\big)$.
\item Assume $k \in \R$ satisfies $y_-(k, c, 0)\ne 0$ for all $c \in \C$, then,  for any $c\in \C \backslash U([-h, 0])$,  
\be \label{E:Y-Cauchy-1} \begin{split}
&Y(k, c)=\frac{1}{\pi}\int_{U(-h)}^{U(0)} \frac{Y_I (k, c')}{c'-c}dc'+k\coth kh, \\
%& \p_c Y(k, c)=\frac{1}{\pi}\int_{U(-h)}^{U(0)} \frac{Y_I (k, c')}{(c'-c)^2}dc'=\frac{1}{\pi}\int_{U(-h)}^{U(0)} \frac{\p_{c_R} Y_I (k, c')}{c'-c}dc' -\frac {U''(0)}{U'(0) \big(U(0)-c\big)},
\end{split} \ee
and for $c\in U\big([-h, 0)\big)$, 
\be \label{E:Y-Cauchy-2} \begin{split} 
&Y(k, c)= - \CH \big(Y_I(k, \cdot)\big) (c) + i Y_I(k,c) +k\coth kh.
%& \p_c Y(k, c)= - \CH \big(\p_{c_R} Y_I(k, \cdot)\big) (c) + i \p_{c_R} Y_I(k,c)-\frac {U''(0)}{U'(0) \big(U(0)-c\big)}.
\end{split} \ee 
\item For any $k\in \R$ and $c\in \R \setminus U\big((-h, 0]\big)$
\[
\p_K \big(Y\big(\sqrt{K}, c \big)\big) >0, \quad \p_K^2 \big(Y\big(\sqrt{K}, c \big)\big) <0, 
\]
\[
(U(0)-c)^2\p_K Y(0, c) \le (U(0) - U(-h))^2 \p_K Y(0, U(-h)) = \int_{-h}^0 (U(x_2) - U(-h))^2 dx_2,
\]
where $K=k^2$ and the equal sign in the second inequality happens if and only if $c = U(-h)$. 
\end{enumerate} 
\end{lemma}

Here we recall that $\CH$ denotes the Hilbert transform, namely,
\[
\CH (f) (\tau) =\frac{1}{\pi} \text{P.V.} \int_\R \frac{f(s)}{\tau-s}ds, 
\] 
 for any function $f(s)$ defined on $\R$, where P.V.$\int$ is the principle value of the singular integral. 
%Statement (6) is basically Lemma 3.23(1) in \cite{LiuZ21}. Even though Lemma 3.23  in \cite{LiuZ21} is under an additional assumption $y_-(k, c, 0)\ne 0$ for all $c\in U([-h, 0])$, the same proof easily yields the same result when restricted to $c \in U\big([-h, 0)\big) \cap D(Y(k, \cdot))$. 
%In particular, the proof of statement (8) is based on the following claim which will also be used in Section \ref{S:e-values}.

%{\it Claim.} Assume $c_0\in \C \setminus U\big((-h, 0)\big)$ and $\tilde y, y\in C^0([-h, 0])\cap C^2((-h, 0))$ are solutions to 
%\[
%(\CR+k^2) \tilde y=0, \; \tilde y(-h)=0, \; \tilde y(0)=1; \quad (\CR+k^2) y = f\in C^0([-h, 0)), \; y(-h)=y(0)=0; 
%\]
%%with the above inflection value $c_0 = U(x_{20}) $ of $U$, where $f \in C^0([-h, 0])$,  
%then we have the following through direct computations using $(\tilde y'  y - \tilde y  y')' = \tilde y f$, \be \label{E:F-signs-0}
%%  \Rightarrow  
%y'(0) =- \int_{-h}^0 \tilde yf dx_2, \quad 
%y (x_2) =  \tilde y(x_2) \int_{x_2}^0 \frac 1{\tilde y (x_2')^{2}} \int_{-h}^{x_2'} \tilde y(x_2'') f(x_2'') dx_2'' dx_2'.
%\ee
%Here Lemma \ref{L:y0}(4) insures that $\tilde y = \frac {y_-}{y_-(0)} >0$ over $(-h, 0]$.  

\subsection{Basic properties of eigenvalues} \label{SS:e-values-0}

From \eqref{E:e-func}, $-ikc \in \C$, with $c \in \C \setminus U([-h, 0])$, is an eigenvalue of \eqref{E:LEuler} in the $k$-th Fourier modes iff 
\be \label{E:BF} \begin{split} 
& \BF (k, c) = \BF_R + i \BF_I =  \big(U(0)-c\big)^2 y_-'(k, c, 0) - \big(U'(0)\big(U(0)-c\big)+ g\big) y_-(k, c, 0) =0.  
\end{split} \ee
 In the limit as $c \to U([-h, 0])$, let 
\[
\BF(k, c)= \lim_{\ep \to 0+} \BF(k, c+i\ep) = \overline{\lim_{\ep \to 0+} \BF(k, c-i\ep)}, \quad c \in U\big([-h, 0]\big). 
\]
Clearly, $y_-(k, c, x_2)$ also generates the corresponding eigenfunction of \eqref{E:LEuler}
if $\BF(k, c)=0$. The roots $c$ of $\BF$ with $c_I>0$ are often referred to as unstable modes, while those roots $c\in \R$ as neutral modes. We recall that Yih proved that the semicircle theorem also holds for \eqref{E:LEuler}, i.e., \eqref{E:semi-circle} holds for all unstable modes \cite{Yih72}.

Since $\BF$ may not be $C^1$ in $c$ near $c=U(-h)$ (see Lemma \ref{L:y0}) which would turn out to be a key bifurcation point, we also consider an actually equivalent equation 
\be \label{E:dispersion}
F(k, c) = y_-(k, c, 0)^{-1} \BF = F_R + i F_I = Y(k, c)\big(U(0)-c\big)^2-U'(0)\big(U(0)-c\big)-g =0,  
\ee
with $Y(k, c)$ defined in \eqref{E:Y}, and 
\[
F(k, c)= \lim_{\ep \to 0+} F(k, c+i\ep) = \overline{\lim_{\ep \to 0+} F(k, c-i\ep)},  \quad  c \in U\big([-h, 0]\big). 
\]
Clearly it holds 
\be \label{E:BF-conj}
\BF(-k, c) = \BF(k, c) = \overline {\BF(k, \bar c)}, \; \forall c\notin U\big((-h, 0)\big); 
%\quad \BF(k, c) \in \R, \; \forall c\in \R\setminus U\big((-h, 0)\big); 
\ee
\be \label{E:f-conj}
F(-k, c) = F(k, c) = \overline {F(k, \bar c)}, \;  c\in D(Y)\setminus U\big((-h, 0)\big). 
%\quad F(k, c) \in \R, \; \forall c\in (D(Y)\cap \R) \setminus U\big((-h, 0)\big);
\ee
%and enjoys at least the same regularity as $Y$ given in Lemma \ref{L:Y-def}, which is 
According to the following Lemma \ref{L:Y-1}, which is an improvement of Lemma 3.24 in \cite{LiuZ21}, $F$ is $C^{1, \alpha}$ near $c=U(-h)$ (without additional assumptions, unlike Lemma 3.24 in \cite{LiuZ21}). 
Obviously $\BF$ and $F$ are analytic in their domains except for $c \in U([-h, 0])$. 
  
Our strategy to analyze the eigenvalue distribution includes the following key ingredients:
\begin{itemize} 
\item asymptotic analysis of eigenvalues for $|k|\gg 1$,
%there are exactly three branches of eigenvalues $-ikc^\pm (k)$ and $-ik \overline{c^-(k)}$ from the unstable mode $c^- (k)$ and the non-singular neutral mode $c^+(k)$, 
which turn out to accumulate at $U(0)$;
\item the existence and bifurcation of (possibly unstable) eigenvalues for $c$ near singular neutral modes at $c_0 =U(-h)$ or interior inflection values of $U$;
\item an analytic continuation argument on the extension of the non-singular modes $c(k)$ away from the above bifurcation points.   
\end{itemize}
This strategy had been successfully applied to analyze the eigenvalue distribution of the linearized capillary gravity  water waves at monotonic shear flows in Section 4 in \cite{LiuZ21}. The absence of the surface tension does not affect some basic properties of $\BF$ and $F$, some of the bifurcation analysis of unstable eigenvalues, or the continuation argument. In the rest of this subsection, we shall outline the basic properties and the continuation argument obtained in \cite{LiuZ21}. The bifurcation analysis will be given in Section \ref{S:e-values} with a similar approach, but under considerably relaxed assumptions.  The existence of singular neutral modes and the eigenvalues for $|k|\gg1$ will also be studied in Section \ref{S:e-values}, where we shall see the phenomena would turn out to be substantially different from the capillary gravity waves. \\

\noindent $\bullet$ We first give some {\bf elementary properties} of $\BF$ and $F$ for a monotonic shear flow $U$, starting with some relatively qualitative properties. 

\begin{lemma} \label{L:e-v-basic-1}
Assume $U \in C^{l_0}$, $l_0\ge 3$, then for any $k \in \R$, the following hold. 
\begin{enumerate}
\item $\BF$ is well defined for $\forall (k ,c) \in \R \times \C$. When restricted to $c_I\ge 0$, $\BF$  is $C^\infty$ in $k$ and is $C^{l_0-2}$ in both $k$ and $c \notin \{U(-h), U(0)\}$ and $\BF$ is also $C^\alpha$ in both $k$ and $c$ with $c_I\ge 0$. The same holds for $F(k, c)$ except at $(k, c)$ where $y_-(k, c, 0)=0$. 
%In particular,  
%\[\BF(k, U(0)) = - (g+ \sigma k^2) y_-(k, U(0), 0) < 0. \]
\item $\BF(k, c) \ne 0$ if $y_-(k, c, 0)=0$ and thus $\{ \BF(k, c)=0\} = \{ F(k, c) =0\}$.  
\item  $\BF(k, c)=0$ iff there exists a $C^2$ solution $y(x_2)$ to \eqref{E:Ray-H1-1} satisfying the corresponding homogeneous boundary conditions of (\ref{E:Ray-2}-\ref{E:Ray-3}).
\item For any $x_2^c\in (-h, 0)$ and $c = U(x_2^c)$, 
\[
F_I(k, c) = \big(U(0)-c\big)^2 Y_I(k, c) = \frac {\pi (U(0)-c)^2 U''(x_{2}^c) y_{-} (k, c, x_{2}^c)^2}{U'(x_{2}^c) |y_{-} (k, c, 0)|^2},
\]
and thus $\BF_I(k, U(x_2))\ne 0$ if $x_2 \in (-h, 0)$ and $U''(x_2)\ne 0$. 
\end{enumerate}
\end{lemma}

These statements are mainly contained in Lemma 4.1 in \cite{LiuZ21} where one could easily see that the lack of surface tension (namely, $\sigma=0$) does not affect the proof. The next lemma is on some quantitative properties of $\BF$ and $F$. 

\begin{lemma} \label{L:e-v-basic-2}
Suppose $U \in C^3$. The following hold for any $k \in \R$.  
\begin{enumerate}
\item  $F(k, c)$  is $C^\infty$ in $k$ and is well-defined for $c$ close to $U(-h)$ and $U(0)$, $C^1$ near $c=U(0)$, and 
\[
F(k, U(-h)) \in \R, \quad F\big(0, U(-h)\big) =F\big(k, U(0)\big)=  -g, \quad 
\p_c F\big(k, U(0)\big)= U'(0),
%\quad, = \p_c F\big(k, U(-h) \big)=0.
\] 
\[
F(0, c) =  \frac 1{\int_{-h}^0 (U-c)^{-2} dx_2} -g = \frac {(U(0) - c)(U(-h)-c)}{y_-(0, c, 0)}-g, \quad c  \in \C\setminus U((-h, 0)).
\]
\item Let $K=k^2$, we have 
\[
\p_K \big(F\big(\sqrt{K}, c \big)\big) >0, \quad \p_K^2 \big(F\big(\sqrt{K}, c \big)\big) <0, \;\ \forall k\in \R, \; c\in \R \setminus U\big((-h, 0]\big),
\]
\[
\p_K F(0, c) \le \int_{-h}^0 \big(U(x_2) - c \big)^2 dx_2, \quad \forall c \in \R \setminus U((-h, 0)),  
\]
where ``='' occurs only at $c= U(-h)$. 
%\item Assume $U \in C^6$, then (a) $\p_c F(k, U(-h)) <0$ for all $k \in \R$ if $U''>0$ on $[-h, 0]$; and (b) if $U''<0$ on $[-h, 0]$, then $\p_c F(k, U(-h)) <0$ if $F(k, U(-h)) =0$, where $F$ is understood as restricted to $c_I \ge 0$. 
\end{enumerate}
\end{lemma}

Again, as the missing surface has no impact on the proofs, the results in this lemma are mostly contained in the statements and the proofs of Lemmas 4.1 and 4.5 in \cite{LiuZ21}, except $\p_K F>0$ in (2) follows directly from Lemma \ref{L:Y-def}(8). In particular, the form of $F(0, c)$ for $c \in \C\setminus U([-h, 0])$ was obtained in Section 4 in \cite{LiuZ21} based on the explicit formulas   
%in the proof of Lemma 4.6.
\be \label{E:y_-0} \begin{split}
& y_- (0, c, x_2) = (U(x_2)-c) \int_{-h}^{x_2} \frac {U(-h)-c}{(U(x_2')-c)^2} dx_2', \\ 
&Y(0, c) = \frac{U'(0) \int_{-h}^0 (U-c)^{-2} dx_2 + (U(0)-c)^{-1} }{ (U(0)-c) \int_{-h}^0 (U-c)^{-2} dx_2}, 
\end{split} \qquad c \in \C\setminus U([-h, 0]). 
\ee

\noindent $\bullet$ As $\BF$ and $F$ are analytic functions in their domains outside $c \in U([-h, 0])$, the {\bf analytic continuation} argument is a standard tool in the study of the spectra of the linearized Euler equation at shear flows. The following lemma also applies to $F(k, c)$ due to Lemma \ref{L:e-v-basic-1}(2). 

\begin{lemma} \label{L:continuation}
Assume $U \in C^3$. Suppose $k_0\in \R$ and $c_0 \in \C \setminus U([-h, 0])$ satisfy $\BF(k_0, c_0)=0$ and $\p_c \BF(k_0, c_0)\ne 0$, then the following hold. 
\begin{enumerate}
\item There exists an analytic function $c(k) \in \C \setminus U([-h, 0])$ defined on a max interval $ (k_-, k_+) \ni k_0$ such that $\BF\big(k, c(k)\big)=0$ and $\p_c \BF\big(k, c (k)\big)\ne 0$. 
\item $c(k) \in \R$ for all $k \in (k_-, k_+)$ if and only if $c_0 \in \R$.  
\item If $k_+ < \infty$ (or $k_- > -\infty$), then 
\begin{enumerate} \item 
$\lim_{k\to (k_+)-} dist(c(k), U([-h, 0])) = 0$ (or $\lim_{k\to (k_-)+} dist(c(k), U([-h, 0])) = 0$ if $k_->-\infty$), or 
\item $\liminf_{k\to (k_+)-} \min \{ |c(k) -c| \, : \, \forall c \text{ s. t. } \BF(k, c)=0, \, c\ne c(k) \} = 0$ (or $\liminf_{k\to (k_-)+} \min \{ |c(k) -c| \, : \,  \BF(k, c)=0, \, c\ne c(k) \} = 0$ if $k_- > -\infty$).
\end{enumerate}
%%\item $\lim_{k\to 0} c(k) \in \R$ exists if $0 \in [k_-, k_+]$.
%%at least one of the following three possibilities occurs 
%%\item $|c(k)| \to +\infty$; or 
%%\item $\min \{ |c(k) -U(x_2)| \, : \, x_2 \in [-h, 0]\} \to 0$; or 
%%\item . 
\end{enumerate}
\end{lemma}

\begin{remark} \label{R:continuation}
This is exactly Lemma 4.3 in \cite{LiuZ21} where the absence of the surface tension does not affect the proof. It is based on the non-negative integer valued index of $\BF$ (or any general complex analytic functions) on various appropriate bounded piecewise smooth domains $\Omega \subset \C \setminus U([-h, 0])$ satisfying that $\BF(k, \cdot)$ is holomorphic in $\Omega$, $C^0$ in $\bar \Omega$, and $\BF(k, \cdot) \ne 0$ on $\p \Omega$, 
\be \label{E:ind}
\text{Ind}\big(\BF(k, \cdot), \Omega\big):=\frac 1{2\pi i} \oint_{\p \Omega} \frac {\p_c \BF(k, c)}{\BF(k, c)} dc \in \mathbb{N} \cup \{0\},
\ee 
which is equal to the total number of zeros of $\BF(k, \cdot)$ inside $\Omega$, counting their multiplicities. In particular, given a bounded domain $\Omega \subset \C$ with piecewise smooth $\p \Omega$ and $k_1 <k_2$ such that $\BF(k, c) \ne 0$ for any $k \in [k_1, k_2]$ and $c \in \p \Omega$, Ind$\big(\BF(k, \cdot), \Omega\big)$ is a constant in $k \in [k_1, k_2]$. If Ind$\big(\BF(k, \cdot), \Omega\big)=1$, then the unique root $c(k)$ of $\BF(k, \cdot)$ in $\Omega$ is simple and analytic in $k$. In addition, if $k_2=-k_1$, then $c(k)$ is even in $k$ due to the evenness of $\BF$ in $k$ and the uniqueness of its root in $\Omega$. The proof of statement also needs Lemma \ref{L:e-v-large}(2) which ensures the roots to be bounded for $k$ in any compact interval. 
\end{remark}

\section{Eigenvalue distribution of the linearized gravity water waves} \label{S:e-values}

In this section, we study the eigenvalues of the linearized gravity waves in details. A complete eigenvalue distribution will be obtained under certain conditions. On the one hand, we shall focus on those aspects which are different from the linearized capillary gravity waves, e.g. the instability in wave number $|k|\gg1$. On the other hand, some similar results as in \cite{LiuZ21} will be obtained under substantially weaker assumptions, which requires some further detailed analysis on $Y(k, c)$ and $F(k, c)$ given in the next two subsections.

\subsection{Further analysis of $Y(k, c)$} \label{SS:Y}

In this subsection, we extend the analysis of $Y(k, c)$ defined in \eqref{E:Y}. In particular, instead of using Lemma 3.24 in \cite{LiuZ21}, we shall obtain the same improved regularity of $Y(k, \cdot)$ near $c= U(-h)$ 
%based on a general formula of $Y$ 
%under a weak non-degeneracy assumption.  
without any assumption additional to the monotonicity of $U(x_2)$. 

\begin{lemma} \label{L:Y-1}
Suppose $U\in C^{l_0}$. Assume $l_0\ge 3$, then for any $M>0$ there exists $C>0$ depending on $U$ and $M$ such that for any $j \ge 0$ and $c \in \C$ satisfying 
\be \label{E:circle-M}
\big|c - \tfrac 12 \big( U(-h) + U(0)\big)\big| \ge \tfrac 12 \big( U(0) - U(-h)\big) + M, 
\ee
it holds 
\be \label{E:Y-bound}
|\p_c^j \big(Y(k, c) - k \coth kh\big)| \le Cj! \mu  M^{-j-1}, \quad \mu = (1+k^2)^{-\frac 12}. 
\ee
Moreover, if $l_0\ge 4$ 
%and $k\in \R$ and $U$ satisfy that there exists $n \ge 0$ such that  
%\be \label{E:ND-neutral-M-1} \begin{split} 
%& \{ c \in U\big((-h, 0)\big) \mid y_{0-} (k, c, 0) =0\} = \{c_j \mid j=1, \ldots n\}\; \text{ does not contain } \\
%& \text{accumulation points of } \; \{ c \in \C \mid y_{0-} (k, c, 0) =0\} \; \text{ and } \; \forall j =1, \ldots, n, \\
%&\; \exists \, 1\le l_j \le \tfrac 12(  l_0-2), \; \text{ s. t. } \; \p_{c_R}^{m} y_{0-} (k, c_j) = 0 \ne \p_{c_R}^{l_j} y_{0-} (k, c_j), \; \forall \, 0\le m < l_j,
%\end{split} \ee
then for any $q\in [1, \infty)$, $j_1, j_2\ge 0$, $j_2 \le 2$, and $j_1+j_2\le l_0-4$, $\p_k^{j_1} \p_{c_R}^{j_2} Y(k, c)$ are $L_k^\infty L_{c_R}^q$ locally in $k$ and $c_R$ in the domain $D(Y)$. 
\end{lemma} 
  
%%This lemma of the improved regularity in $c$ of $Y(k, c)$ near $c=U(-h)$, based on integral representations \eqref{E:Y-Cauchy-1} and \eqref{E:Y-Cauchy-2}, is for a fixed $k \in \R$ under a weak non-degeneracy condition \eqref{E:ND-neutral-M-1}. 
%Assumption \eqref{E:ND-neutral-M-1} is much weaker than $y_- (k, c, 0)\ne 0$ for all $c \in U([-h, 0])$, which was assumed for several lemmas in \cite{LiuZ21}. After the proof of the lemma we shall give a more explicit sufficient condition for  \eqref{E:ND-neutral-M-1}. The proof is based on integral representations \eqref{E:Cauchy-3} and \eqref{E:Cauchy-4} for $Y(k, c)$. 

\begin{proof} 
%The proof of statement (1) is basically same as that of Lemma 3.23(1) in \cite{LiuZ21}. Namely, Lemma \ref{L:y0}(8)(9) implies the $C^{l_0-3}$ smoothness of $y_-(k, c, 0)$ in $k$ and $c \in U\big( (-h, 0)\big)$ and that of $y_-(k, c, x_2^c)$ in $c \in U([-h, 0])$, respectively, the latter of which also yields $y_-(k, c, x_2^c) = O(|c-U(-h)|)$ for $0\le c-U(-h) \ll 1$. Hence $Y_I$ is $C^{l_0-3}$   in $k$ and $c \in U\big( (-h, 0)\big) \cap D(Y(k, \cdot))$. Despite the logarithmic singularity of $\p_c y_-(k, c, 0)$ at $c = U(-h)$ (Lemma 3.16 in \cite{LiuZ21}), we obtain the regularity of $Y_I$ for $c$ near $U(-h)$ from the vanishing of $y_-(k, c, x_2^c)$ and Lemma \ref{L:y0}(11).   
We first work on the regularity of $Y$ near $c=U(-h)$. Fix $k_0 \in \R$. 
%assuming \eqref{E:ND-neutral-M-1}. 
Our proof is based on an integral formula of $Y(k, c)$ for $k$ near $k_0$ and $c\in D(Y(k, \cdot)) \setminus U([-h, 0])$. 

The first step is to identify a domain where there are only finitely many zero points of $y_-(k, \cdot, 0)$.  
Due to the Semi-circle Theorem for the linearized Euler equation at a shear flow in a fixed 2-dim channel $x_2 \in (-h, 0)$, its unstable and stable modes, which correspond to the zero points of $y_-(k , \cdot, 0)$, are contained inside the disk \eqref{E:semi-circle} with a diameter $U([-h, 0])$. The analyticity of $y_-$ in $c$ with $c_I >0$ yields all zero points of $y_-(k_0 , \cdot, 0)$ with $c_I>0$ are isolated. Due to the continuity of $y_-$ in $c$ restricted to $c_I\ge 0$, Lemma \ref{L:y0}(4)(5), and the Semi-circle Theorem, the set $\CK$ of  all accumulation points of the zero points of $y_-(k_0 , \cdot, 0)$ is compact and  
\[
\CK \subset \subset U\big((U'')^{-1}(0)\big) \cap U\big((-h, 0)\big).
\]   
Therefore, for any $\ep_0>0$, there exists a function $\phi \in C^\infty (\R, [0, \ep_0])$ such that 
\[ \begin{split} 
&\exists n \in \N \cup \{0\}, \ep \in (0, \ep_0], \;\; \inf \{ |c - \tilde c_R - i \phi(\tilde c_R)|: \   y_-(k_0, c, 0) =0, \, \tilde c_R \in \R \} > 2\ep, \\
&\phi|_{\R \setminus [U(-h) +\ep, U(0)-\ep]} \equiv 0,  \;\;  \{ c\in \CU \mid y_- (k_0, c, 0) =0\} = \{ c_j = c_{jR} + i c_{jI} \mid j=1, \ldots, n\},  \\
&\text{where } \;  \CU \triangleq \{c_R +i c_I \mid  \pm c_I \ge \phi(c_R)\}. 
%\;\; \CU = \CU_+ \cup\CU_-, \;\; \CU_0 \triangleq \{c_R +i c_I \mid  |c_I| < \phi(c_R)\}.
\end{split}\]
Roughly $\phi (c_R) \in [0, \ep_0]$ is supported inside $U\big( (-h, 0)\big)$ and the region bounded by the graphs of $\pm \phi$ contains all except finitely many roots of $y_-(k_0, \cdot, 0)$.  
%Obviously $\CK \subset \subset \CU_0$ and $U(-h), U(0) \in \CU$. 
Clearly $\phi$ can be constructed so that it is supported in any prescribed neighborhood of $\CK \subset \C$. Let 
\[
\wt \CK = \{ c_R + ic_I \mid c_R \in [U(-h), U(0)], \, |c_I|\le \phi (c_R)\} \supset [U(-h), U(0)], 
\]
which is a compact subset of $\C$ with its smooth upper and lower boundaries given by the graph of $\pm \phi$ restricted to $[U(-h), U(0)]$. For $k$ close to $k_0$, $Y(k, c)$ is holomorphic in $c \in \C\setminus \big(\wt \CK \cup (\cup_{j=1}^n B(c_j, \ep)) \cup (\cup_{j=1}^n B(\overline {c_j}, \ep))\big)$, where $B(c, \ep)$ denotes the open ball in $\C$ centered at $c$ with radius $\ep$. 

Our next step is to derive an integral formula of $Y(k, c)$. For any $r>0$, let 
\be \label{E:CD_r}
\wt \CK_r = \{ c_R + ic_I \mid c_R \in [U(-h) - r, U(0)+r], \, |c_I| \le \phi(c_R) + r\}  \subset \C,
\ee 
which is a neighborhood of $\wt \CK$ roughly with the margin $r$. For 
\[
r_2 \gg 1\gg r_1>0 \; \text{ and } \; c \in \wt \CK_{r_2} \setminus \big(\wt \CK_{r_1} \cup (\cup_{j=1}^n B(c_j, \ep)) \cup (\cup_{j=1}^n B(\overline {c_j}, \ep)) \big), 
\]
the Cauchy Integral Theorem yields 
\begin{align*} 
Y(k, c) = & \frac 1{2\pi i} \Big( \oint_{\p \wt \CK_{r_2}} - \oint_{\p \wt \CK_{r_1}} \Big) \frac {Y(k, c')}{c'-c} dc' - \frac 1{2\pi i} \sum_{j=1}^{n} \Big( \oint_{\p B(c_j, 2\ep)} + \oint_{ \p B(\overline {c_j}, 2\ep)} \Big) \frac {Y(k, c')}{c'-c} dc'. 
\end{align*}
Since the term $\frac {U'' (x_2)}{U(x_2) - c'} \to 0$ in the Rayleigh equation \eqref{E:Ray-H1-1} as $ |c'|\to \infty$, one may prove $y_-(k, c', x_2) \to k^{-1} \sinh k(x_2+h)$ and $Y(k, c') \to k\coth kh$ as $ |c'|\to \infty$ (see Lemma 3.3 and the proof of Lemma 3.21 in \cite{LiuZ21}), which also holds even if $k=0$. Therefore  the outer integral along $\p \wt \CK_{r_2}$ converges to $k\coth kh$ as $r_2 \to +\infty$ and we obtain 
\begin{align*} 
Y(k, c) =& k\coth kh - \frac 1{2\pi i} \oint_{\p \wt \CK_{r_1}} \frac {Y(k, c')}{c'-c} dc' - \frac 1{2\pi i} \sum_{j=1}^{n} \Big( \oint_{\p B(c_j, 2\ep)} + \oint_{ \p B(\overline {c_j}, 2\ep)} \Big) \frac {Y(k, c')}{c'-c} dc'.   
\end{align*}
%To analyze the integral along  the contour $\p \wt \CK_{r_1}$, 
We observe that $\p \wt \CK_{r_1}$ is the union of the two graphs of $\pm (r_1+\phi)$ over $[U(-h), U(0)]$, the left half of the boundary of the square centered at $U(-h)$ with the horizontal and vertical side length $2r_1$, and the right half boundary of such a square centered at $U(0)$. As $r_1 \to 0+$,  due to the continuity of $Y$ at $c\ne U(0)$ and its logarithmic upper bound near $U(0)$ (Lemma \ref{L:Y-def}(2)), the Cauchy integrals along the half boundaries of the squares converge to zero as $r_1 \to 0+$. Therefore the above integral formula yields 
\be \label{E:Y-Cauchy-2.5} 
Y(k, c) = k\coth kh - \frac 1{2\pi i} \sum_{j=1}^{n} \Big( \oint_{\p B(c_j, 2\ep)} + \oint_{ \p B(\overline {c_j}, 2\ep)} \Big) \frac {Y(k, c')}{c'-c} dc'   + I_1(k, c), 
\ee
for 
\[
c\in \C \setminus \big(\wt \CK \cup (\cup_{j=1}^n B(c_j, \ep)) \cup (\cup_{j=1}^n B(\overline {c_j}, \ep)) \big),
\]
where, for  $c \notin \wt \CK$, 
\be \label{E:Y-Cauchy-2.7} \begin{split}
I_1 (k, c) = & \lim_{r\to 0+} \frac{1}{2\pi i} \sum_\pm \Big(\pm \int_{U(-h)}^{U(0)} \frac{Y (k, c' \pm i\phi(c')\pm ir) }{c' \pm i\phi(c') \pm ir-c}d(c' \pm i\phi(c'))\Big)\\
=&  \frac{1}{2\pi i} \sum_\pm \Big(\pm \int_{c'\in supp(\phi)} \frac{Y (k, c' \pm i\phi(c') \pm 0i) }{c' \pm i\phi(c') -c}d(c' \pm i\phi(c'))\Big) \\
&+  \frac 1\pi \int_{\R \setminus supp(\phi)}\frac {Y_I(k, c')}{c'-c} dc'.
\end{split} \ee  
Here we used the regularity of $Y(k, c)$ (Lemma \ref{L:Y-def}) and the property $supp(Y_I) \cap \R \subset [U(-h), U(0)]$. For $c \in \R \setminus supp (\phi)$ and $c_I>0$, the above formula applies to $c+ i c_I \notin \wt \CK$. By taking $c_I \to 0+$, we obtain  
\be \label{E:Y-Cauchy-2.8} \begin{split}
I_1 (k, c)= &  \frac{1}{2\pi i} \sum_\pm \Big(\pm \int_{c'\in supp(\phi)} \frac{Y (k, c' \pm i\phi(c') \pm 0i) }{c' \pm i\phi(c') -c}d(c' \pm i\phi(c'))\Big) \\
&-\CH \big( \chi_{\R\setminus supp(\phi)}  Y_I(k, \cdot)\big) (c) + i Y_I(k, c), \qquad\qquad c \in \R \setminus supp (\phi). 
\end{split} \ee  
For any $c\in \C$ such that $(k_0, c) \in D(Y)$ (namely, $y_-(k_0, c, 0)\ne 0$), by taking sufficiently small $\ep_0$ in the choice of $\phi$, the above formulas \eqref{E:Y-Cauchy-2.5}--\eqref{E:Y-Cauchy-2.8} apply to $Y(k, c)$ for $k$ close to $k_0$. 

While 
%sufficiently smooth in $k$, $Y$ is 
analytic in $(k, c) \in D(Y)\setminus \R \times U([-h, 0])$ and, when restricted to $c_I\ge 0$, $Y$ is $C^{l_0-2}$ in $c \in D(Y(k, \cdot)) \setminus \R \times \{U(-h)\}$ (Lemma \ref{L:Y-def}), so we only need to focus on $c$ near $U(-h)$ restricted to $c_I\ge 0$. In the above formulas, the integrals along $\p B(c_j, 2\ep)$ and its conjugate are smooth in $c$ near $U(-h)$ and thus we only need to consider the regularity of the term $I_1(k, c)$. We recall $U(-h)$ is an interior point of $supp(\phi)^c \subset \R$ according to the definition of $\phi$. From Lemma \ref{L:Y-def}(6), even though $Y_I$ is $W^{3, q}$ locally in $c\in U\big([-h, 0)\big)$ if $l_0\ge 4$, when viewed as a function of $c$ in a whole neighborhood of $U(-h)$, only $\p_k^{j_1} \p_{c_R}^2 Y_I \in L^\infty$ holds due to the jump of $\p_{c_R}^2 Y_I$ at $c=U(-h)$. The desired regularity of $Y$ follows from that of $Y_I$, the above representation formula of $Y$, and the boundedness in $L^q$ of the convolution by $\frac 1{c' + i c_I}$ uniform in the parameter $c_I \ge 0$. 

Finally, it remains to prove inequality \eqref{E:Y-bound}.  According to \eqref{E:y-_1}, there exists $k_1>0$ such that for any $|k| \ge k_1$, it holds  
\[
|y_-(k, c, 0)| \ge \tfrac 12 \mu \sinh \mu^{-1}h, \quad \forall \, c \in \C, 
\]
and thus $Y(k, c)$ is well-defined for all $c\in \C$. In this case, \eqref{E:Y-Cauchy-1} for $c \notin U([-h, 0])$  becomes 
\begin{align*}
Y(k, c) - k\coth kh= & \frac{1}{\pi}\int_{U(-h)}^{U(0)} \frac{Y_I (k, c')}{c'-c}dc' =  \frac{1}{\pi}\int_{U(-h)}^{U(0)} \frac {\pi U''(U^{-1} (c')) y_{-} (k, c', U^{-1}(c'))^2}{U'(U^{-1} (c')) |y_{-} (k, c', 0)|^2 (c'-c)}  dc'. 
\end{align*}
Again,  \eqref{E:y-_1} implies, for $|k| \ge k_1$ and $c' = U(\tilde x_2^c)$ with $\tilde x_2^c \in [-h, c]$, 
\[
|y_-(k, c', \tilde x_2^c)| \le 2 \mu \sinh \mu^{-1} (\tilde x_2^c +h) \implies |Y_I (k, c')| \le C e^{\frac 2\mu \tilde x_2^c}. 
\]
Therefore 
\begin{align*}
& |\p_c^j \big(Y(k, c) - k \coth kh\big)| \le \frac {C (j!)}{dist(c, U([-h, 0]))^{j+1}} \int_{U(-h)}^{U(0)} e^{\frac 2\mu \tilde x_2^c} dc' \\
\le & \frac {C (j!)}{dist(c, U([-h, 0]))^{j+1}} \int_{-h}^0 e^{\frac {2x_2}\mu } dx_2 \le  \frac {C (j!) \mu}{dist(c, U([-h, 0]))^{j+1}},
\end{align*}
where the substitution $c'= U(\tilde x_2^c)$ was used in the above integration. Hence inequality \eqref{E:Y-bound} follows for $|k|\ge k_1$. For $|k| \le k_1$, let $B_M \subset \C$ be the open disk centered at  $\tfrac 12 \big( U(-h) + U(0)\big)$ with radius $\tfrac 12 \big( U(0) - U(-h)\big) + \frac M2$. From the Semi-circle Theorem, $y_- (k, c, 0) \ne 0$ for all $k$ and $c \notin B_M$ and $Y(k, c)$ is analytic. From the same procedure in deriving \eqref{E:Y-Cauchy-2.5}, but replacing the inner contour by $\p B_M$, we obtain 
\[
\p_c^j \big(Y(k, c) - k \coth kh\big) =  - \frac{j!}{2\pi i}\oint_{\p B_M} \frac {Y(k, c')}{(c' -c)^{j+1}} dc'. 
\]
Since $\max \{|Y(k, c')| \mid |k| \le k_1, \, c' \in \p B_M\} <\infty$, due to the Semi-circle Theorem and the regularity of $Y$, \eqref{E:Y-bound} follows immediately.  
\end{proof}

The following corollary is a direct consequence of the above lemma, the definition \eqref{E:dispersion} of $F(k, c)$, and Lemma \ref{L:Y-def}. 

\begin{corollary} \label{C:F}
Assume $U\in C^{l_0}$, $l_0\ge 4$, then, when restricted to $c_I\ge 0$ and near $c=U(-h)$, $F$ is $C^{\infty}$ in $k$ and $\p_k^{j_1} \p_c^{j_2} F$ is $C^{\alpha}$ for any $\alpha \in [0, 1)$, $j_2=0,1$, and $0\le j_1 \le l_0-4-j_2$. 
\end{corollary}

\subsection{Basic properties of $\BF(k,c)$ and $F(k, c)$} \label{SS:F}

Recall that eigenvalues of the linearized gravity waves at the shear flow $U(x_2)$ are given in the form of $-ikc$ where $(k, c)$ satisfy $\BF(k, c) =F(k, c)=0$ with $\BF$ and $F$ defined in \eqref{E:BF} and \eqref{E:dispersion}. In this subsection, we derive some basic estimates of them. In particular, in Lemma \ref{L:pcF} we prove a non-degenerate sign property of $\p_c F$ at $c=U(-h)$, crucial for the bifurcation of eigenvalues, without assuming the convexity/concavity of $U$ as in Lemma 4.9 in \cite{LiuZ21}.   

\begin{lemma} \label{L:e-v-large}
Assume $U \in C^3$, then we have the following for any $\alpha \in (0, \frac 12)$. 
\begin{enumerate} 
\item There exists $C>0$ depending on $\alpha$, $|U'|_{C^2}$, and $|(U')^{-1}|_{C^0}$,  such that 
%\[
%k \sinh kh \le C|\BF(k, c)|, \quad \forall |k|\ge k_0, \; |c|\le M,
%\]
\begin{align*}
&|\BF - (U(0)-c)^2 \cosh \mu^{-1}h| \le  C \big(  \mu +  \mu^\alpha|U(0)-c|^2  \big)\cosh \mu^{-1} h,
%\le & C \big(\mu |c| + \mu^{\alpha-1} + (1+ |c|^2) (\mu^\alpha+ \mu |\log \min\{1, |U(0)-c|\}|\big)\cosh \mu^{-1} h, 
\end{align*}
where we recall $\mu = (1+k^2)^{-\frac 12}$.
\item For any $k_*, M>0$, there exists $C>0$ depending only on $k_*$, $M$, and $|U''|_{C^0}$,  such that, for any $|k|\le k_*$ and  $c$ satisfying $dist(c, U(-h, 0]) \ge M$, 
\[
|\BF -  (U(0)-c)^2 \cosh kh| \le C \big(1+ |c| + |U(0)-c|^2  dist\big(c, U([-h, 0]) \big)^{-1} \big).
\]
\end{enumerate}
\end{lemma}

Statement (1) is mainly applied for $|k|\gg 1$ while statement (2) mainly for $|c|\gg1$. 

\begin{proof}
From Lemma \ref{L:y_pm-1}, one observes that the logarithmic singularity  in $y_-'(k, c,0)$ is significant only when $|c| \le C$ and it becomes bounded when multiplied by $U(0)-c$. Therefore 
\begin{align*}
|\BF - (U(0)-c)^2 \cosh \mu^{-1}h| \le & C(1+ |U(0) - c|)|y_-(k, c, 0)| \\
& \qquad  + |U(0)-c|^2 |y_-'(k , c, 0) - \cosh \mu^{-1}h| \\
\le & C(\mu + \mu |U(0) - c|+ \mu^\alpha |U(0)-c|^2 ) \cosh \mu^{-1}h,  
\end{align*}
which implies statement (1) in the lemma. 

When $|k|\le k_*$,  $dist\big(c, U([-h, 0])\big) \ge M$, and $x_2 \in [-h, 0]$, it is easy to estimate the Rayleigh equation \eqref{E:Ray-H1-1} and derive that $y_- (k, c, x_2)$ and $y_-'(k, c, x_2)$ are uniformly bounded by some $C>0$ depending on $k_*$, $M$, and $|U''|_{C^0}$.  Hence by regarding $\frac {U''}{U-c} y_-$ in \eqref{E:Ray-H1-1} as a perturbation term, from the variation of parameter formula, it is straight forward to estimate 
\[
|y_-(k, c, 0) - k^{-1} \sinh kh| + |y_-'(k, c, 0) - \cosh kh| \le C dist\big(c, U([-h, 0])\big)^{-1}. 
\]
Statement (2) follows immediately. 
\end{proof}

The next lemma states that $U(-h)$ becomes a singular neutral mode for a unique $k_-$. 

\begin{lemma} \label{L:U(-h)}
Assume $U\in C^3$, then $y_-(k, U(-h), 0) > 0$ for any $k \in \R$ and there exists $k_- >0$, unique among $k \in [0, \infty)$, such that $F(k_-, U(-h))=0$. 
\end{lemma} 

\begin{proof} 
The statement $y_-(k, U(-h), 0) > 0$ had been given Lemma \ref{L:y0}(4) and thus $F(k, U(-h)) \in \R$ is well-defined for all $k \in \R$ and even in $k$. The existence and uniqueness of $k_-$ is a direct corollary of Lemma \ref{L:e-v-basic-2}, which says $F(0, U(-h))=-g$ and $F(\cdot, U(-h))$ is increasing in $k>0$, and Lemma \ref{L:e-v-large}(1), which implies $F(k, U(-h))>0$ for $k \gg 1$. 
\end{proof}

\subsection{Interior singular neutral modes} \label{SS:int-N-M}

Lemma \ref{L:continuation} on the analytic continuation and the continuity of $\BF$ (Lemma \ref{L:e-v-basic-1}) when restricted to $c_I \ge 0$ imply that any branch of roots of $\BF$ can be continued until it either collides with another branch or approaches singular neutral modes -- roots of $\BF$ with $c \in U([-h, 0])$ which are at the boundary of analyticity of $\BF$. While $c =U(0)$ is impossible (Lemma \ref{L:e-v-basic-2}(1)), besides $(k_-, U(-h))$ obtained in lemma \ref{L:U(-h)}, in the following we give some basic properties of singular neutral modes with $c\in U((-h, 0))$. 

\begin{lemma} \label{L:neutral-M}
Assume $U \in C^3$. The following are equivalent for $c_0 = U(x_{20}) \in U((-h, 0])$.
\begin{enumerate}
\item There exists $k_0 \in \R$ such that $\BF(k_0, c_0)=0$.
\item  There exists a $k_0 \ge 0$ such that $F(k_0, c_0)=0$.
%\item $F(0, c_0) \le 0$ and either $x_{20}=-h$ or $U''(x_{20}) =0$;  
\item $U''(x_{20}) =0$ and the Sturm-Liouville operator $\CR$ has a non-positive eigenvalue, where $\CR :=  - \p_{x_2}^2 + \tfrac {U''}{U-c_0}$ with boundary conditions 
\[
y(-h)=0, \quad (U(0)-c_0)^2 y'(0) - (U'(0)(U(0)-c_0) +g) y(0)=0.
\]
\end{enumerate}
Moreover, it is also  satisfied that $\p_K F(\sqrt{K}, c_0)|_{K=k^2} >0$ for such $c_0$ and any $k \in \R$ with $y_- (k, c_0, 0)\ne 0$. 
\end{lemma}

We observe that the assumption $U''(x_{20})=0$ implies that $\CR$ along with its boundary conditions is a self-adjoint operator with $C^0$ coefficient. 

\begin{remark} 
The last statement also applies to $c_0 = U(-h)$. Namely $-(k_-)^2$ is the eigenvalue of $\CR$ at $c= U(-h)$. 
\end{remark}

\begin{proof}
Before we prove the equivalence of the above statements, we first extend the monotonicity, as well as the concavity under certain conditions, of $Y(k, c)$ and $F(k, c)$ with respect to $K=k^2$ (originally for $c \in \R \setminus U((-h, 0])$ in Lemma \ref{L:Y-def}(8) and \ref{L:e-v-basic-2}(2), respectively) also to any $c_0 = U(x_{20}) \in U((-h, 0))$ satisfying $U''(x_{20})=0$. Here we observe that $U''(x_{20})=0$ implies $y_-(k, c_0, x_2) \in \R$ and is $C^2$ even for $x_2$ near $x_{20}$ for such $c_0$ and any $k \in \R$ 
%(Lemma \ref{L:y0}(3)), 
and thus $Y(k, c_0), F(k, c_0) \in \R$ as well. The proof is similar to that of Lemma 4.5 in \cite{LiuZ21}, starting with the following claim. 

{\it Claim.} Assume $c_0\in \C \setminus U\big((-h, 0]\big)$ or $c_0= U(x_{20}) \in  U\big((-h, 0)\big)$ with $U''(x_{20})=0$ and $\tilde y, y\in C^0([-h, 0])\cap C^2((-h, 0))$ are solutions to 
\[
(\CR+k^2) \tilde y=0, \; \tilde y(-h)=0, \; \tilde y(0)=1; \quad (\CR+k^2) y = f\in C^0([-h, 0]), \; y(-h)=y(0)=0; 
\]
%with the above inflection value $c_0 = U(x_{20}) $ of $U$, where $f \in C^0([-h, 0])$,  
then we have 
\be \label{E:F-signs-0}
y'(0) =- \int_{-h}^0 \tilde yf dx_2.
\ee
Moreover, if $\tilde y\ne 0$ on $(-h, 0)$ which in particular is true if $c_0\in \R \setminus U\big((-h, 0)\big)$ due to lemma \ref{L:y0}(4), then it also holds 
\be \label{E:F-signs-1}
%y (x_{20}) =   \frac 1{\tilde y' (x_{20})} \int_{-h}^{x_2'} \tilde y(x_2'') f(x_2'') dx_2'', \quad 
y (x_2) =  \tilde y(x_2) \int_{x_2}^0 \frac 1{\tilde y (x_2')^{2}} \int_{-h}^{x_2'} \tilde y(x_2'') f(x_2'') dx_2'' dx_2'.
\ee

Here the assumption on $c_0$ insures that $\CR$ has $C^0$ coefficients over $[-h, 0]$. The claim follows through direct computations using $(\tilde y'  y - \tilde y  y')' = \tilde y f$. 
%If $\tilde y$ does vanish somewhere on $(-h, 0)$, \eqref{E:F-signs-1} can still be generalized, but it won't necessary for this paper. 

If $y_-(k, c_0, 0) \ne 0$, then apparently $\tilde y = \tfrac {y_-}{y_-(0)} \in \R$ and $Y(k, c_0)= \tilde y'(0)$. It is straight forward to compute by differentiating with respect to $K=k^2$,   
\[
\p_K Y(k, c_0) = \p_K \tilde y' (0), \;\; \p_K^2 Y(k, c_0) = \p_K^2 \tilde y' (0), \;\; (\CR+k^2) \p_K \tilde y = - \tilde y, \;\; (\CR+k^2) \p_K^2 \tilde y = - 2 \p_K \tilde y,
\] 
and $\p_K \tilde y$ and $\p_K^2 \tilde y$ satisfy the zero Dirichlet boundary conditions assumed in the claim. Applying \eqref{E:F-signs-0} to $\p_K \tilde y$ and $\p_{K}^2 \tilde y$, along with Lemma \ref{L:y0}(4)(5), implies, for $c_0\in \R \setminus U\big((-h, 0]\big)$ or $c_0= U(x_{20}) \in  U\big([-h, 0)\big)$ with $U''(x_{20})=0$ as long as $y_-(k, c_0, 0) \ne 0$,  
\be \label{E:F-signs-0.1} 
%\begin{split}
\p_K Y (k, c_0) = \p_K \tilde y' (0) = \int_{-h}^0 \tilde y^2 dx_2 >0, \quad \p_K F(k, c_0) > 0.
\ee
Moreover, if $y_- (k, c_0, x_2) \ne 0$ on $(-h, 0)$, then 
\be \label{E:F-signs-0.2} 
\p_{K}^2 Y(k, c_0) = -2 \int_{-h}^0 \tilde y(x_2)^2 \int_{x_2}^0 \tilde y (x_2')^{-2} \int_{-h}^{x_2'} \tilde y(x_2'')^2 dx_2'' dx_2' dx_2 <0, \quad \p_K^2 F(k, c_0) <0. 
%\end{split} 
\ee

%With the above concavity of $F$ in $K=k^2$ established, 
We are ready to prove the lemma. According to Lemmas \ref{L:e-v-basic-1}(2)(3) and \ref{L:e-v-basic-2}(1), $F(k_0, c_0)=0$ iff $\BF(\pm k_0, c_0)=0$ where $y_-(k_0, c_0, 0)\ne 0$ and $x_{20} \ne 0$ are necessary. Consequently \eqref{E:F-signs-0.1} immediately yields the equivalence of statements (1) and (2). Moreover Lemmas \ref{L:e-v-basic-1}(4) and \ref{L:y0}(5) imply $U''(x_{20})=0$ or $c_0= U(-h)$, where $y_- (k, c_0, x_2) \in \R$ in both cases. We obtain the equivalence with statement (3) simply by observing the associated Sturm-Liouville problem structure for the neutral modes. 
\end{proof}

Next we prove that any interior inflection value of $U$ is a singular neutral mode for some wave number, which is through a different proof as in \cite{Yih72, HL13}. 

\begin{lemma} \label{L:neutral-M-1} 
Suppose $U\in C^3$ and $c_0 = U(x_{20}) \in U((-h, 0))$ satisfy $U''(x_{20}) =0$, then there exists $k_0 >0$ such that $F(k_0, c_0)=0$
%, $\p_K^2 F(k_0, c_0)<0$, 
with its eigenfunction $y_-(k_0, c_0, x_2) >0$ for all $x_2 \in (-h, 0]$. 
\end{lemma} 

\begin{proof}
According to Lemma \ref{L:neutral-M}, $F(k, c_0)=0$ iff $y_-(k, c_0, x_2)$ is an eigenfunction with the eigenvalue $-k^2$ of the Sturm-Liouville operator $\CR$ along with boundary conditions in \eqref{E:Ray}. It corresponds to the variational functional 
\[
I(c_0, y) = \frac 12 \int_{-h}^0 |y'(x_2)|^2 + \frac {U''(x_2)}{U(x_2) -c_0} |y(x_2)|^2 dx_2 - \frac {U'(0) (U(0) - c_0) +g}{2 (U(0) - c_0)^2} |y(0)|^2,
% \le I_C (y),
\]
for $y \in H^1 \big((-h, 0)\big)$ with $y(-h)=0$. 
%Therefore $\CR$ has a non-positive eigenvalue is equivalent to that the above quadratic form is not positive definite on $H^1$. This allows one to use test functions to see whether an inflection value $c_0$ of $U$ is a neutral mode for some $k_0\in \R$. 
%\end{remark}
Consider 
%the test function 
\[
y(x_2) = 0, \; \text{ if } \; x_2 \in [-h, x_{20}]; \quad y(x_2) = U(x_2) - c_0, \; \text{ if } \; x_2 \in [x_{20}, 0], 
\]
which is clearly a qualified test function. It is straight forward to verify 
\begin{align*}
I(c_0, y)=& \frac 12 \int_{x_{20}}^0 U'(x_2)^2 + U''(x_2)(U(x_2) -c_0) dx_2 - \frac 12 \big(U'(0) (U(0) - c_0) +g\big)\\
=& \frac 12 U'(x_2)(U(x_2) -c_0)\big|_{x_{20}}^0 - \frac 12 \big(U'(0) (U(0) - c_0) +g\big) = - \frac g2 <0. 
\end{align*}
It yields the negative sign of the first eigenvalue $-k_0^2<0$ of $\CR$ with its eigenfunction $y_-(k_0, c_0, x_2) >0$. 
%Therefore \eqref{E:F-signs-0.2} applies and implies $\p_K^2 F(k_0, c_0)<0$. 
\end{proof}

\begin{remark} \label{R:inflectionV}
The proof of Lemma \ref{L:neutral-M-1}, which is different from that in \cite{Yih72, HL08, HL13}, actually implies the existence of singular neutral modes at inflection values without assuming the monotonicity of $U$. More precisely, it holds that\\  
"Suppose $U\in C^3$, $U^{-1}(c_0) = \{x_{2,1} < \ldots < x_{2,n} \} \subset [-h, 0)$, $n \ge 1$, and $U''(x_2)=0$ for all $x_2 \in U^{-1} (c_0)$, then $F(\cdot, c_0)$ has at least $n$ distinct roots in $[0, +\infty)$ and at least $\max\{1, n-1\}$ roots in $\R^+$." \\
To see this, one may consider the test functions $y_j (x_2) = (U-c_0) \chi_{(x_{2,j}, x_{2, j+1})}$ where $x_{2, n+1}=0$ is understood. Clearly they are $L^2$-orthogonal and $I(y_1) = \ldots = I(y_{n-1})=0$ and $I(y_n) = -g/2$. Hence the non-positive subspace is at least $n$-dim. The statement follows immediately since each eigenvalue of this Sturm-Liouville problem has only one linear independent eigenfunction. The above assumptions on $U$ is more general than those on the so-called class $\mathcal{K}^+$ in \cite{HL08, HL13}. 
\end{remark}

Given an interior inflection value $c_0$, the property $\p_K F(k, c_0)>0$, whenever $y_-(k, c_0, 0)\ne 0$, does not imply the uniqueness of the root of $F(\cdot, c_0)=0$ due to the possibility of $y_-(k, c_0, 0)=0$ for some $k > 0$.  In the following we address the number of wave numbers which make $c_0$ a singular neutral mode. We first prove a lemma on the relationship among several critical wave numbers. 

\begin{lemma} \label{L:neutral-M-2} 
Suppose $U\in C^3$ and $c_0 = U(x_{20}) \in U((-h, 0))$ satisfy $U''(x_{20}) =0$ and $k_0>0$ is the maximal root of $F(\cdot, c_0)$, then the following hold.
\begin{enumerate} 
\item If $k_C \ge 0$ satisfies $y_-(k_C, c_0, 0)=0$, then $k_0> k_C$. 
\item If $\frac {U''}{U-c_0} \le 0$ on $(-h, 0)$, then $k_0> k_-$ where $k_-$ is the unique root of $F(\cdot, U(-h))$ given in Lemma \ref{L:U(-h)}. 
\end{enumerate}
\end{lemma} 

Clearly the above $(k_C, c_0)$ coincides with a singular neutral mode for the linearized channel flow at the monotonic $U$. 

\begin{proof}
It is clear that $y_-(k, c_0, 0)=0$ iff $y_-(k, c_0, x_2)$ is an eigenfunction with the eigenvalue $-k^2$ of the Sturm-Liouville problem
\be \label{E:S-L-channel}
\CR y =\lambda y, \quad y(-h)=y(0)=0,
\ee
which corresponds to the variational functional 
\[
I_C(c_0, y) = \frac 12 \int_{-h}^0 |y'(x_2)|^2 + \frac {U''(x_2)}{U(x_2) -c_0} |y(x_2)|^2 dx_2, \quad y \in H_0^1 \big((-h, 0)\big). 
\]
Since obviously $I \le I_C$ and $I_C$ has one more restriction $y(0)=0$ on its domain, so the first eigenvalue $-k_0^2$ of $I(c_0, y)$ must satisfy $-k_0^2 < - k_C^2$. 

One may verify
% for any $y \in H^1$ satisfying $y(-h)=0$, 
\begin{align*}
I(c_0, y) - I(U(-h), y) =& \frac 12 \int_{-h}^0 \frac {(c_0-U(-h)) U'' y^2}{(U -c_0)(U - U(-h))} dx_2 -\Big( \frac {U'(0) (c_0-U(-h))}{(U(0) - c_0) (U(0)-U(-h))} \\
&+\frac {g (c_0-U(-h)) (2U(0) -c_0-U(-h))}{ (U(0) - c_0)^2 (U(0) - U(-h))^2} \Big) \frac {|y(0)|^2}2. 
\end{align*}
Hence the assumption $\frac {U''}{U-c_0} \le 0$ on $(-h, 0)$ implies $I(c_0, y) \le I(U(-h), y)$. Since $y_-(k_-, U(-h), 0) \ne 0$,  we obtain that the first eigenvalue $-k_0^2$ of $I(c_0, y)$ satisfy $-k_0^2 < - (k_-)^2$.
\end{proof}

Finally we show how whether an interior inflection value $c_0$ has either one or two critical wave numbers depends on $k_C$ and $g$. 

\begin{lemma} \label{L:neutral-M-3} 
Suppose $U\in C^3$ and $c_0 = U(x_{20}) \in U((-h, 0))$ satisfy $U''(x_{20}) =0$, then the following hold. 
\begin{enumerate}
\item There exists at most one $k_C \ge 0$ such that $y_-(k_C, c_0, 0)=0$. 
\item Such $k_C> 0$ exists iff $F^0 (c_0) >0$, where 
\be \label{E:F0}
F^0 (c_0) \triangleq F(0, c_0) + g = (U(0)-c_0)^2 Y(0, c_0) - U'(0) (U(0)-c_0).
\ee
\item Let $k_*= k_C$ if $k_C$ exists or $k_*=0$ otherwise, there exists a unique $k_0> k_*$ such that $F(k_0, c_0)=0$. It also holds that the corresponding eigenfunction $y_-(k_0, c_0, x_2) >0$ on $(-h, 0]$ and $\p_{K}^2 F(k_0, c_0) <0$. 
\item If $k_C>0$, then 
\begin{enumerate}
\item if $g \ge F^0(c_0)$, then there exists  $k_1\in [0, k_C)$ such that $\{ k\ge 0 \mid F(k, c_0) =0\} =\{k_1, k_0\}$, and 
\item if $g < F^0(c_0)$, then $k_0$ is the only root of $F(\cdot, c_0)$ on $[0, \infty)$. 
\end{enumerate}\end{enumerate}
\end{lemma}

%We notice the above $k_C$ corresponds to a wave number such that $(k_C, c_0)$ is a singular neutral mode of the linearized channel flow 
%%with the slip boundary condition at 
%between the fixed boundaries at $x_2=-h, 0$. 

\begin{proof}
Recall that $y_-(k, c_0, 0)=0$ iff $y_-(k, c_0, x_2)$ is an eigenfunction with the eigenvalue $-k^2$ of \eqref{E:S-L-channel} corresponding to the variational functional $I_C(y)$ on $H_0^1 \big((-h, 0)\big)$. 
Suppose $-\lambda \le 0$ is an eigenvalue with an eigenfunction $y(x_2)$, then Lemma \ref{L:y0}(5) applied to initial conditions at both $x_2=-h$ and $x_2=0$ implies $y \ne 0$ on $(-h, x_{20}]$ and $[x_{20}, 0)$, and thus $y \ne 0$ on both $(-h, 0)$. Hence it is the first eigenvalue and therefore the only non-negative one of \eqref{E:S-L-channel}, which also implies $k_C^2 =\lambda\ge 0$ is unique if it exists. This proves statement (1). 

%If $k_C=0$, then $F$ is not well-defined as $(k_0, c_0)$. Hence we may consider either $k_C$ does not exist or $k_C>0$ in the proof of statement (2). 
On the one hand, suppose such $k_C>0$ does not exist, then $F(k, c_0)$ is well-defined for all $k > 0$. The fact $\p_k F(k, c_0)>0$ for all $k>0$ (Lemma \ref{L:neutral-M}) and $F(k_0, c_0)=0$ with $k_0>0$ (Lemma \ref{L:neutral-M-1}) imply $F(0, c_0)<0$ if it is well-defined. Since this conclusion is independent of $g> 0$, we obtain $F^0 (c_0) \le 0$ if $F(0, c_0)$ is well-defined. On the other hand, assume $k_C>0$ exists. From the Sturm-Liouville theory (or one may easily prove it directly), $y_-(0, c_0, x_2)$ has a root $\tilde x_2 \in (-h, 0)$ with $y_-'(0, c_0, \tilde x_2)<0$. Moreover Lemma \ref{L:y0}(5) yields $\tilde x_2 > x_{20}$. As in \eqref{E:y_-0}, it is straight forward to verify, for $x_2 >\tilde x_2$,   
\[
y_- (0, c_0, x_2) = y_-'(0, c_0, \tilde x_2) (U(x_2) - c_0) \int_{\tilde x_2}^{x_2} \frac {U(\tilde x_2) -c_0}{(U-c_0)^2} dx_2.  
\]
Consequently, one may compute 
\[
F^0 (c_0) = \Big(\int_{\tilde x_2}^0 \frac 1{(U-c_0)^2} dx_2\Big)^{-1} >0.
\]
This completes the proof of statement (2).  

Statements (3) and (4) follow from Lemmas \ref{L:neutral-M-1}, \ref{L:neutral-M-2}, and $\p_k F(k, c_0)>0$ for all $k \ne k_C$ if $k_C\ge0$ exists (Lemma \ref{L:neutral-M}).  
\end{proof}

\subsection{Non-degeneracy of $F$ at $(k_-, c=U(-h))$.} \label{SS:non-deg}

As indicated in Lemmas \ref{L:e-v-basic-1}(4) and \ref{L:e-v-basic-2}(1), $c=U(-h)$ is the only singular neutral mode which is not an inflection value of $U$. It would be one of the key bifurcation points where the instability occurs. 
Our goal in the following is to verify a non-degeneracy of $F$ at $c=U(-h)$ to be used in the bifurcation analysis. 
We need the following two lemmas of technical preparation. The first is an estimate related to the fundamental solution $y_- (k, c, x_2)$. 

\begin{lemma} \label{L:temp-1}
Suppose $U \in C^{l_0}$, $l_0\ge 4$, and let $\tilde y(k, c, x_2) = \frac {y_-(k, c, x_2)}{y_-(k, c, 0)}$ if $y_-(k, c, 0)\ne 0$, then, when restricted to $c_I\ge 0$ and $x_2 \in (-h, 0]$, for $j_1, j_2 \ge 0$, $j_1+j_2\le l_0-4$, $j_2\le 2$, and $q\in [1, \infty)$, it holds that $\p_k^{j_1} \p_c^{j_2}\tilde y, \, \p_k^{j_1} \p_c^{j_2}\tilde y' \in L_k^\infty L_{c_R}^q$ locally in $k$ and $c_R$. 
%are $C^3$ in $k$ and $C^{1, \alpha}$ in $c$ near $U(-h)$, for any $\alpha \in [0, 1)$. 
Moreover,   at $c= U(-h)$,  it holds,
\[
\lim_{x_2 \to (-h)+} \frac 1{(x_2+h)^{2}} \Big( \tilde y' (x_2) -\frac {U'(x_2) \tilde y(x_2)}{U(x_2) - U(-h)} \Big) = \frac 13 k^2 \tilde y' ( -h),   
\]
and there exists $\tilde C>0$ depending on $U$ and $k$ such that, for $x_2 \in (-h, 0]$,  
\[
%\be \label{E:temp-5} 
|\tilde y'(x_2)| \le \tilde C, 
%\quad \tilde y'(x_2) \ge \tilde C^{-1} - \tilde C(\rho+ x_2+h), 
\quad \tilde C (x_2+h) \ge \tilde y (x_2) \ge  \tilde C^{-1} (x_2+h), 
%\quad x_2 \in [-h, 0].
\] 
\[
|\p_c \tilde y( x_2)| \le \tilde C (x_2+h) (1+ | \log (x_2+h)|), \;\; |\p_c \tilde y' (x_2)| \le \tilde C (1+ | \log (x_2+h)|). 
\]
\end{lemma} 

%Recall for $c \in U([-h, 0])$, $y_-(k, c, x_2) = \lim_{\ep \to 0+} y_- (k, c+i\ep, x_2)$. 
According to Lemmas \ref{L:y_pm-1} and \ref{L:y0}(4)(5), it holds $y_-(k, c, x_2) \ne 0$ for any $x_2\in (-h, 0]$ and $(k, c) \in \R \times \C$  in an open set containing  $\R \times \{U(-h)\}$, and thus $\tilde y$ is well defined for such $(k, c)$. 
For $c= U(-h)$, Lemma \ref{L:y0}(3) and $y_-(-h)=0$ imply $\tilde y$ is $C^4$ in $x_2$ and $\tilde y' - U' \tilde y /(U- U(-h))$ vanishes at $x_2=-h$. The above lemma shows that the latter is actually of quadratic order and gives the leading order coefficient. 

\begin{proof}
From the Rayleigh equation \eqref{E:Ray-H1-1}, $\tilde y$ satisfies 
\be \label{E:temp-4} 
-\tilde y'' + \big(k^2 + \frac {U''}{U-c}\big) \tilde y =0, \quad \tilde y(-h)=0, \;  \tilde y (0)=1, \; Y(k, c) = \tilde y' (0).
\ee
For $x_2 \in (-h, 0]$, $\tilde y(k, c, x_2)$ obeys the above Rayleigh equation with smooth coefficients and initial values $(1, Y(k, c))$ given at $x_2=0$. Lemma \ref{L:Y-1} implies the regularity of $\tilde y$ and $\tilde y'$.  
%Even though the equation is singular at $x_2=-h$,  $\tilde y(-h)=0$ yield $\tilde y \in C^4 ([-h, 0])$. 

The desired estimates of $\tilde y$ and $|\tilde y'|$ at $c=U(-h)$ follow immediately from Lemma \ref{L:y0}(3)(4) along with $\tilde y'(-h)\ne 0$. With the differentiability in $c$ near $U(-h)$ had been obtained, differentiating \eqref{E:temp-4} in $c$ yields 
%that $\p_c y$ satisfies    
\be \label{E:temp-8}
- \p_c \tilde y'' + \big(k^2 + \tfrac {U''}{U - c}\big)\p_c \tilde y = - \p_c \big( \tfrac {U''}{U - c}\big) \tilde y, \quad \p_c \tilde y (-h)=\p_c \tilde y(0)=0, \quad \p_c Y(k, c) = \p_c \tilde y'(0). 
\ee
For $c\le  U(-h)$, the claim in the proof of Lemma \ref{L:neutral-M} applies and Lemma \ref{L:y0}(4) and \eqref{E:F-signs-1} imply 
\be \label{E:temp-8.5} \begin{split}
\p_c \tilde y (x_2) = & - \int_{x_2}^0 \frac {\tilde y (x_2)}{\tilde y (x_2')^{2}} \int_{-h}^{x_2'}  \p_c \big( \frac {U''}{U - c}\big) \tilde y^2  dx_2'' dx_2' 
= - \int_{x_2}^0 \frac {\tilde y (x_2)}{\tilde y (x_2')^{2}} \int_{-h}^{x_2'}   \frac {U'' \tilde y^2}{(U - c)^2}    dx_2'' dx_2'
\end{split} \ee
Using the estimates on $\tilde y$ and $\tilde y'$ we obtain, at $c=U(-h)$,    
\begin{align*}
\big|\p_c \tilde y (x_2)\big| \le & \int_{x_2}^0 \frac {\tilde y (x_2)}{\tilde y (x_2')^{2}} \int_{-h}^{x_2'}   \frac {|U''| \tilde y^2 }{(U - U(-h))^2}   dx_2'' dx_2'  \\
\le & \tilde C \int_{x_2}^0 \frac {x_2+h }{x_2'+h}dx_2' \le \tilde C (x_2+h) \big(1+ | \log (x_2+h)|\big),
\end{align*}
%Moroever,  at $c=U(-h)$,  
\begin{align*}
\big|\p_c \tilde y' (x_2)  \big| = & \Big|- \int_{x_2}^0 \frac {\tilde y' (x_2)}{\tilde y (x_2')^{2}} \int_{-h}^{x_2'}   \frac {U'' \tilde y^2}{(U - U(-h))^2}    dx_2'' dx_2' + \frac 1{\tilde y (x_2)} \int_{-h}^{x_2}   \frac {U'' \tilde y^2}{(U - U(-h))^2}  dx_2' \Big| \\
\le &  \tilde C  \big(1+ | \log (x_2+h)|\big).
\end{align*}

To complete the proof of the lemma, for $c= U(-h)$, we consider the property of $\tilde y' - U' \tilde y /(U- U(-h))$ near $x_2=-h$. From \eqref{E:temp-4}, it holds 
\[
\big( (U- U(-h)) \tilde y' - U' \tilde y\big)' = k^2 (U-U(-h)) \tilde y. 
\]
Hence 
\[
\tilde y' (x_2) -\frac {U'(x_2) \tilde y(x_2)}{U(x_2) - U(-h)} = \frac {k^2}{U(x_2) - U(-h)} \int_{-h}^{x_2} (U(x_2')-U(-h)) \tilde y (x_2') dx_2',
\]
which is clearly of the quadratic order in $x_2 +h$. We obtain the desired asymptotics using the leading order expansions of $U$ and $\tilde y$ near $x_2 =-h$.  
\end{proof} 

The next lemma is a property of families of real analytic functions whose roots obey certain properties related to the Semi-circle Theorem \eqref{E:semi-circle}. 

\begin{lemma} \label{L:temp-2} 
Let $a, \rho_0, q_0 >0$ and $f(\rho, q, z=z_1+iz_2)$ be a real analytic function of $z \in \Omega= \{ z= z_R+iz_I \in \C \mid |z + a/2| < a/2\}$ (i.e. additionally it satisfies $f(\bar z) = \overline{f(z)}$) and is also a $C^1$ function of $\rho, q, z_R, z_I$ for $|\rho|\le \rho_0$, $|q| \le q_0$, and $z\in \overline{\Omega}$. Moreover assume 
\be \label{E:temp-4.2}
f(\rho, q, z) \ne 0 \; \text{ if } \; z \notin [-a, 0]; \quad f(0, 0, z) =0 \Leftrightarrow z = 0; \quad \p_q f (0, 0, 0) \ne 0; 
\ee 
then there exists $\tilde \rho \in (0, \rho_0]$ and a $C^1$ function $\psi(\rho)$ defined for $|\rho|\le \tilde \rho$ such that, for any such that $\rho$, 
%$\rho \in [- \tilde \rho, \tilde \rho]$, 
\[
f(\rho, \psi(\rho), z) =0 \Leftrightarrow z = 0,
\]
which clearly implies 
\[
f(\rho, \psi(\rho), -a) \p_z f(\rho, \psi (\rho), 0) \le 0.  
\]
\end{lemma}

Note that  $f$ being a real analytic function of $z \in \Omega$ and $C^1$ up to $\p \Omega$ yields $f \in \R$ for $z \in [-a, 0]$, hence $\p_q f(\rho, q, z), \p_z f(\rho, q, z) \in \R$ for any $z\in [-a, 0]$. The first assumption in \eqref{E:temp-4.2} is crucial. Otherwise a counter example is $f = q - z^2 + 2 \rho z$ where clearly it has to be $\psi(\rho)=0$, but $f(\rho, 0, \cdot)$ has another root $z=2\rho \in (-a, 0)$ for $\rho<0$. 

\begin{proof} 
It is easy to see how $\psi$ should be defined. In fact, 
from $\p_q (\RP\,  f)(0, 0, 0) =\p_q f(0, 0, 0) \ne 0$ and the Implicit Function Theorem applied to $\RP\, f$, there exists $\delta, q_1, \rho_1>0$ and a $C^1$ function $\phi(\rho, z)$ defined for $|\rho|\le \rho_1$ and $z\in [-\delta, 0]$ such that $\phi(0, 0)=0$ and  
\be \label{E:temp-4.2.5}
f(\rho, q, z)= \RP\, f(\rho, q, z)=0, \; |\rho|\le \rho_1, \; |q| \le q_1, \;   z\in [-\delta, 0] \Longleftrightarrow q = \phi(\rho, z). 
%\in [-q_1, q_1]. 
\ee
Let $\psi(\rho) = \phi(\rho, 0)$. 
As we assumed $f(0, 0, z)=0$ iff $z=0$ and $f=0$ only if $z\in [-\delta, 0]$, by the compactness of $\overline{\Omega} \setminus \{z\in \C \mid |z +\delta/3| < 2\delta/3\}$, there exist $\tilde \rho\in (0, \rho_1]$ and $\ep>0$ such that 
\be \label{E:temp-4.3}
f(\rho, q, z)=0, \; |\rho| \le \tilde \rho, \; |q| \le \max_{|\rho'|\le \tilde \rho} |\psi(\rho')| +\ep \implies z\in [-\delta, 0] \implies q = \phi(\rho, z).  
\ee
It remains to prove $f(\rho, \psi(\rho), z) \ne 0$ if $z \in [-\delta, 0)$ and $|\rho| \le \tilde \rho$, which we proceed by an argument by contradiction. 

Assume 
\be \label{E:temp-4.4}
\exists \rho_* \in [-\tilde \rho, \tilde \rho], \; z_* \in [-\delta, 0) \; \text{ such that } \; f(\rho_*, q_*, z_*)=0, \; \text{ where } \; q_* = \psi(\rho_*).
\ee
The definition of $\phi$ yields $ \phi(\rho_*, z_*) =q_* = \phi(\rho_*, 0)$. When $\tilde \rho$ is sufficiently small,  the function $f(\rho_*, q_*, \cdot)$ of $z$ is not identically zero, which, along with its analyticity, implies that it is not identically zero on $[z_*, 0]$. Therefore $\phi(\rho_*, \cdot)$ is not identically equal to $q_*$ for $z \in [z_*, 0]$. Without loss of generality, suppose 
\[
z_\# \in (z_*, 0), \quad q_\# =\phi(\rho_*, z_\#) =  \max_{z\in [z_*, 0]} \phi(\rho_*, z) > q_*.
\] 
Let 
\[
\Omega_1 = \{ z\in \C \mid |z-z_*/2| < |z_*|/2\} \subset \Omega.
\] 
From \eqref{E:temp-4.2.5}, \eqref{E:temp-4.3}, and $ \phi(\rho_*, 0) =q_* = \phi(\rho_*, z_*)$, the index Ind$\big(f(\rho_*, q, \cdot), \Omega_1\big)$ is well defined and takes a constant non-negative integer value for $q> q_*$. Since $f(\rho_*, q_\#, z_\#) =0$ with $z_\# \in (z_*, 0) \subset \Omega_1$ and $q_\# \ge q_*$, it holds Ind$\big(f(\rho_*, q, \cdot), \Omega_1\big) \ge 1$ for $q > q_*$. However, $f(\rho_*, q, z)\ne 0$ for any $z \in (z_*, 0)$ and $q > q_\#$. It implies Ind$\big(f(\rho_*, q, \cdot), \Omega_1\big)=0$ for $q > q_\#$, which is a contradiction. Therefore \eqref{E:temp-4.4} can not be true and we complete the proof of the lemma. 
\end{proof}

We are ready to prove the main lemma of the non-degeneracy of $F$ at $c=U(-h)$. In order for this result to be also applicable to the capillary gravity waves to improve  results in \cite{LiuZ21}, we include the surface tension in the consideration. Let 
\be \label{E:dispersion-CG}
\CF_\sigma (k, c) = F(k, c) - \sigma k^2 = Y(k, c)\big(U(0)-c\big)^2-U'(0)\big(U(0)-c\big)-g -\sigma k^2, \quad \sigma \ge 0,  
\ee 
whose zero points correspond to the eigenvalues of the capillary gravity waves linearized at the shear flow $U(x_2)$ if $\sigma>0$ (see \cite{LiuZ21}). 

\begin{lemma} \label{L:pcF}
Suppose $U\in C^6$, $k_0\in \R$, and $c_1< U(-h)$ satisfy 
\[
\CF_\sigma (k_0, U(-h))=0, \quad \CF_\sigma (k_0, c)\ne 0, \; c \in [c_1, U(-h)),  
\]
then  
%$(\p_k \CF_\sigma  \p_c \CF_\sigma) (k_0, U(-h)) <0$. 
\[
\CF_\sigma (k_0, c_1) \p_c \CF_\sigma (k_0, U(-h)) <0.
\]
\end{lemma}

\begin{proof}
According to Lemma \ref{L:e-v-basic-2}(1), it must hold $k_0\ne 0$. Due to the evenness of $\CF_\sigma$ in $k$, without loss of generality, we assume $k_0 >0$. Essentially we need to prove $\p_c \CF_\sigma (k_0, U(-h))\ne 0$, which will be done by an argument by contradiction based on Lemma \ref{L:temp-2} and a carefully constructed localized perturbation to $U$.  
We shall first prove the lemma under the assumption 
\[
\p_k \CF_\sigma (k_0, U(-h))\ne 0
\]
and then remove it at the end of the proof. 

Let $\gamma \in C^\infty(\R, [0, 1])$ be an auxiliary function satisfying 
\be \label{E:gamma-temp-1}
\gamma''(s)\ge 0, \; \forall s\in \R, \; \; \gamma''(s) = 0, \; \forall |s|\ge 1, \;\;  \int_{-1}^1 \gamma''(s') ds' = 1,  \;\; \gamma(-1) = \gamma'(-1)=0,
\ee  
which implies 
\be \label{E:gamma-temp-1.3}
|\gamma(s) - s \chi_{s > -1} |\le C, \quad \gamma' \in [0, 1], 
\ee
where $\chi$ is the characteristic function. For $ |\rho|\ll 1$ and 
\be \label{E:gamma-temp-2}
x_{20} \in (-h, 0), \quad 0< \delta <  L(x_{20})/2 \triangleq \min\{ -x_{20}, \, x_{20}+h \}/2, 
\ee
to be determined later, let 
\[
\BBU(\rho, x_2) = U(x_2) + \rho \delta \gamma \big( (x_2 - x_{20})/\delta \big),
\]   
which coincides with $U(x_2)$ in the $\delta$-neighborhood of $-h$. 
%to denote the corresponding function defined by $U(\rho, x_2)$ and 
Sometimes we skip some of the variables to prevent the notations from being overly complicated. Clearly,  
\[
\BBU(\rho, \cdot) \in C^6, \quad  
%, $\BBU(\rho, -h)=U(-h)$, and, when $|\rho| \ll 1$, 
U'+1 \ge \BBU' \ge U'. 
\]
Let $y_-(\rho, k, c, x_2)$ be the solution to the homogeneous Rayleigh equation \eqref{E:Ray-H1-1} at the shear flow $\BBU(\rho, \cdot)$ with the initial condition \eqref{E:y-pm} and $\BBY(\rho, k, c)$ and $\BBF (\rho, k, c)$ be defined in terms of $y_-(\rho, k, c, x_2)$ as in \eqref{E:Y} and \eqref{E:dispersion-CG}, respectively, which also depend on $x_{20}$ and $\delta$. 

To see the regularity of $\BBY$ and $\BBF$ in $\rho$ as well as $c$ near $\BBU(\rho, -h)=U(-h)$, 
let $u(\rho, k, c, x_2)$ be the solution to the homogeneous Rayleigh equation \eqref{E:Ray-H1-1} with the initial conditions 
\[
u\big(-h + L(x_{20})/2\big) = 1, \quad u'\big( -h + L(x_{20})/2\big) = \tilde Y (k, c) \triangleq \tfrac {y_-'(\rho, k, c, -h + L(x_{20})/2)}{y_-(\rho, k, c, -h + L(x_{20})/2)}, 
\]
which clearly implies 
\[
\BBY (\rho, k, c) = \tfrac {u'(\rho, k, c, 0)}{u(\rho, k, c, 0)}.
\]
Recall that $y_-(\rho, k, c, x_2)$ is independent of $\rho$ for $x_2 \in [-h, -h + L(x_{20})/2]$ due to the definition of $\BBU$. In particular, $y_-(\rho, k, c, -h + L(x_{20})/2) \ne 0$ as $c \le U(-h)$ (Lemma \ref{L:y0}(4)), which implies that  $\tilde Y(k, c)$ is well-defined for $c$ near $U(-h)$.  Through the same proof of Lemma \ref{L:Y-1} and Corollary \ref{C:F} (where $x_2=0$ being replaced by $x_2= -h + L(x_{20})/2$ does not affect the arguments), when restricted to $c_I\ge 0$, $\p_k^{j_1} \p_c^{j_2}\tilde Y(k, c) \in L_k^\infty L_{c_R}^q$ locally in $k$ and $c_R$, for $j_1, j_2 \ge 0$, $j_2 \le 2$, $j_1+j_2 \le 3$, and $q\in [1, \infty)$. 
%is smooth in $k$ and $C^{1, \alpha}$ in $c$ near $\BBU(\rho, -h)=U(-h)$. 
Apparently $u(\rho, k, c, 0)$ is smooth in $k$, $\rho$, $\tilde Y$, and $c$ near $U(-h)$. Therefore $\BBY$ and $\BBF$ are smooth in $\rho$ and satisfy the same regularity in $k$ and $c$. 
%are $C^{1, \alpha}$ for $c$ near $U(-h)$ and $\BBY$, $\BBF$, $\p_c \BBY$ and $\p_c \BBF$ are smooth in $k$ and $\rho$. 
Due to the assumptions on $k_0$ and $\p_k \CF_\sigma (k_0, U(-h))\ne 0$ and the semi-circle Theorem of the unstable modes of the water wave problems, the hypotheses of Lemma \ref{L:temp-2} are satisfied. Hence there exist $\rho_0>0$ and a $C^{1, \alpha}$ function $k_*(\rho)$ defined for $|\rho|\ll \rho_0$ such that 
\be \label{E:sign-pcF-1}
 \BBF (\rho, k_*(\rho), U(-h))=0, \quad \BBF (\rho, k_0, c_1) \p_c \BBF (\rho, k_*(\rho), U(-h)) \le 0, \quad \forall |\rho| \le \rho_0.      
\ee 

To prove the lemma by an argument by contradiction, we assume 
\be \label{E:temp-neg}
\p_c \CF_\sigma (k_0, U(-h)) = 0,  
\ee
and then prove that there exist $x_{20}$ and $\delta$ satisfying \eqref{E:gamma-temp-2} such that 
\be \label{E:temp-2.4}
\p_\rho \big(\p_c \BBF(\rho, k_*(\rho), U(-h))\big)\big|_{\rho=0}\ne 0. 
\ee
This would immediately lead to a contradiction to \eqref{E:sign-pcF-1} for some small $\rho\ne 0$.  

The definition \eqref{E:dispersion-CG} and $\BBF (\rho, k_*(\rho), U(-h))=0$ yield   
\be \label{E:temp-2.4.5} \begin{split}
%0= &\BBF (\rho, k_*(\rho), U(-h)) \\
\BBY(k_*(\rho), U(-h))= (\BBU( 0)- U(-h))^{-1} \BBU'(0) + (\BBU( 0)- U(-h))^{-2} (g+ \sigma k_*(\rho)^2),  
%& \big( (\BBU( 0)- c)^2\BBY(k, c) - \BBU'(0) (\BBU(0)- c) -g- \sigma k^2 \big)\big|_{(k, c) = (k_*(\rho), U(-h))},
\end{split} \ee
where we skipped some $\rho$ arguments. Subsequently we can compute  
\be \label{E:temp-2.4.6} \begin{split}
\p_c \BBF (\rho, k_*(\rho), U(-h)) = & \big( (\BBU( 0)- c)^2 \p_c \BBY  - 2(\BBU(0)- c) \BBY + \BBU'(0) \big)\big|_{(k, c) = (k_*(\rho), U(-h))} \\
=& \big( (\BBU( 0)- c)^2 \p_c \BBY  - \BBU'(0) - \tfrac {2(g + \sigma k^2)} {\BBU(0)- c}  \big)\big|_{(k, c) = (k_*(\rho), U(-h))}. 
\end{split} \ee 
The negation assumption \eqref{E:temp-neg} also yields 
\be \label{E:temp-2.4.7} 
\p_c Y (k_0, U(-h)) = (U( 0)- U(-h))^{-2} U'(0) + 2 (U( 0)- U(-h))^{-3} (g+ \sigma k_0^2). 
\ee
Meanwhile, from the Implicit Function Theorem, the definition of $\BBU$, \eqref{E:gamma-temp-1}, and \eqref{E:gamma-temp-2}, 
\be \label{E:temp-2.5}\begin{split}
\p_\rho k_* (0) = & - \frac {\p_\rho \BBF (0, k_0, U(-h))}{\p_k \CF_\sigma (k_0, U(-h))} \\
=& - \big(\p_k \CF_\sigma (k_0, U(-h)) \big)^{-1} \Big( (U(0) - U(-h))^2 \p_\rho \BBY(0, k_0, U(-h))  \\
& + \delta \gamma (- \tfrac {x_{20}}\delta) \big(2 (U(0)-U(-h))Y(k_0, U(-h)) - U'(0)\big)-(U(0) - U(-h)) \Big)  \\
=& - \big(\p_k \CF_\sigma (k_0, U(-h)) \big)^{-1} \Big( (U(0) - U(-h))^2 \p_\rho \BBY(0, k_0, U(-h))  \\
& + \delta \gamma (- \tfrac {x_{20}}\delta) \big(U'(0) + \tfrac {2(g+\sigma k_0^2)}{U(0)-U(-h)}\big)-(U(0) - U(-h)) \Big). 
\end{split} \ee
Moreover, from \eqref{E:temp-2.4.6} one may compute  
\begin{align*} 
\p_\rho \big(\p_c \BBF(\rho, k_*(\rho), & U(-h))\big)\big|_{\rho=0} =  - 1 -\frac {4\sigma k_0 } {U(0)- U(-h)}\p_\rho k_* (0) \\
& + (U(0)- U(-h))^2 \Big(\p_{c\rho} \BBY(0, k_0, U(-h))  + \p_{ck} Y (k_0, U(-h)) \p_\rho k_* (0) \Big) \\
&+ 2\delta \gamma (- \tfrac {x_{20}}\delta) \Big(  \frac {g + \sigma k_0^2} {(U(0)- U(-h))^2} + (U(0)-U(-h))\p_c Y(k_0, U(-h))\Big),
\end{align*}
where $\p_\rho k_*(0)$ and $\p_c Y(k_0, U(-h))$ should be substituted by \eqref{E:temp-2.5} and \eqref{E:temp-2.4.7}, respectively. According to  \eqref{E:gamma-temp-1.3}, $\delta \gamma (- \tfrac {x_{20}}\delta)$ can be approximated by $-x_{20}$. Along with  \eqref{E:temp-2.5}  and \eqref{E:temp-2.4.7}, we obtain 
\be \label{E:temp-2.6} \begin{split}
\big| \p_\rho \big(\p_c \BBF(\rho, k_*(\rho), U(-h))\big)\big|_{\rho=0} -& A_0 + A_1x_{20} - A_2 \p_\rho \BBY(0, k_0, U(-h)) \\
& - (U(0)- U(-h))^2 \p_{c\rho} \BBY(0, k_0, U(-h)) \big| 
%+1 - (U(0)- U(-h))^3 \Big(\frac {\p_{c\rho} \BBY(0, k_0, U(-h))}{U(0)- U(-h)} \\
%& \qquad \qquad - \frac {\p_{ck} Y (k_0, U(-h))}{\p_k \CF_\sigma (k_0, U(-h))}  \big((U(0) - U(-h)) \p_\rho \BBY(0, k_0, U(-h)) -1 \big)\Big) \\
%& \qquad \qquad +2 x_{20} \Big(  - \frac {\p_{ck} Y (k_0, U(-h))}{\p_k \CF_\sigma (k_0, U(-h))}  \Big( U'(0) (U(0)- U(-h))^2 + 2(g+\sigma k_0^2) (U(0)- U(-h))\Big)\\
%& \qquad \qquad 
%\frac {g + \sigma k_0^2} {(U(0)- U(-h))^2} + (U(0)-U(-h))\p_c Y(k_0, U(-h))  \Big| 
\le \tilde C \delta, 
\end{split}\ee
where 
\[
A_0= -1 + \big( \p_k \CF_\sigma (k_0, U(-h))\big)^{-1} \big( -4\sigma k_0  +  (U(0)- U(-h))^3 \p_{ck} Y (k_0, U(-h)) \big),
\]
\begin{align*}
A_1=& - \frac {(U(0)- U(-h))^2 }{\p_k \CF_\sigma (k_0, U(-h))} \Big( -\frac {4\sigma k_0 } {(U(0)- U(-h))^3} +\p_{ck} Y (k_0, U(-h)) \Big) \\
&\times \Big(U'(0) + \frac {2(g+\sigma k_0^2)}{U(0)-U(-h)}\Big)  + 2 \Big(\frac {U'(0)}{U(0)-U(-h)} + \frac {3(g+\sigma k_0^2)}{(U(0)-U(-h))^2}\Big),
\end{align*}
\[
A_2= - \frac {(U(0)- U(-h))^4 }{\p_k \CF_\sigma (k_0, U(-h))} \Big( -\frac {4\sigma k_0 } {(U(0)- U(-h))^3} +\p_{ck} Y (k_0, U(-h)) \Big), 
\]
and $\tilde C$ is proportional to $|A_1|$ depending on $U$, $g$, $\sigma$, $\gamma$, $k_0$, but independent of $x_{20}$ and $\delta$ satisfying \eqref{E:gamma-temp-2}.
In the rest of the proof, we will estimate $\p_\rho \BBY(0, k_0, U(-h))$ and $\p_{c\rho} \BBY(0, k_0, U(-h))$ carefully to show that there exists $x_{20} \in (-h, 0)$ such that \eqref{E:temp-2.4} holds for $0< \delta \ll 1$ based on \eqref{E:temp-2.6}. In the following the generic constant $\tilde C$ depends on $U$, $g$, $\sigma$, $\gamma$, and $k_0$, but always independent of $x_{20}$ and $0 < \delta \ll 1$.

%\noindent 
$\bullet$ {\it Estimates of $\p_\rho \BBY (0, k_0, U(-h))$.}  For $c \in (-\infty, U(-h)]$, let  
\[
\tilde y (\rho, k, c, x_2) = \tfrac {y_- (\rho, k, c, x_2)}{y_- (\rho, k, c, 0)}> 0, 
\;\; \tilde y_0(x_2) = \tilde y (0, k_0, U(-h), x_2), 
\quad \forall x_2 \in (-h, 0], 
\]
which are well-defined and positive due to Lemma \ref{L:y0}(4) and satisfies \eqref{E:temp-4} and those estimates given in Lemma \ref{L:temp-1}. 
Differentiating \eqref{E:temp-4} with respect to $\rho$, we have 
\[
- \p_\rho \tilde y'' + \big(k^2 + \tfrac {\BBU''}{\BBU - c}\big)\p_\rho \tilde y = - \p_\rho \big( \tfrac {\BBU''}{\BBU - c}\big) \tilde y, \quad \p_\rho \tilde y (-h)=\p_\rho \tilde y(0)=0.
\]
For $c\le  U(-h)$, 
%this is in the setting of 
the claim in the proof of Lemma \ref{L:neutral-M} applies 
%and \eqref{E:F-signs-0} 
and yields  
\be \label{E:temp-5.5} 
\p_\rho \tilde y (x_2) = - \int_{x_2}^0 \frac {\tilde y (x_2)}{\tilde y (x_2')^{2}} \int_{-h}^{x_2'}  \p_\rho \big( \frac {\BBU'' (\rho, x_2'')}{\BBU(\rho, x_2'') - c}\big) \tilde y(x_2'')^2  dx_2'' dx_2', 
\ee
\[
\p_\rho \BBY (0, k_0, U(-h)) = \p_\rho \tilde y' (0, k_0, U(-h), 0) = \int_{-h}^{0}  \p_\rho \big( \frac {\BBU''}{\BBU - U(-h)}\big)\big|_{\rho =0} \tilde y_0^2  dx_2.
\] 
%\[
%\p_\rho \tilde y' (x_2) = - \int_{x_2}^0 \frac {\tilde y' (x_2)}{\tilde y (x_2')^{2}} \int_{-h}^{x_2'}  \p_\rho \big( \frac {U'' (\rho, x_2'')}{U(\rho, x_2'') - c}\big) \tilde y(x_2'')^2  dx_2'' dx_2' + \frac 1{\tilde y (x_2)} \int_{-h}^{x_2}  \p_\rho \big( \frac {U'' }{U- c}\big) \tilde y^2  dx_2'.
%\]
One may compute 
\be \label{E:temp-6} \begin{split} 
& \p_\rho \big( \frac {\BBU''(\rho, x_2) }{\BBU(\rho, x_2)- c}\big)\Big|_{\rho=0} = \frac {\frac 1\delta \gamma''(\frac {x_2-x_{20}}\delta) }{U(x_2)- c} - \frac {\delta \gamma(\frac {x_2-x_{20}}\delta) U''(x_2) }{(U(x_2)- c)^2}. 
%One is reminded that the $O(|x_2+h|^{-2})$ singularity in $\p_\rho \big( \frac {\BBU''}{\BBU - c}\big)$ near $x_2=-h$ is canceled by $\tilde y_0 (x_2)^2 = O(|x_2+h|^2)$. 
% which, along with $\tilde y_0'(-h) > 0$, implies that the above integrands and functions are continuous.  
%Using \eqref{E:gamma-temp-1}, \eqref{E:temp-5}, 
%\implies & |\p_\rho \tilde y (x_2) | \le \tilde C \int_{x_2}^0 \frac {x_2+h }{(x_2'+h)^2} \int_{-h}^{x_2'} \frac {x_2''+h}\delta  \gamma''(\frac {x_2''-x_{20}}\delta) + \delta +\rho dx_2'' dx_2' \le \tilde C.
\end{split}\ee
From \eqref{E:gamma-temp-1}, Lemma \ref{L:temp-1}, and 
%,  \eqref{E:temp-6}, and 
\[
\big| \big(\tfrac {\tilde y_0^2}{U-U(-h)} \big)'\big| \le 2 |\tilde y_0'| \big|\tfrac {\tilde y_0}{U-U(-h)}\big|  + U' \big| \tfrac {\tilde y_0^2}{(U-U(-h))^2} \big| \le \tilde C, 
\]
along with \eqref{E:gamma-temp-1.3} we obtain an estimate on $\p_\rho \BBY(0, k_0, U(-h))$   
\be \label{E:temp-7} \begin{split}
& \Big| \p_\rho \BBY (0, k_0, U(-h)) - \frac {\tilde y_0 (x_{20})^2}{U(x_{20})-U(-h)} + \int_{x_{20}}^0 \frac {(x_2-x_{20}) U''(x_2) \tilde y_0(x_2)^2 }{(U(x_2)- U(-h))^2} dx_2  \Big| \\
\le & \tilde C \delta   +  \int_{x_{20}-\delta}^{x_{20}+\delta} \frac 1\delta \gamma''(\frac {x_2-x_{20}}\delta) \Big| \frac {\tilde y_0^2}{U - U(-h)} - \frac {\tilde y_0 (x_{20})^2}{U(x_{20})-U(-h)}\Big| dx_2 \le \tilde C\delta, 
\end{split} \ee
where we used the fact that $\gamma''\ge 0$ is supported in $[-1, 1]$ with total integral equal to 1. 

%\noindent 
$\bullet$ {\it Estimates of $\p_{c\rho} \BBY (0, k_0, U(-h))$.} Differentiating \eqref{E:temp-4} with respect to $\rho$ and $c$ yields 
\[
- \p_{c\rho} \tilde y'' + \big(k^2 + \tfrac {\BBU''}{\BBU - c}\big)\p_{c\rho} \tilde y = - \p_{c\rho} \big( \tfrac {\BBU''}{\BBU - c}\big) \tilde y - \p_{\rho} \big( \tfrac {\BBU''}{\BBU - c}\big) \p_c \tilde y - \p_{c} \big( \tfrac {\BBU''}{\BBU - c}\big) \p_\rho \tilde y,
% \; \p_{c\rho} \tilde y(-h) = \p_{c\rho} \tilde y (0)=0.
\]
with zero boundary values at $x_2=-h, 0$. 
Again from \eqref{E:F-signs-0}, we have 
\be \label{E:temp-7.5} \begin{split}
& \p_{c\rho} \BBY (0, k_0, U(-h))= \p_{c\rho} \tilde y' (0, k_0, U(-h), 0) = I_1+ I_2+I_3\\
\triangleq & \int_{-h}^0  \p_{c\rho} \big( \frac {\BBU''}{\BBU - c}\big) \tilde y^2 +\p_{\rho} \big( \frac {\BBU''}{\BBU - c}\big) \tilde y  \p_c\tilde y + \p_{c} \big( \frac {\BBU''}{\BBU - c}\big)  \tilde y \p_\rho\tilde y   dx_2\Big|_{(\rho, k, c)= (0, k_0, U(-h))},
\end{split} \ee
where $\p_c \tilde y$ and  $\p_\rho\tilde y$ were given in \eqref{E:temp-8.5} and \eqref{E:temp-5.5}, respectively. Using  \eqref{E:temp-8.5} and  through an integration by parts, 
it follows 
\begin{align*}
I_2 = & \int_{-h}^0  \p_{\rho} \big( \frac {\BBU''}{\BBU - c}\big)  \tilde y \p_c \tilde y   dx_2\Big|_{(\rho, k, c)= (0, k_0, U(-h))}  \\
=& - \int_{-h}^0   \Big(\p_{\rho} \big( \frac {\BBU''}{\BBU - c}\big)  \tilde y^2\Big)(x_2) \int_{x_2}^0 \frac 1{\tilde y (x_2')^{2}} \int_{-h}^{x_2'}  \Big(\p_c \big( \frac {\BBU''}{\BBU - c}\big) \tilde y^2 \Big)(x_2'')  dx_2'' dx_2'   dx_2\\
=& - \int_{-h}^0  \frac 1{\tilde y (x_2)^{2}}\int_{-h}^{x_2}  \Big(\p_{\rho} \big( \frac {\BBU''}{\BBU - c}\big)  \tilde y^2\Big)(x_2') dx_2'  \int_{-h}^{x_2}  \Big(\p_c \big( \frac {\BBU''}{\BBU - c}\big) \tilde y^2 \Big)(x_2') dx_2'   dx_2,
\end{align*}
evaluated at $(\rho, k, c)= (0, k_0, U(-h))$. Through another  integration by parts in a similar fashion applied to $I_3$, we obtain $I_2=I_3$. 
 
Like \eqref{E:temp-7} for $\p_\rho \BBY(0, k_0, U(-h))$, we will also identify the leading terms in $I_{1,2,3}$ of $\p_{c\rho} \BBY$. From Lemma \ref{L:temp-1}, we have 
\begin{align*}
\big| \big(\tfrac {(\tilde y \p_c \tilde y)(0, k_0, U(-h), x_2) }{U(x_2)-U(-h)} \big)' \big| \le & \tilde C \big( 1+ | \log (x_2+h)|\big),
\end{align*}
which implies the $C^\alpha$ continuity of $\frac {\tilde y \p_c \tilde y }{U-c}\big|_{(\rho, k)= (0, k_0)}$ in $x_2$ for any $\alpha\in [0, 1)$. Since \eqref{E:gamma-temp-2} implies $x_{20} -\delta +h > \delta$, with the above inequality, \eqref{E:gamma-temp-1.3}, and \eqref{E:temp-6}, 
we are ready to obtain the leading order term of $I_2=I_3$ 
\be  \label{E:temp-9} \begin{split}
& \Big| I_2- \frac {(\tilde y \p_c \tilde y)(0, k_0, U(-h), x_{20}) }{U(x_{20})-U(-h)} + \int_{x_{20}}^0 \frac {(x_2-x_{20}) (\tilde y \p_c \tilde y)(0, k_0, U(-h), x_{2}) U'' }{(U(x_{2})-U(-h))^2}dx_2  \Big| \\
\le & \Big(\int_{x_{20}-\delta}^{x_{20}+\delta}  \frac 1\delta \gamma''(\frac {x_2-x_{20}}\delta) \Big| \frac {\tilde y \p_c \tilde y}{U - c} - \frac {(\tilde y \p_c \tilde y)(x_{20}) }{U(x_{20})-c}\Big| dx_2\\
% \qquad \qquad \qquad \qquad   
& + \tilde C \delta\int_{\delta-h}^0   \frac { |\p_c \tilde y|\tilde y}{(U-c)^2} dx_2\Big) \Big|_{(\rho, k, c)=(0, k_0, U(-h))} \\
\le & \tilde C \int_{x_{20}-\delta}^{x_{20}+\delta} \delta^{\alpha-1} \gamma''(\frac {x_2-x_{20}}\delta) dx_2 + \tilde C \delta\int_{\delta-h}^0 \big(1+ | \log (x_2+h)|\big)  dx_2 \le \tilde C \delta^\alpha.
%, \quad \; \forall \alpha \in [0, 1).
\end{split} \ee

The last term $I_1$ is handled similarly starting with 
\[
\p_{c\rho} \Big( \frac {\BBU''}{\BBU - c}\Big) (x_2) \Big|_{\rho=0} = \frac {\frac 1\delta \gamma''(\frac {x_2-x_{20}}\delta) }{(U(x_2)- c)^2} - \frac {2\delta \gamma(\frac {x_2-x_{20}}\delta) U''(x_2) }{(U(x_2)- c)^3}. 
\]
Moreover we can estimate using 
%\eqref{E:temp-4} and 
Lemma \ref{L:temp-1} 
\begin{align*}
\Big|\Big(\frac {\tilde y_0^2}{(U-U(-h))^2} \Big)'\Big| =& \frac {2\tilde y_0}{(U-U(-h))^2} \Big| \tilde y_0' - \frac {U' \tilde y_0}{U-U(-h)} \Big|  \le \tilde C |x_2+h|. 
\end{align*}
Again, using \eqref{E:gamma-temp-1.3} we obtain   
\begin{align*} 
& \Big| I_1- \frac {\tilde y_0 (x_{20})^2 }{(U(x_{20})-U(-h))^2} + \int_{x_{20}}^0 \frac{2(x_2-x_{20}) U''(x_2) \tilde y_0(x_2)^2}{(U(x_2)-U(-h))^3} dx_2  \Big|\\
\le & \int_{x_{20}-\delta}^{x_{20}+\delta} \frac 1\delta \gamma''(\frac {x_2-x_{20}}\delta)\Big| \frac {\tilde y_0^2}{(U - U(-h))^2} - \frac {\tilde y_0 (x_{20})^2 }{(U(x_{20})-U(-h))^2}\Big| dx_2
%& \qquad \qquad \qquad  \qquad \qquad \qquad \qquad   
+  \int_{\delta-h}^0  \frac {\tilde C \delta  \tilde y_0^2}{(U-U(-h))^3} dx_2  \\
\le & \tilde C \delta +\tilde C \int_{\delta-h}^{0}   \frac \delta{x_2+h}   dx_2 \le \tilde C \delta^\alpha, \quad \; \forall \alpha \in [0, 1).
\end{align*}
Therefore \eqref{E:temp-7.5} implies  
\be  \label{E:temp-10} \begin{split}
\Big|\p_{c\rho} \BBY (0, & k_0,  U(-h))  - \frac {\tilde y_0 (x_{20})^2 }{(U(x_{20})-U(-h))^2} - \frac {2\tilde y_0 (x_{20})  \tilde y_1 (x_{20}) }{U(x_{20})-U(-h)} \\
&+  \int_{x_{20}}^0 \frac{2(x_2-x_{20}) U''(x_2) }{(U(x_2)-U(-h))^2} \Big( \tilde y_0 (x_2) \tilde y_1 (x_2) + \frac {\tilde y_0 (x_2) ^2}{U(x_{2})-U(-h)} \Big) dx_2  \Big|\le \tilde C \delta^{\frac 12},
\end{split}\ee
where 
\[
\tilde y_1 (x_2)= \p_c \tilde y(0, k_0, U(-h), x_2).
\]

From \eqref{E:temp-2.6}, \eqref{E:temp-7}, and \eqref{E:temp-10} we obtain 
\be \label{E:temp-11}
|\p_\rho \big(\p_c \BBF(\rho, k_*(\rho), U(-h))\big)\big|_{\rho=0} - 
%(U(0)- U(-h))^2 
I(x_{20}) | \le \tilde C \delta^{\frac 12}
\ee
where 
\be \label{E:temp-12} \begin{split}
I(x_{20}) = & A_0 - A_1 x_{20} + A_2 I_4(x_{20}) + (U(0)- U(-h))^2 I_5(x_{20}),\\
I_4(x_{20}) = &  \frac {\tilde y_0 (x_{20})^2}{U(x_{20})-U(-h)} - \int_{x_{20}}^0 \frac {(x_2-x_{20}) U''(x_2) \tilde y_0(x_2)^2 }{(U(x_2)- U(-h))^2} dx_2 \\
I_5(x_{20}) = & \frac {\tilde y_0 (x_{20})^2 }{(U(x_{20})-U(-h))^2} + \frac {2 \tilde y_0 (x_{20})  \tilde y_1 (x_{20}) }{U(x_{20})-U(-h)},\\
& -  \int_{x_{20}}^0 \frac{2(x_2-x_{20}) U''(x_2) }{(U(x_2)-U(-h))^2} \Big( \tilde y_0 (x_2) \tilde y_1 (x_2) + \frac {\tilde y_0 (x_2) ^2}{U(x_{2})-U(-h)} \Big) dx_2, \\
\end{split}\ee
and $A_0, A_1$, and $A_2$ are defined right bellow \eqref{E:temp-2.6}. Clearly for \eqref{E:temp-2.4} to hold for some $x_{20} \in (-h, 0)$ and $0< \delta \ll 1$, it suffices to show $I(x_{20})$, which is independent of $\delta$, is not identically zero. This will be achieved by computing $I'' (x_{20})$. 

%$\bullet$ {\it Computing $I_4'' (x_{20})$.} 
Direct calculations and using the Rayleigh equation \eqref{E:temp-4} yield  
\begin{align*} 
I_4'' =&  \frac {2 (\tilde y_0')^2 + 2 \tilde y_0 \tilde y_0''}{U-U(-h)} -  \frac {4 U' \tilde y_0 \tilde y_0' + U'' \tilde y_0^2}{(U-U(-h))^2} +  \frac { 2(U')^2 \tilde y_0^2}{(U-U(-h))^3} -  \frac {U'' \tilde y_0^2}{(U-U(-h))^2} \notag \\
= & \frac {2 (\tilde y_0')^2 + 2 k_0^2 \tilde y_0^2 }{U-U(-h)} -  \frac {4 U' \tilde y_0 \tilde y_0' }{(U-U(-h))^2} +  \frac { 2(U')^2 \tilde y_0^2}{(U-U(-h))^3} \notag \\
=& \frac {2 k_0^2 \tilde y_0^2 }{U-U(-h)} + \frac {2}{U-U(-h)} \Big( \tilde y_0' - \frac {U' \tilde y_0 }{U-U(-h)} \Big)^2.
%\label{E:temp-13} 
\end{align*}
Similar direct calculations lead to 
 \begin{align*} 
I_5'' =&  \frac {2 (\tilde y_0')^2 + 2 \tilde y_0 \tilde y_0''}{(U-U(-h))^2} -  \frac {8 U' \tilde y_0 \tilde y_0' + 2 U'' \tilde y_0^2}{(U-U(-h))^3} +  \frac { 6(U')^2 \tilde y_0^2}{(U-U(-h))^4} \notag \\
& + \frac {4 \tilde y_0'  \tilde y_1' + 2\tilde y_0'' \tilde y_1 + 2 \tilde y_0 \tilde y_1''}{U-U(-h)} - \frac {4 U' (\tilde y_0'  \tilde y_1 + \tilde y_0 \tilde y_1') + 2U'' \tilde y_0 \tilde y_1}{(U-U(-h))^2} + \frac {4 (U')^2 \tilde y_0 \tilde y_1}{(U-U(-h))^3} \notag \\
&- \frac {2U'' \tilde y_0 \tilde y_1}{(U-U(-h))^2} -  \frac {2U'' \tilde y_0^2}{(U-U(-h))^3}.  
\end{align*}
Using the Rayleigh equations \eqref{E:temp-4} and \eqref{E:temp-8} it follows 
\begin{align*}
I_5'' =& \frac {2 (\tilde y_0')^2 + 2 k_0^2 \tilde y_0^2}{(U-U(-h))^2} -  \frac {8 U' \tilde y_0 \tilde y_0' }{(U-U(-h))^3} +  \frac { 6(U')^2 \tilde y_0^2}{(U-U(-h))^4} \\
& + \frac {4 \tilde y_0'  \tilde y_1' + 4k_0^2 \tilde y_0 \tilde y_1 }{U-U(-h)} - \frac {4 U' (\tilde y_0'  \tilde y_1 + \tilde y_0 \tilde y_1') }{(U-U(-h))^2} + \frac {4 (U')^2 \tilde y_0 \tilde y_1}{(U-U(-h))^3} 
\end{align*}
Reorganizing the terms we obtain 
\begin{align*} 
I_5''  =&  \frac { 2 k_0^2 \tilde y_0^2}{(U-U(-h))^2}  + \frac {2}{(U-U(-h))^2} \Big( \tilde y_0' - \frac {U' \tilde y_0 }{U-U(-h)} \Big) \Big( \tilde y_0' - \frac {3U' \tilde y_0 }{U-U(-h)} \Big) \notag \\
&+ \frac {4k_0^2 \tilde y_0 \tilde y_1 }{U-U(-h)} + \frac {4}{U-U(-h)} \Big( \tilde y_0' - \frac {U' \tilde y_0 }{U-U(-h)} \Big) \Big( \tilde y_1' - \frac {U' \tilde y_1 }{U-U(-h)} \Big).
\end{align*}
As $A_0, A_1, A_2$ are independent of $x_{20}$, Lemma \ref{L:temp-1} and the above computations imply 
\[
\lim_{x_{20} \to (-h)+} I'' (x_{20}) = 2 k_0^2 (U(0)- U(-h))^2 \tilde y_0'(-h)^2 /(3 U' (-h)^2) >0, 
\]
and thus $I(x_{20})$ is not a constant of $x_{20}$. 

Therefore, according to \eqref{E:temp-11}, \eqref{E:temp-2.4} holds if $x_{20}$ is close to $-h$ and $0< \delta \ll1$. This contradicts \eqref{E:sign-pcF-1} and thus \eqref{E:temp-neg} can not occur. This prove the lemma under the assumption $\p_k \CF_\sigma (k_0, U(-h)) \ne 0$.  

Finally, we shall complete the proof of the lemma in the case of $\p_k \CF_\sigma (k_0, U(-h))=0$. From Lemma \ref{L:e-v-basic-2}, this can happen only if $\sigma>0$, namely, in the case of the linearized capillary gravity waves. Our strategy is to consider the problem with modified  parameters 
\[
\tilde g(\ep)= g + \ep k_0^2, \quad \tilde \sigma(\ep)= \sigma -\ep, \ \text{ where } \
0< \ep \ll 1. 
\]
The corresponding function $\tilde \CF (\ep, k, c)$ associated to the eigenvalue problem becomes 
\[
\tilde \CF (\ep, k, c) = (U(0) - c)^2 Y(k, c) - U'(0) (U(0)-c) - \tilde g -\tilde \sigma k^2 = \CF_\sigma (k, c) + \ep(k^2 - k_0^2),
\]
which satisfies 
\[
\tilde \CF (\ep, k_0, c) = \CF_\sigma (k_0, c), \quad \p_k \tilde \CF (\ep, k_0, U(-h)) = 2\ep k_0. 
\]
Therefore, for any $\ep\le \sigma$, $\tilde \CF (\ep, \cdot, \cdot)$ satisfies the assumption on $\CF_\sigma$ along with $\p_k \tilde \CF (\ep, k_0, U(-h)) \ne 0$. Hence the above proof implies $\p_c \CF_\sigma (k_0, U(-h)) = \p_c \tilde \CF (\ep, k_0, U(-h)) \ne0$. It completes the proof of the lemma. 
\end{proof}

%\subsection{Eigenvalue distribution} \label{SS:e-values}

\subsection{Bifurcation analysis} \label{SS:bifurcation} 

With the technical preparations of the previous subsections, we shall consider the bifurcation of the unstable eigenvalues from limiting neutral modes. 
%and the asymptotic properties of eigenvalues as $|k| \to \infty$. They are connected through the analytic continuation argument and thus provide a complete picture of the distribution of the eigenmodes
%%. Such a picture of the total eigenvalue distributions would also be given 
%under certain conditions. 
%\noindent $\bullet$ {\it Bifurcations.} 
In the following lemma we incorporate  the bifurcation analysis of $c$ near $U(-h)$ and inflection values of $U$, which often leads to linear instability. The lemma is stated for $\CF_\sigma (k, c)$ defined in \eqref{E:dispersion-CG} with $\sigma \ge 0$ so that it also applies to linearized capillary gravity waves.    

\begin{lemma} \label{L:bifurcation}
Suppose $U\in C^{l_0}$ and $(k_0, x_{20}) \in \big(\R \times [-h, 0) \big) 
%\setminus \{(0, -h)\}
$ satisfy  
\[
\CF_\sigma (k_0, c_0)=0, \; \text{ where } \; c_0 =U(x_{20}), \;\; \text{ and }\;\; \p_{c_R} \CF_\sigma (k_0, c_0) =\p_{c_R} F (k_0, c_0)\ne 0, 
\]
then there exist $\ep>0$, $0< \rho \in U(0)- c_0$, and $\CC \in C^{1, \alpha} \big([k_0-\ep, k_0+ \ep], \C\big)$ for any $\alpha \in [0, 1)$ if $x_{20}=-h$ and $l_0\ge 5$, or $\CC \in C^{l_0-2}$ if $x_{20} \in (-h, 0)$ and $l_0\ge 4$, such that $\CC(k_0)=c_0$, 
\[
\CC'(k_0) = - \p_k \CF_\sigma (k_0, c_0) / \p_{c_R} F (k_0, c_0) \implies \CC_I'(k_0) = \p_k \CF_\sigma  \p_{c_R} F_I/ |\p_{c_R} F |^2 \big|_{(k, c)=(k_0, c_0)},  
\]
and 
\be \label{E:bif-curve-1}
\CF_\sigma (k, c) =0, \;\; |k -k_0|\le \ep, \; |c_R-c_0|,c_I\in [0, \rho], \;
\text{ iff } \; c = \CC(k), \; \CC_I(k)\ge 0. 
\ee
%and
%\[
%\CC(k_0)=c_0, \;\; \CC(k) \notin U([-h, 0]), \; \forall 0< |k -k_0|\le \ep. 
%%\;\; \p_c F(k, \CC(k)) \ne 0, \; \text{ if } \;  \CC_I(k) \ge 0.
%\]
Moreover, the following properties hold for $k \in [k_0-\ep, k_0+ \ep]$ and some $\tilde C>0$ determined by $\alpha$, $k_0$ and $U$. 
\begin{enumerate}
\item Suppose $\p_{c_R} \RP\,  \CF_\sigma (k_0, c_0)=\p_{c_R} F_R (k_0, c_0) \ne 0$, then  
\[
\big| \CC_I(k) + \big(U(0)- c_0\big)^2Y_I \big(k, \CC_R(k)\big)/ \p_{c_R} F_R (k_0, c_0 )\big| \le \tilde C |k- k_0|^\alpha \big| Y_I \big(k, \CC_R(k)\big)\big|. 
\]
\item Suppose $\p_k \CF_\sigma (k_0, c_0) =0$, then either i.) $\sigma >0$ or ii.) $\sigma =0$, $k_0=0$, and $x_{20} \in (-h, 0)$. Moreover the following hold. 
\begin{enumerate}
\item If $k_0 \ne 0$, then 
\[
\big|\CC'(k) +  \p_k \CF_\sigma (k, c_0 )|\p_{c_R} F (k_0, c_0 )|^{-2} \overline {\p_{c_R} F (k_0, c_0 )} \big|
\le  \tilde C  |k-k_0|^{\alpha(1+\alpha)}. 
\]
%and 
%\[
%(k_0-k) \p_k \CF_\sigma (k, c_0 ) \ge \tilde C^{-1} |k-k_0|^2.
%\]
\item If $k_0 = 0$, then 
\[
\big|\CC'(k)/k +  2\p_K \CF_\sigma (k, c_0 )|\p_{c_R} F (0, c_0 )|^{-2} \overline {\p_{c_R} F (0, c_0 )} \big|\le  \tilde C  |K|^{\alpha(1+\beta)},  
\]
where $K=k^2$, $\beta=0$ if $\p_K \CF_\sigma(0, c_0) \ne 0$, and any $\beta \in (0, 1)$ if $\p_K \CF_\sigma(0, c_0) = 0$. 
%Moreover 
%\[
%\tilde C^{-1}  \le -\p_K \CF_\sigma (k, c_0 ) \le \tilde C, \; \text{ if } \; \p_K \CF_\sigma(0, c_0)\ne 0; 
%\]
%\[
%\tilde C^{-1} |K| \le -\p_K \CF_\sigma (k, c_0 ) \le \tilde C |K|, \; \text{ if } \; \p_K \CF_\sigma(0, c_0)= 0.
%\]
\end{enumerate}
\end{enumerate}
\end{lemma}

According to \eqref{E:bif-curve-1}, clearly $\CC(k)$ is relevant if and only if $\CC_I(k)\ge 0$. The %Implicit Function Theorem and Lemma \ref{L:e-v-basic-1}(4),  
formula for $\CC'(k_0)$ 
indicates the behavior of $\CC_I(k)$ for $|k-k_0| \ll 1$ if $\p_k \CF_\sigma (k_0, c_0) \p_{c_R} F_I(k_0, c_0) \ne 0$. The above statements (1) and (2) in combination are useful to provide such information in the degenerate cases including when $c_0= U(-h)$. 

\begin{remark} \label{R:bifurcation}
a.) Due to Lemma 
%\ref{L:e-v-basic-1}(5) and 
\ref{L:e-v-basic-2}(1), 
%and \ref{L:neutral-M-1}, 
$F(0, U(-h)) = F(k, U(0))=-g$ and thus $x_{20}=0$ and $(k_0, x_{20})=(0, -h)$  
%$F(0, U(x_2)), F(k, U(0)) \ne 0$ and thus $k_0=0$ and $x_{20}=0$ 
are actually excluded. \\
b.) Assume $\p_{c_R} F_R (k_0, c_0) \ne 0$, which in particular holds if $x_{20}=-h$ due to Lemma \ref{L:pcF}. Statement (1) yields that $\CC_I(k)$ has the opposite sign as $\p_{c_R} F_R (k_0, c_0) Y_I (k, \CC_R(k))$. Therefore, Lemma \ref{L:Y-def}(5) implies that $\CC_I(k)=0$ if $\CC_R(k) \notin U\big((-h, 0) \big)$ or $U'' \big( U^{-1} (\CC_R(k))\big) =0$. More importantly, \eqref{E:bif-curve-1} implies that whether $\CF_\sigma (k, \cdot)$ has zero points near $c_0$ for $0< |k-k_0| \ll 1$ is determined by the sign of $\p_{c_R} F_R (k_0, c_0) Y_I \big(k, \CC_R(k)\big)$. A sufficient test for the latter is obviously the signs of $U''(x_{20})$ and $U'''(x_{20})$.     \\
c.) From Lemma \ref{L:e-v-basic-1}(4), it holds that either $x_{20} =-h$, where $\p_{c_R} \IP\, \CF_\sigma (k_0, c_0)=\p_{c_R} F_I (k_0, c_0)=0$, or $x_{20} \in (-h, 0)$ with $U''(x_{20}) =0$. In the latter case, $U'''(x_{20}) \ne 0$ is equivalent to $\p_{c_R} F_I (k_0, c_0) \ne 0$ which is sufficient for $\p_{c_R} \CF_\sigma (k_0, c_0)\ne 0$. \\
d.) According to Lemmas \ref{L:e-v-basic-2}(2) and \ref{L:neutral-M-2} and \eqref{E:F-signs-0.2}, if a.) $x_{20}=-h$, or b.) $x_{20} \in (-h, 0)$ and $k_0>0$ is the greatest solution to $\CF_\sigma(k_0, c_0) =0$, then $\p_K^2 F(k_0, c_0)<0$. Hence, when $\p_k \CF_\sigma (k_0, c_0 )=0$, we have $\p_k^2 \CF_\sigma (k_0, c_0 )<0$ and 
%Moreover \eqref{E:F-signs-0.1} implies $\p_k F (k_0, c_0) =0$ occurs for such $c_0$ only if i.) $\sigma >0$ or ii.) $\sigma =0$, $k_0=0$, and $x_{20} \in (-h, 0)$. 
\[
(k_0-k) \p_k \CF_\sigma (k, c_0 ) \ge \tilde C^{-1} |k-k_0|^2.
\]
Hence in the above statement (2a), $-\p_k \CF_\sigma (k, c_0 )/ \p_{c_R} F (k_0, c_0 )$ gives the leading order term of $\CC_I(k)$ for $k$ near $k_0$. 
\end{remark}   
  
\begin{proof}   
Since $x_{20} \ne 0$, when restricted to the upper half plane $c_I\ge 0$, according to Lemma \ref{L:e-v-basic-1}(1)(2) and Corollary \ref{C:F}, $\CF_\sigma$ is $C^{1, \alpha}$ in $k$ and $c$ near $(k_0, c_0)$  (actually $C^{l_0-2}$ if $x_{20} \in (-h, 0)$ due to Lemma \ref{L:e-v-basic-1}(1)). Since $F_I = \IP\, \CF_\sigma$ is not continuous at $c \in U\big((-h, 0)\big) \subset \C$ in general, let $\tilde F(k, c) = \tilde F_R + i \tilde F_I \in \C$ be a $C^{1, \alpha}$ 
%or $C^{l_0-2}$ (depending on whether $x_{20}=-h$)  
extension of $\CF_\sigma$ into a neighborhood of $(k_0, c_0) \in \R \times \C$ which coincides with $\CF_\sigma$ for $c_I\ge 0$ (or $C^{l_0-2}$ extension if $x_{20} \in (-h, 0)$). The $2\times 2$ real  Jacobian matrix of $D_c \tilde F$ satisfies 
\[
D_c \tilde F (k_0, c_0) = \begin{pmatrix} \p_{c_R} \tilde F_R & \p_{c_I} \tilde F_R \\ \p_{c_R} \tilde F_I & \p_{c_I} \tilde F_I \end{pmatrix}\Big|_{(k_0, c_0)} =
\begin{pmatrix} \p_{c_R} \RP\, \CF_\sigma &- \p_{c_R} \IP\, \CF_\sigma \\ \p_{c_R} \IP\, \CF_\sigma & \p_{c_R} \RP \, \CF_\sigma  \end{pmatrix}\Big|_{(k_0, c_0)}, 
%\p_{c} F (k_0, c_0) I_{2\times 2},  \quad \p_c F \big(k_0, U(-h)\big)<0,
\]
where we also used the Cauchy-Riemann equation satisfied by $\CF_\sigma$ when restricted to $c_I\ge 0$. From 
\[
\det D_c \tilde F (k_0, c_0) = |\p_c \CF_\sigma (k_0, c_0)|^2 \ne 0,
\]
the Implicit Function Theorem implies that all roots of $\tilde F(k, c)$ near $(k_0, c_0)$ form the graph of a $C^{1, \alpha}$ complex-valued function $\CC(k)$ which contains $(k_0, c_0)$ (or $\CC(k) \in C^{l_0-2}$ if $x_{20} \in (-h, 0)$). This and \eqref{E:f-conj} prove the existence and the basic properties of $\CC(k)$ and \eqref{E:bif-curve-1}. In the rest of the proof, we study the properties of $\CC(k)$. 

Suppose $\p_{c_R} F_R (k_0, c_0) \ne 0$. 
%which is equivalent to $\p_{c_R} \RP\, \CF_\sigma (k_0, c_0)\ne 0$. 
For any $k \in [k_0-\ep, k_0+ \ep]$, from the Mean Value Theorem, there exists $\tau$ between $0$ and $\CC_I(k)$ such that   
\[
0= \tilde F_I \big(k, \CC(k)\big) = F_I \big(k, \CC_R(k)\big) + \CC_I(k) \p_{c_I} \tilde F_I \big(k, \CC_R(k) + i \tau\big). 
\]
The $C^{1, \alpha}$ regularity of $\tilde F$ and $\CC (k)$ and the Cauchy-Riemann equation
%, and $\p_{c} F \big(k_0, U(-h)\big)<0$, and $Y_I (k, c) \ne 0$ for $c \in U\big((-h, 0)\big)$, 
imply  
\begin{align*}
\CC_I (k)= &-\frac {F_I \big(k, \CC_R(k)\big)}{\p_{c_I} \tilde F_I \big(k, \CC_R(k) + i \tau\big)}=-\frac { \big(U(0)- \CC_R(k)\big)^2Y_I \big(k, \CC_R(k)\big)}{\p_{c_I} \tilde F_I \big(k, \CC_R(k)\big) + O\big(|\CC_I(k)|^\alpha\big)} \\
=&-\frac {\big(U(0)- \CC_R(k)\big)^2Y_I \big(k, \CC_R(k)\big) }{\p_{c_I} F_I \big(k, \CC_R(k)\big) + O\big(|\CC_I(k)|^\alpha\big)} 
= -\frac { \big(U(0)- c_0 +O( |k-k_0|)\big)^2Y_I \big(k, \CC_R(k)\big)}{\p_{c_R} F_R (k_0, c_0) + O\big(|k- k_0|^\alpha\big)}.  
\end{align*}
The desired estimate on $\CC_I(k)$ in statement (1) follows immediately. 

In the rest of the proof, we consider the case of $\p_k \CF_\sigma (k_0, c_0) =0$. In this situation, Lemma \ref{L:e-v-basic-2}(1) and \ref{L:neutral-M} (and \eqref{E:F-signs-0.1} as well) implies either i.) $\sigma >0$ or ii.) $\sigma =0$, $k_0=0$, and $x_{20} \in (-h, 0)$, where we used $\p_k = 2 k \p_K$. The analysis relies on the following equality obtained from the Implicit Function Theorem 
\[
%be \label{E:temp-13}
\CC'(k) =  - D_c \tilde F (k, \CC(k) )^{-1} \p_k \tilde F (k, \CC(k) ).  
\]

Let us first consider the case of $k_0\ne 0$. Without loss of generality, we may assume $k_0>0$. The $C^{1, \alpha}$ regularity of $\tilde F$ and $\CC$ yields 
\[
|\p_k \tilde F (k, \CC(k) )|, |\p_k \tilde F (k, c_0 )|, |\CC'(k)| \le\tilde C |k-k_0|^\alpha, \quad |\CC(k)- c_0| \le \tilde C |k-k_0|^{1+\alpha},
\]
and thus 
\begin{align*}
\big|\CC'(k) +  \p_k \CF_\sigma (k, c_0 )|\p_{c_R} F (k_0, c_0 )|^{-2} \overline {\p_{c_R} F (k_0, c_0 )} \big|=& \big|\CC'(k) +  \p_k \tilde F (k, c_0 )/\p_{c_R} \tilde F (k_0, c_0 )\big| \\
\le & \tilde C  |k-k_0|^{\alpha(1+\alpha)}.  
\end{align*}
%Remark d.)  on the lower bound of $|\p_k \CF_\sigma (k, c_0 )|$ follows from the concavity \eqref{E:F-signs-0.1} of $ \CF_\sigma (\cdot, c_0 )\in \R$ in $K=k^2$ and its $C^{l_0-2}$ smoothness in $k$ ($l_0=5$ is assumed).  

In the case of $k_0=0$, where $x_{20} \in( -h, 0)$ must hold. 
%we use $K=k^2$ to replace $k$. 
Since the dependence of $Y$ and $F$ on $k$ is actually through $K=k^2$, the same conclusions in Lemmas \ref{L:y0}(8) and \ref{L:e-v-basic-1}(1) still hold that $F$ is $C^3$ in both $K$ and $c$ near $(0, c_0)$ when restricted to $c_I\ge 0$. Therefore $\CC$ can also be viewed as a function of $K$ and we have 
\[
%\be \label{E:temp-13}
\p_K \CC =  - D_c \tilde F (k, \CC )^{-1} \p_K \tilde F (k, \CC ).  
\]
Much as in the above, we first obtain 
\[
%\tilde C^{-1} |K|^\beta \le 
|\p_K \tilde F (k, \CC(k) )|, |\p_K \tilde F (k, c_0 )|, |\p_K \CC| \le\tilde C |K|^\beta, \quad |\CC(k)- c_0| \le \tilde C |K|^{1+\beta},
\] 
where $\beta=0$ if $\p_K \CF_\sigma(0, c_0) \ne 0$ and for any $\beta\in (0, 1)$ if $\p_K \CF_\sigma(0, c_0) = 0$. Along with the evenness of $\tilde F$ and $\CC$ in $k$, it implies, for $|k| \in (0, \ep]$,  
\begin{align*}
\big|\CC'(k)/k +  2\p_K \CF_\sigma (k, c_0 )|\p_{c_R} F (0, c_0 )|^{-2} \overline {\p_{c_R} F (0, c_0 )} \big|=&  2\big|\p_K \CC +  \p_K \tilde F (k, c_0 )/\p_{c_R} \tilde F (0, c_0 )\big| \\
\le & \tilde C  |K|^{\alpha( 1+\beta)}.  
\end{align*}
%The rest of statement (2b) in the case of $\p_K \CF_\sigma(0, c_0) = 0$ again follows from \eqref{E:F-signs-0.1}.
It completes the proof of the lemma. 
\end{proof}

In the following lemma, we consider a special case of bifurcation of unstable modes from an interior non-degenerate inflection value. 

\begin{lemma} \label{L:bifurcation-P}
Suppose $U\in C^4$ and $c_0 = U(x_{20}) \in U((-h, 0))$ satisfy 
\[
U''(x_{20}) =0 >U''' (x_{20}), \quad g_0 \triangleq (U(0)-c_0)^2 Y(0, c_0) - U'(0) (U(0)-c_0) >0, 
\]
then there exists $\ep_0, \delta>0$ determined by $U$ and $c_0$ such that for any $g= g_0+ \ep \in (g_0, g_0+\ep_0)$, there exist $C>0$ and $k(\ep)>0$ depending on $U$ and $c_0$ only and $\CC(\ep, k) = \CC_R(\ep, k) + \CC_I(\ep, k)$, $|k| \le k(\ep)$, $C^1$ in $\ep>0$ and $k$ and even in $k$, such that  
\[
\CC(\ep, k(\ep)) = c_0,  \quad \p_k \CC_I (\ep, k)/(2k) + \big(\p_K F \p_{c_R} F_I\big)\big|_{(0, c_0)} \to 0 \;\text{ as } \; \ep \to 0+  \text{ uniformly}, 
\]
where $K=k^2$, hence $\CC_I(\ep, k) > 0$ for $|k| < k(\ep)$, and moreover 
\[
F(\ep, k, c) =0, \quad |k| \le k(\ep), \; |c-c_0| \le \delta, 
\]
iff $c=\CC(\ep, k)$ or $c=\overline {\CC(\ep, k)}$. 
Here $F$ is defined as in \eqref{E:dispersion} in terms of $U$ and $g= g_0+\ep$. 
\end{lemma} 

According to Lemma \ref{L:neutral-M-3}(2), the assumption $g_0>0$ holds only if there exists $k_C > 0$ such that $y_-(k_C, c_0, 0)=0$. 
%i.e. $(k_C, c_0)$ is a singular neutral mode of the linearized channel flow. 
This lemma shows that the branch of the unstable eigenvalues bifurcating from the interior inflection value $c_0$, starting at the smaller wave number $k(\ep)$ which makes $(k(\ep), c_0)$ a singular neutral mode, connect back to $(-k(\ep), c_0)$. 

\begin{proof} 
Apparently $y_-(k, c_0, x_2)$ and $y_-'(k, c_0, x_2)$, and thus $Y(k, c_0)$ and $F(k, c_0)$ as well if $y_-(k, c_0, 0)\ne 0$, are smooth in $K=k^2$. From $\p_K F(k, c_0) >0$ given by \eqref{E:F-signs-0.1}, there exists $k(\ep)>0$ such that $k(0)=0$,  $F(\ep, k(\ep), c_0) =0$, and $k(\ep)^2$ is smooth in $0\le \ep \ll 1$. Hence there exists $C>0$ depending only on $U$ and $c_0$ such that $C^{-1} \sqrt {\ep} \le k(\ep) \le C \sqrt {\ep}$. Lemma \ref{L:bifurcation} and the assumption $U'''(x_{20})<0$ imply that, for any $\ep\ll1$, $F(\ep, k, c)=0$ has roots near $(k(\ep), c_0)$ with $0<  k(\ep) - |k| \ll 1$ and $\IP\, c >0$. In order to obtain a global branch for each small $\ep>0$, let   
\[
\tilde F(\ep, \tau, c) = F(\ep, k(\ep) \tau, c), \quad |\tau| \le 1, \; |c-c_0| \ll 1, \; 0\le \ep \ll 1.  
\]
According to Lemma \ref{L:Y-def}, $F$ is $C^2$ in $c$ (when restricted to $c_I\ge 0$ near $c_0$) and $K= k^2$, while $k(\ep)^2$ is smooth in $\ep$, $\tilde F$ is also $C^2$ in such $c$, $\ep$, and (also even in) $\tau$. We also extend $F$ and $\tilde F$ as $C^2$ functions defined on a whole neighborhood of $c_0$. Since $\tilde F(0, \tau, c_0)=0$ and the Jacobian satisfies 
\[
\det D_c \tilde F(0, 0, c_0) = \det D_c F(0, 0, c_0) = |\p_c F(0, 0, c_0)|^2 \ge |\p_{c_R} F_I(0, 0, c_0)|^2 >0, 
\]
the Implicit Function Theorem yields the (even in $\tau$) roots $\tilde \CC(\ep, \tau)$ of $\tilde F(\ep, \tau, c)$. Clearly $\tilde \CC(\ep, \pm 1) = c_0$ due to the definition of $k(\ep)$. One may compute, for $\tau \in (0, 1)$ and $0< \ep \ll 1$, 
\begin{align*}
\p_\tau \tilde \CC_I & =  - \IP \big( (D_c \tilde F)^{-1} \p_{\tau} \tilde F\big) = - 2 k(\ep)^2 \tau \IP \big( (D_c \tilde F)^{-1} \p_{K} F\big)\\
& \le k(\ep)^2 \tau  \big(\p_K F \p_{c_R} F_I / |\p_{c}  F|^2\big)\big|_{(0, c_0)} < 0, 
\end{align*}
where $U'''(x_{20}) <0$ is used. Hence $\tilde \CC_I > 0$ for $|\tau|<1$. Obviously $\CC(\ep, k) = \tilde \CC(\ep, k/k(\ep))$ satisfies the desired properties. 
\end{proof}

\subsection{Eigenvalues for $|k|\gg1$} \label{SS:E-asymp}

A major difference between the eigenvalue distributions of the linearized gravity waves and that of the capillary gravity waves is when $|k|\gg1$. We shall work on $\BF$ and $F$ of the linearized gravity waves. The analysis is divided into several steps starting with some rough bounds of the zeros of $\BF(k, c)$. 

\begin{lemma} \label{L:e-v-asymp-1}
Assume $l_0\ge 3$, here exists $k_0>0$ depending only on $U$, such that, for any $|k|\ge k_0$, solutions to \eqref{E:BF} (which is $\BF(k, c)=0$) belong to $S_{k}^L \cup S_{k}^R$ where  
%\[
%(g|k|)^{\frac 12} |U(0) -c| \in [1/2, 2], \quad |c_I| \in C k^{-\frac 78} .
%\]
\[
S_{k}^{L, R}= \{ c = c_R + c_I \in \C \mid \pm (|k|/g)^{\frac 12} (U(0) -c_R)\in [1/2, 2], \ |c_I| \le k^{-\frac 78}\}.  
\]
\end{lemma} 

\begin{proof}
From Lemma \ref{L:y_pm-1}, 
\[
\exists k_0>0, \text{ s. t. } | y_-(k, c, x_2)| \ge (\mu/2) \sinh \mu^{-1} (x_2+h)  >0, \quad \forall |k| \ge k_0, \, c\in \C, \, x_2 \in (-h, 0], 
\]
where we recall $\mu= (1+k^2)^{-\frac 12}$. 
Hence we are able to work on \eqref{E:dispersion} with $F(k, c)$ and $Y(k, c)$. Due to the evenness of the problem in $k$, we only consider $k >0$.   

%{\it Step 1. Identifying a possible region of the roots of $F(k, \cdot)$.} 
From Lemma \ref{L:Y-def}(3), for any $\alpha \in (0, \frac 12)$, there exist $k_0, C>0$ depending only on $\alpha$ and $U$, such that for any $k> k_0$,  
\[
\big|F(k, c) - (U(0)-c)^2k + U'(0) (U(0)-c) +g\big| \le C k^{1-\alpha} \big(|U(0)-c|^{\frac 74}+ |U(0)-c|^2\big). 
\]

On the one hand, if $|U(0)-c| \le C_1 k^{-\frac 12}$ where $C_1 \in (0, 1)$, then 
\[
|F(k, c)| \ge g - C_1^2 - C C_1 k^{\frac 18 -\alpha}. 
\]
Taking $\alpha =\frac 13$ and  sufficiently large $k_0$ and $C_1^{-1}$ determined by $U$, we obtain $F(k, c)\ne0$ if $k\ge k_0$ and $|U(0)-c| \le C_1 k^{-\frac 12}$. 

On the other hand, if $|U(0)-c| \ge \sqrt{2g /k}$, then  
\begin{align*}
|F(k, c)| \ge & |U(0)-c|^2 k - |U'(0) (U(0)-c) +g| - C k^{1-\alpha} \big(|U(0)-c|^{\frac 74}+ |U(0)-c|^2\big) \\
\ge & |U(0)-c|^2 \big( k  - g |U(0)-c|^{-2}- C \big( k^{1-\alpha} (1+ |U(0)-c|^{-\frac 14}) + |U(0)-c|^{-1} \big)\big) \\
\ge & |U(0)-c|^2 \big( k/2 - C( k^{\frac 98-\alpha} +  k^{\frac 12} )\big).   
\end{align*} 
Again, taking $\alpha =\frac 13$ and $k_0$ sufficiently large determined by $U$, we obtain $F(k, c)\ne0$, either.

The above analysis implies 
\be \label{E:temp-14}
F(k, c) =0, \; k\ge k_0 \implies 
%c \in S_{k, C_1} = 
%S \triangleq \{ c \in \C \mid 
k^{\frac 12} |c- U(0)|\in [C_1 , \sqrt{2g}], 
\ee
which also yields the desired upper bound of $|U(0) - c_R|$.
Let us take the real part of $F(k, c)=0$, namely, 
\[
\big((U(0)-c_R)^2 - c_I^2) Y_R = g + U'(0) (U(0)-c_R) - 2 (U(0)-c_R) c_I Y_I. 
\]
For $k\ge k_0$ and $c$ satisfying \eqref{E:temp-14}, a simple bound of $Y$ can be derived from Lemma \ref{L:Y-def}(3), 
\[
\big|Y(k, c) - k \big| \le C k^{1-\alpha}, \; \alpha \in (0, 1/2)  \implies k \big((U(0)-c_R)^2 - c_I^2) \ge g/2, 
\]
which implies the desired lower bound of $|U(0) - c_R|$. 
Taking the imaginary part of $F(k, c)=0$, we obtain 
\be \label{E:temp-15} \begin{split}
&\big((U(0)-c_R)^2 - c_I^2) Y_I - 2 (U(0)-c_R) c_I Y_R + U'(0) c_I =0 \\
\implies & c_I = \frac {(U(0)-c_R)^2 Y_I(k, c)}{2 (U(0)-c_R) Y_R + c_I Y_I - U'(0)},
%\implies |c_I| \le C k^{-\frac 32} |Y_I(k, c)| \le C k^{-\frac 12-\alpha}.
\end{split} \ee
and thus the estimate on $c_I$ follows from \eqref{E:temp-14}, the above bounds on $|Y- k|$ and $|U(0)-c_R|$ and letting $\alpha = 2/5$ and $k_0\gg1 $. 
\end{proof}

Due to the evenness and conjugacy property \eqref{E:f-conj}, we shall only need to consider 
\[
S_{k+}^{L, R} = \{ c = c_R + c_I \in S_{k}^{L, R} \mid c_I \ge 0 \}.  
\]
Before we proceed to obtain the roots of $F(k, \cdot)$, we first establish some estimates on $Y$ and $\p_c Y$ for $c \in S_{k+}^{L, R}$. 

\begin{lemma} \label{L:Y-2} 
Assume $l_0\ge 5$. There exist $k_0, C>0$ depending only on $U$, such that, for any $|k|\ge k_0$, $c \in S_{k+}^\dagger$ where $\dagger =L, R$, and $0\le j\le l_0-5$, it holds 
\[
|\p_c^j (Y(k, c)- |k|)| \le C|k|^{\frac {j-1}2}. 
\]
\end{lemma}

\begin{proof}
%As in the proof of Lemma \ref{L:e-v-asymp-1}, it is equivalent to work with $F(k, c)$ for $|k|\ge k_0$ and, 
Due to the evenness 
%and conjugacy property \eqref{E:f-conj}, 
we shall focus on $k>0$. We first consider $c \in S_{k+}^L$. 

One may compute, for $c \in U([-h, 0])$, 
\begin{align*}
\p_{c_R} Y_I (k, c) = & \p_{c} \left( \frac {\pi U''(x_2^c)}{U'(x_2^c)}\right) \frac{y_{-} (k, c, x_2^c)^2} {|y_{-} (k, c, 0)|^2} + \frac {2\pi U''(x_2^c) y_{-} (k, c, x_2^c) \p_{c} \big(y_{-} (k, c, x_2^c)\big)}{U'(x_2^c) |y_{-} (k, c, 0)|^2} \\
& - \frac {2\pi U''(x_2^c) y_{-} (k, c, x_2^c)^2 \big(y_-(k, c, 0) \cdot \p_c y_-(k, c, 0)\big)}{U'(x_2^c) |y_{-} (k, c, 0)|^4}, 
\end{align*}
and $\p_{c_R}^j Y_I (c)$ can be computed in a similar fashion,  where we recall $x_2^c= U^{-1}(c)$. 
Using Lemmas \ref{L:y_pm-1} and \ref{L:y0}(8b)(9), we obtain that there exist $C, k_0>0$ depending only on $U$ such that, for $k \ge k_0$ and $0\le j\le l_0-4$, 
\[
|\p_c^j Y_I(k, c)| \le C k^j e^{2 k x_2^c}, \quad  \forall c \in [U(-h) + 2(\rho_0 k)^{-1}, U   (0)- 2(\rho_0 k)^{-1}].  
\]
Here we recall that $\rho_0$ was given in \eqref{E:CIs}. 
Due to the singularity in the integral representation of $Y(k, c)$ given in Lemma \ref{L:Y-def}(7), we use  
 a cut-off function $\gamma\in C^\infty(\R, [0, 1])$ satisfying 
\[
\gamma|_{[\frac 14, 3]} =1, \quad \gamma|_{[\frac 18, 4]^c} =0. 
\]
For $c \in U([-h, 0])$, let 
\[
f_1 (c) = Y_I(k, c) \gamma \big( (k/g)^{\frac 12} (U(0) -c)\big), \quad f_2 (c) = Y_I(k, c) - f_1(c). 
\]
We rewrite Lemma \ref{L:Y-def}(7) for $c \in S_{k+}^{L}$ and $k\ge k_0$, 
\[
Y(k, c)- k\coth kh= Y^1(k, c) + Y^2 (k, c),    
\] 
where 
\[
Y^1 (k, c)  = \frac{1}{\pi} \int_\R \frac{f_1 (c')}{c'-c}dc', \quad Y^2 (k, c) = \frac{1}{\pi} \int_{U(-h)}^{U(0)} \frac{f_2 (c')}{c'-c}dc',
\]
and,  for $c \in U([-h, 0]) \cap S_{k+}^{L}$, $Y^1 (k, c) = \lim_{c_I\to 0+} Y^1 (k, c+i c_I)$ is understood. Using the above estimates on $Y_I$ and the definition of $f_1$, it is straight forward to see, for $j \le l_0-3$, 
\[
|f_1^{(j)}(c)| \le C k^j e^{2 k x_2^c}, \;\text{ if }\; (k/g)^{\frac 12} (U(0) -c) \in [1/8, 4], \; \text{ and } \; f_1(c) =0, \; \text{ otherwise}.
\]
It implies, for all $c \in \C$ with $c_I\ge 0$, 
\[
|\p_c^j Y^1(k, \cdot) |_{L^2 (\R)}^2 \le C |f_1^{(j)}|_{L^2 (\R)}^2 \le C k^{2j} \int_{U(0) - 4 \sqrt{g/k}}^{U(0) -  \frac 18\sqrt{g/k}} e^{4k U^{-1} (c')} dc'  \le C k^{2j-1} e^{-\frac {\sqrt{k}}C}, 
\]
where the substitution $c'= U(x_2)$ was also used in the last step. From interpolation, we obtain, for $j\le l_0-5$, 
\[
|\p_c^j Y^1 (k, \cdot)|_{L^\infty (\R)} \le C k^{j} e^{-\frac {\sqrt{k}}C}, \quad \forall k \ge k_0. 
\]
Concerning $Y^2(k, c)$, we have, for $k\ge 0$, $c \in S_{k+}^{L}$, and $j\ge 0$, 
\begin{align*}
|\p_c^j Y^2 (k, c)| \le & \frac{1}{\pi} \Big(\int_{U(-h)}^{U(0)- 3 \sqrt{g/k}} + \int_{U(0)- \sqrt{g/16 k}}^{U(0)} \Big) \frac{|Y_I (k, c')|}{|c'-c|^{j+1}}dc'  \\
\le & Ck^{\frac {j+1}2} \int_{U(-h)}^{U(0)} e^{2k U^{-1} (c')} dc' \le Ck^{\frac {j-1}2}, 
\end{align*}
where the substitution $c'= U(x_2)$ was used as well. Therefore we obtain the desired estimate on $Y$ in $S_{k+}^{L}$. 

The estimate of $Y(k, c)$ for $c\in S_{k+}^R$ is similar, but simpler. Actually in this case, the norm of denominator $c'-c$ in the integrand of the integral representation in Lemma \ref{L:Y-def}(7) of $Y$ is bounded below by $O(\sqrt{|k|})$. Hence $Y$ can be estimated much as the above $Y^2(k, c)$ and we omit the details. 
\end{proof} 

In the following we analyze possible zeros of $\BF(k, c)=0$ for $k\gg 1$ in the rectangles $S_{k+}^{L, R}$. 

\begin{lemma} \label{L:e-v-asymp-2}
Assume $l_0\ge 6$. There exist $k_0, C>0$ depending only on $U$, such that there exist an analytic function $c^+(k) \in (U(0), +\infty)$ and a $C^{l_0-2}$ function $c^-(k) = c_R^- (k) + ic_I^- (k)$ both defined and even for $|k|\ge k_0$ and the following hold. 
\begin{enumerate}
\item For any $|k| \ge k_0$, 
\[
\BF (k, c) =0, \;\; |k|\ge k_0, \;\; c_I\ge 0 \text{ iff } \; c \in \{c^+(k), \, c^-(k)\}, \; \text{Im}\, c \ge 0.
\]
Namely, for $|k| \ge k_0$, the roots of $\BF(k, \cdot)$ are either only $c^+(k)$ if $c_I^-(k) <0$, or  $\{c^+(k), \, c^-(k),\, \overline {c^-(k)} \}$ if $c_I^-(k) \ge 0$. 
\item  $c^\pm (k)$ satisfy the estimates 
\[
\Big|c^\pm (k) - U(0) +  \frac {U'(0)}{2|k|} \mp \sqrt{ \frac g{|k|}\Big(1 + \frac {U'(0)^2}{4g |k|}} \Big) \Big| \le C k^{-2}, 
\]
\[
\big| c_I^- (k) - \sqrt{g/(4|k|^3)} Y_I(k, c_R^- (k))\big| \le C k^{-2} |Y_I(k, c_R^- (k))|, 
\]
\[
C^{-1} \big|U''(U^{-1} (c_R^-(k) \big)\big|\le |k|^{\frac 32} e^{ \frac {2\sqrt{g|k|}}{U'(0)}} |c_I^- (k)| \le C\big|U''(U^{-1} (c_R^-(k) \big)\big|.
\]
Moreover, the following inequality holds for $c^-(k)$ if $c_I^-(k) \ge 0$ and $c^+(k)$,  
\[
|D_c F (k, c^\pm(k)) \mp 2\sqrt{g|k|} | \le C|k|^{-\frac 12}.
\] 
\end{enumerate}
\end{lemma} 

%\begin{remark}
%While $c^+ (k) >U(0)$, again $c^-(k)$ is relevant only if $c_I^-(k)\ge 0$ which is equivalent to $U''( c_R^-(k)) \ge 0$.
%\end{remark}

\begin{proof} 
Again we only need to consider $F(k, c)=0$ for $k\gg 1$. As in the bifurcation analysis in the proof of Lemma \ref{L:bifurcation}, we consider a $C^{l_0-2}$ extension $\tilde Y(k, c)$ of $Y(k, c)$ from the domain $c\in S_{k+}^{L, R}$ to $S_{k}^{L, R}$, which is different from the original $Y(k, c)$ defined on $S_{k}^{L, R}\setminus S_{k+}^{L, R}$. The extension $\tilde Y(k, c)$ can be defined as, e.g., 
\[
%\tilde Y(k, c) = Y(k, c), \; \text{ if } \; c \in S_{k+}^{L, R}; \quad 
\tilde Y(k, c)= \tilde Y_R(k, c) + i\tilde Y_I(k, c) = \sum_{l=1}^{l_0-1} a_l Y(k, c_R-il c_I), \; \text{ for } \;  \bar c \in S_{k+}^{L, R},
%= 6 Y(k, \bar c) - 8 Y(k, c_R-2ic_I) +3Y(k, c_R-3i c_I).
\]
where $a_1, \ldots, a_{l_0-1}$ can be chosen so that all derivatives up to the order of $l_0-2$ are matched at $c_I=0$. 
Extend $F(k, c)$ to $\tilde F(k, c)$ accordingly and consider $\tilde F(k, c)=0$, namely,  
\be \label{E:temp-16}
\tilde F(k, c) =(U(0)-c)^2 \tilde Y(k, c) - U'(0)(U(0)-c) -g=0, \quad c \in S_{k}^{L, R}. 
\ee
Treating it as a quadratic equation of $U(0)-c$, its two branches of roots must satisfy   
\[
%\be \label{E:temp-17} 
%\begin{split}
c = f_\pm (k, c)\triangleq  U(0) -  \frac {U'(0)}{2\tilde Y(k, c)} \pm \sqrt{\frac {U'(0)^2}{4\tilde Y(k, c)^2} + \frac g{\tilde Y(k, c)}}. 
%\\=&  U(0) -  \frac {U'(0)}{2Y(k, c)} \pm \sqrt{\frac gk} \sqrt{\frac k {Y(k, c)} \Big( 1 +  \frac {U'(0)^2}{4g Y(k, c)} \Big)}.
%\end{split} 
\] 
For $k \ge k_0$, the argument of the last square root is close to $g/|k|$ and the $\pm$ signs correspond to the square roots close to $\pm \sqrt{g/|k|}$. 
From Lemma \ref{L:Y-2} and the definition of the extension $\tilde Y$, there exist some $k_0, C>0$ determined by $U$ such that, for $k>k_0$ and $c \in S_{k}^{L, R}$, it holds 
\[
\Big|\Big(\frac {U'(0)^2}{4\tilde Y(k, c)^2} + \frac g{\tilde Y(k, c)}\Big) - \frac gk\Big(1 + \frac {U'(0)^2}{4g k} \Big) \Big| \le C k^{-\frac 52}. 
\]
The other term $U'(0)/(2\tilde Y(k, c))$ can be handled similarly and thus for such $(k, c)$ we obtain,  
\[
\Big|f_\pm (k, c) - U(0) +  \frac {U'(0)}{2k} \mp \sqrt{ \frac gk\Big(1 + \frac {U'(0)^2}{4g k}} \Big) \Big| \le C k^{-2} \implies f_\pm (k, c) \in S_{k}^{R, L}. 
\]
Moreover, from Lemma \ref{L:Y-2} we also have the derivative estimates of $f_\pm$
\begin{align*}
|D_c f_\pm (k, c)| =& |D_c \tilde Y(k, c)| \Big| \frac {U'(0)}{2\tilde Y(k, c)^2} \mp \Big(\frac {U'(0)^2}{4\tilde Y(k, c)^2} + \frac g{\tilde Y(k, c)}\Big)^{-\frac 12} \Big(\frac {U'(0)^2}{4\tilde Y(k, c)^3} + \frac g{2\tilde Y(k, c)^2}\Big)\Big|\\
\le & C k^{-\frac 32}.
\end{align*}
Therefore, for $k\ge k_0 \gg 1$, $f_- (k, \cdot)$ is a $C^{l_0-2}$ contraction on $S_k^L$ and $f_+ (k, \cdot)$ an analytic contraction on $S_k^R$. Let $c^\pm (k) \in S_k^{R,L}$ denote their unique fixed points, which satisfy the same above leading order asymptotics as $f_\pm$ and are $C^{l_0-2}$ in $k$.

It is straight forward to compute 
\[
D_c \tilde F(k, c) =2 (c- U(0)) \tilde Y(k, c) + (U(0)-c)^2 D_c \tilde Y(k, c) + U'(0).
\]
From Lemma \ref{L:Y-2} and the leading order estimates of $c^\pm (k)$ which yields 
\[
|c^\pm (k) - U(0) \mp \sqrt{g/k} + U'(0)/(2k)| \le Ck^{-\frac 32}, 
\]
we obtain the estimate on $ D_c \tilde F(k, c)$
\[
|D_c \tilde F (k, c^\pm(k)) \mp 2\sqrt{gk} | \le Ck^{-\frac 12}.
\] 

To complete the proof, we consider the imaginary parts of $c^\pm (k)$. Using \eqref{E:temp-16} we have  
\[
\big((U(0)-c_R)^2 - c_I^2) \tilde Y_I (k, c) - 2 (U(0)-c_R) c_I \tilde Y_R (k, c) + U'(0) c_I =0, \; \text{ at }\; c= c^\pm (k). 
\]
From the Mean Value Theorem, there exists $\tau^\pm$ between $c_R^\pm (k)$ and  $c^\pm (k)$ such that  
\[
\tilde Y_I(k, c^\pm (k)) - Y_I(k, c_R^\pm (k)) = c_I^\pm (k) \p_{c_I} \tilde Y_I (k, \tau^\pm), 
\]
which, along with Lemma \ref{L:Y-2}, yields, at $c= c^\pm (k)$, 
\[
(U(0)-c_R)^2 \big( Y_I(k, c_R) + c_I \p_{c_I} \tilde Y_I (k, \tau^\pm) \big)  - c_I^2 \tilde Y_I (k, c) - 2 (U(0)-c_R) c_I \tilde Y_R(k, c) + U'(0) c_I =0. 
\]
Therefore we obtain 
\[
c_I^\pm (k) = \frac {(U(0)-c_R)^2 Y_I(k, c_R)}{(U(0)-c_R) \big( 2\tilde Y_R(k, c) - (U(0)-c_R)\p_{c_I} \tilde Y_I (k, \tau^\pm) \big) - U'(0) + c_I \tilde Y_I (k, c)  }\Big|_{c=c^\pm (k)}, 
\]
which, along with Lemma \ref{L:Y-2} and the asymptotics of $c^\pm(k)$, implies the desired estimates on $c_I^\pm (k)$ in term of $Y_I (k, c_R^\pm (k))$. In particular, since $c_R^+ (k) > U(0)$, we have $c_I^+ (k)=0$ and $c^+ (k)> U(0)$. As $c^+(k)$ belongs to the domain of analyticity of $F(k, c)$, it is also an analytic function of $k$.
Finally, the upper and lower bounds of $c_I^-(k)$ follow from Lemmas \ref{L:Y-def}(5) and \ref{L:y_pm-1} and $U^{-1} (c_R^- (k)) = - \sqrt{g} /(U'(0) \sqrt{k}) + O(k^{-1})$. 
\end{proof}

\subsection{Eigenvalue distribution} \label{SS:e-values}

With the above preparations, we are ready to prove the main theorems on the eigenvalue distribution of the linearized gravity water waves. \\

\noindent $\bullet$ {\it Proof of Theorem \ref{T:e-values}.}
Throughout the proof, we recall that $\BF(k, c)$ and $F(k, c)$ have the same zero sets (Lemma \ref{L:e-v-basic-1}(2)) and we often mix them in the arguments.  

Most of statement (1) has been proved in Lemma \ref{L:e-v-asymp-2} except for 
%the semi-simplicity of the eigenvalues $-ik c^\pm (k)$ and 
the global extension {\it etc.} of $c^+(k)$ in (1c) which will be proved here. 
%While the former will be proved in the analysis of the linear solutions (???) based on their simplicity $\p_c F(k, c^\pm (k)) \ne 0$ as roots of $F(k, \cdot)$, we prove statement (1c) here. 

Since $\p_c F (0, c) >0$ for $c \in [U(0), +\infty)$ which can be verified directly using its integral formula given in Lemma \ref{L:e-v-basic-2}(1), there exists a unique $c_1 > U(0)$ such that 
\be \label{E:c1}
F(0, c_1) = 
%\frac 1{\int_{-h}^0 \frac 1{ (U-c)^{2}} dx_2} -g =
0, \; \p_c F(0, c_1) >0, \; \text{ and } \; F(0, c) (c- c_1) >0,  \; c\in [U(0), +\infty)\setminus \{c_1\}. 
\ee
Let  
\[
\Omega_1 = \{ c\in \C \mid c_R \in ( U(0), c_1+1), \, c_I \in (-1, 1)\}. 
\]
Because $F(k, c)$ is also strictly increasing in $k\ge 0$ for any $c\ge U(0)$ (Lemma \ref{L:e-v-basic-2}(2)), we have $F(k, c) >0$ for all $c > c_1$ and $k \in \R$. Therefore, according to Lemma \ref{L:e-v-basic-2}(1) and the semicircle theorem \cite{Yih72}, $\BF(k, c)\ne 0$ for any $k \in \R$ and $c \in \p \Omega_1$ and $c_1$ is the only root of $F(0, \cdot)$ in $\Omega_1$, which is also simple. It implies Ind$\big(\BF(k, \cdot), \Omega_1\big) =1$ for all $k \in \R$. On the one hand, from Lemma \ref{L:continuation} and Remark \ref{R:continuation}, the unique root $c(k)$ of $\BF(k, \cdot)$ in $\Omega_1$ is simple and depends on $k\in \R$ analytically and evenly. On the other hand, Lemma \ref{L:e-v-asymp-2} implies that the root $c^+(k) \in \Omega_1$ for $|k| \ge k_0$. Therefore $c^+(k)$ coincides with $c(k)$ for $|k|\ge k_0$ and thus $c(k)$ serves as its extension for all $k \in \R$ as the only root of $\BF(k, \cdot)$ in $\Omega_1$, which is simple. It is also the only root of $\BF(k, \cdot)$ in $(U(0), +\infty)$ since $F(k, c) >0$ for all $c > c_1$ and $k \in \R$ obtained in the above. Finally, since $\p_c F(k, c_1) >0$, $\p_c F(k, c(k))$ does not change sign as $c(k)$ is a simple root for all $k$, and $\p_K F(k, c^+(k)) >0$ where $K=k^2$ (Lemma \ref{L:e-v-basic-2}(2)), we obtain $c'(k)<0$ for $k> 0$. This completes the proof of statement (1c). 

In statement (2), the existence and uniqueness of $k_-$ has been obtained in Lemma \ref{L:U(-h)}. 
From Lemma \ref{L:e-v-basic-2}, there exist $c_2 <U(-h)$ 
%and  $k_->0$ 
unique in $(-\infty,  U(-h)]$, 
% and $[0, +\infty)$, respectively, 
such that 
\be \label{E:c2} \begin{split}
&F(0, c_2) =0= F(k_-, U(-h)), \; \p_c F(0, c_2)<0, \\ 
& F(0, c) (c- c_2) >0,  \ \forall c\in (-\infty, U(-h)]\setminus \{c_2\}.
\end{split} \ee
%It is straight forward to verify $\p_c F(0, c_2) <0$ directly. 
Due to the semi-circle theorem, $F(0, \cdot)$ has exactly two roots $c_2$ and $c_1=c^+(0)$ in $\Omega_2$, both of which are  simple, and $c_2$ the only one in $\Omega_3$, where 
\[
\Omega_2 = \{ c\in \C \mid |c - (U(-h)+U(0))/2| > |U(0)-U(-h)|/2\}, 
\]
\[
\Omega_3 = \{ c\in \C \mid c_R \in (c_2-1, U(-h)), \, c_I \in (-1, 1)\}.
\]
The defintions of $c_2$ as well as that of $c_1$, the monotonicity of $F(k, c)$ in $k>0$ for $c \in \R \setminus U((-h, 0])$, the local monotonicity of $F(0, c)$ in $c \in \R \setminus U((-h, 0])$, and the semi-circle theorem also imply that a.) for any $k \in \R$, the roots of $\BF(k, c)$ in $\Omega_2$ have to belong to $\Omega_1 \cup \Omega_3$ and b.) the only root of $\BF(k, c)$ for $c \in \p \Omega_1 \cup \p \Omega_3$ is $(k_-, U(-h))$. Consequently,  the total number of roots (counting the multiplicity) of $\BF(k, \cdot)$ in $\Omega_2$ can change only at $k=k_-$. 
From the same argument as in the above proof of statement (1c), 
we obtain that, for $|k|< k_-$, the only root of $\BF(k, \cdot)$ in $\Omega_3$, which is also the only one in $(-\infty, U(-h))$, is given by some $c_-(k) \in [c_2, U(-h))$. It is analytic and even in $k$ and satisfies 
\[
c_-(0)=c_2,\; \p_c F(k, c_-(k))<0, \; c_-'(k) >0, \;  F(k, c)(c_-(k)-c) >0, \; \forall |k| < k_-,  \, c \le U(-h).
\]
Here the last property is due to the monotonicity of $F(k, c)$ in $k>0$ for $c \in \R \setminus U((-h, 0])$ (Lemma \ref{L:e-v-basic-2}(2)) and that of $c_-(k)$. 
Hence $c^+(k)$ and $c_-(k)$ are the only roots of $\BF(k, \cdot)$ in $\Omega_2$ for all $|k| < k_-$. 
Moreover, for $|k| > k_0$, the only root of $\BF(k, \cdot)$ in $\Omega_2$ (and $\Omega_1 \cup \Omega_3$ as well) is $c^+(k)$ implies that this also holds for all $|k| > k_-$.  
The limit $c_-((k_-) -)$ exists due to its boundedness and monotonicity. It has to be equal to $U(-h)$, otherwise, if $c_-((k_-) -) < U(-h)$, by Lemma \ref{L:continuation} $c_-(k)$ could be continued beyond into $k > k_-$ and then $\BF(k, \cdot)$ would have at least two roots $c^+(k)$ and $c_-(k)$ in $\Omega_2$ for $1\gg k-k_- >0$, which contradicts the above analysis. Finally, statement (2c) follows from Lemma \ref{L:bifurcation}(1) and the signs $\p_k F>0$ and $\p_c F<0$ at $(k_-, U(-h))$ (Lemma \ref{L:pcF}).   

Statement (3) regarding $c$ near the inflection values of $U$ has been proved in Lemmas \ref{L:neutral-M-3} and \ref{L:bifurcation}, see also Remark \ref{R:bifurcation}. 
%The bifurcation curve is $C^4$ because the Implicit Function Theorem yields a $C^{l_0-3}$ solution curve when $U \in C^{l_0}$ and $F \in C^{l_0-3}$.  

Statement (4) on the linear instability is a direct corollary of statements (1)--(3). 
\hfill $\square$ \\

We proceed to prove Theorem \ref{T:e-values-1} assuming either $U'' > 0$ or $U''\le 0$ on $(-h, 0)$. \\

\noindent $\bullet$ {\it Proof of Theorem \ref{T:e-values-1}.}
Suppose $U''>0$ on $(-h, 0)$, then Theorem \ref{T:e-values}(2) and Lemmas \ref{L:y0}(4) and \ref{L:e-v-basic-1}(4) imply that $F(k, c) = 0$ with $c \in U([-h, 0])$ iff $(k, c) =(k_-, U(-h))$. 
%Let $k_-$ be given as in the above statement (2) satisfying $F(k_-, U(-h))=0$. 
From Lemma \ref{L:e-v-asymp-2} (or Theorem \ref{T:e-values}(1)), there exists $k_0> k_-$ such that, for any $|k|> k_0$, there are exactly three roots of $F(k, \cdot)$, which are simple and given by $c^+(k) \in (U(0), c^+(0)]$ (already extended for all $k \in \R$ in Theorem \ref{T:e-values}(1)), $c^-(k) = c_R^-(k) + i c_I^-(k)$, and $\overline{c^-(k)}$, where $c_I^-(k) >0$ due to $U''>0$. 
For any $k_1 \in (k_-, k_0]$, $\BF(k, c) \ne 0$ for all $k \in [k_1, k_0]$ and $c \in U([-h, 0])$.   
%due to the uniqueness of $k_-$ and Lemma \ref{L:e-v-basic-1}(3)(5). Therefore 
The conjugacy property \eqref{E:BF-conj} and the continuity of $\BF$ (restricted to $c_I\ge 0$) imply that there exists $\delta_0>0$ such that $\BF(k, c) \ne 0$ if $k \in [k_1, k_0]$ and $dist\big(c, U([-h, 0])\big)\le \delta_0$. For any $\delta \in (0, \delta_0)$, let 
\[
\Omega_4 = \big\{ c \in \C \mid \big|c -\tfrac 12 (U(-h)+U(0))\big)\big| < \delta + \tfrac 12 (U(0)- U(-h)), \ dist\big(c, U([-h, 0])\big)> \delta  \big\}, 
\] 
and $\Omega_4^\pm$ be its (disjoint) upper and lower connected components with $\pm c_I>0$. 
Clearly $c^-(k_0) \in \Omega_4^+$, $\overline{c^-(k_0)} \in \Omega_4^-$, $c^+(k) \notin \Omega_4$ and $\BF(k, c) \ne 0$ for all $k \in [k_1, k_0]$ and $c \in \p \Omega_4= \p \Omega_4^+\cup \p \Omega_4^-$ due to the semi-circle theorem and the choice of $\delta$. Therefore 
\[
\text{Ind}(\BF (k, \cdot), \Omega_4^\pm) =\text{Ind}(\BF (k_0, \cdot), \Omega_4^\pm) =1, \quad \forall k \in [k_1, k_0]. 
\]
%We also observe that $c^-(k_0)$ and $\overline{c^-(k_0)}$ belong to the two (upper and lower) connected components of $\Omega_4$. 
From Lemma \ref{L:continuation} and Remark \ref{R:continuation}, the simple root $c^-(k_0)$ of $\BF(k_0, \cdot)$ can be extended as a simple root $c^-(k) \in \Omega_4^+$ of $\BF(k, \cdot)$  for all $k \in [k_1, k_0]$, which is the only root of $\BF(k, \cdot)$ in $\Omega_4^+$. Hence, by taking all $k_1 \in (k_-, k_0]$, $c^-(k)$ can be extended for all $k \in (k_-, +\infty)$ as the only root (counting the multiplicity) of $\BF(k, \cdot)$ in $\Omega_2^c \cap \{ c_I>0\}$. From the semi-circle theorem and Theorem \ref{T:e-values}(1b), $c^\pm(k)$ and $\overline{c^-(k)}$ are the only roots of $\BF(k, \cdot)$ for $|k| > k_-$, which are also simple.  

Recall the branch $c_-(k)$ of the root of $\BF(k, \cdot)$ for $k\in [-k_- -\ep, k_-+ \ep]$ obtained in  Theorem \ref{T:e-values}(2). The assumption $U''>0$ yields $\IP\, c_-(k_- +\ep)>0$. Choose $\delta >0 $ sufficiently small so that $\BF(k, c) \ne 0$ for all $k \in [k_-+\ep, k_0]$ and $dist\big(c, U([-h, 0])\big)\le \delta$. Hence the semi-circle theorem implies $c_-(k_- +\ep) \in \Omega_4^+$ where $\Omega_4^\pm$ is defined in the same form as in the above. While Ind$(\BF (k_-+\ep, \cdot), \Omega_4^+) =1$, both $c^-(k_-+\ep)$ and $c_-(k_-+\ep)$ are roots of $\BF (k_-+\ep, \cdot)$ in $\Omega_4^+$. Therefore we obtain $c^-(k) = c_- (k)$ for $k > k_-$. 

A similar argument based on the index, continuation, and the Semi-circle Theorem starting at $k=k_-+\ep$ also yields that the roots of $\BF(k, \cdot)$ for $|k-k_-| \ll1$ are either $c^+(k)$ or close to $U(-h)$. Hence Theorem \ref{T:e-values}(2c) implies that $c^+(k), c_-(k) \in \R$ are all the roots for $k\in [k_--\ep, k_-]$. There exists $\beta_0 >0$ such that $\BF(k, c) \ne 0$ for all $|k| \le k_--\ep$ and $dist\big(c, U([-h, 0])\big) \le \beta_0$. Hence $\BF(k, c) \ne 0$ for all $\beta \in (0, \beta_0)$, $|k| \le k_- - \ep$, and $c \in \p \Omega_5$  where 
\[
\Omega_5 =\{ c\in \C \mid c_R \in (c_2-1, c_1+1), \, |c_I | \le U(0)-U(-h), \, dist\big(c, U([-h, 0])\big) > \beta \},
\]
and $c_1$ and $c_2$ have been given in \eqref{E:c1} and \eqref{E:c2}, respectively. 
Clearly, for all $|k| \le k_--\ep$, Ind$\big(\BF(k, \cdot), \Omega_5)=$Ind$\big(\BF(k_--\ep, \cdot), \Omega_5)=2$ and the two roots are $c^+ (k)$ and $c_-(k)$. For $|k| \le k_--\ep$, $\BF(k, \cdot)$ has no other roots outside $\Omega_5$ due to the semi-circle theorem and the choice of $\beta$ and $c_{1,2}$. Therefore $c^\pm(k)$ are also the only roots of $\BF(k, \cdot)$ for all $|k| \le k_-$. It completes the proof of the case of $U''>0$.  

Assume $U''\le 0$ on $(-h, 0)$. Firstly, we show the spectral stability following the most standard technique, which had actually been obtained in \cite{Yih72} (see also \cite{HL08}). Suppose $(k, c)$ is an unstable mode (i.e. $c_I>0$) with the eigenfunction $y_-(x_2)$. Multiplying the Rayleigh equation by $\overline {y_- (x_2)}$, integrating by parts over $(-h, 0)$, taking the imaginary part, and using the boundary condition of \eqref{E:Ray-3}, we have 
\begin{align*}
c_I \int_{-h}^0 \frac {U'' |y_-|^2}{|U-c|^2} dx_2 =& \IP \big( y_-'(0) \overline {y_- (0)} \big) =  \IP \Big( \frac {U'(0)(U(0)-c) +g}{(U(0)-c)^2} |y_-(0)|^2 \Big)\\
= &c_I  \Big( \frac {U'(0)}{|U(0)-c|^2} +\frac {2g (U(0) - c_R)}{|U(0)-c|^4} \Big) |y_-(0)|^2. 
\end{align*} 
The semi-circle theorem implies $c_R \in U([-h, 0])$ and thus Lemma \ref{L:e-v-basic-1}(2) yields $c_I=0$  which is a contradiction. Hence the spectral stability follows. 
Due to Lemmas \ref{L:e-v-basic-1}(2)(4) and \ref{L:e-v-basic-2}(1), the only singular or non-singular modes inside the circle \eqref{E:semi-circle} have to be at inflection values of $U$. The remaining statements in this case has been proved in Theorem \ref{T:e-values}.
%(2b). 
\hfill $\square$\\

\begin{remark} \label{R:stability}
Suppose,  instead of $U'>0$ on $[-h, 0]$, only $U'(0)>0$ is assumed along with $U''\le 0$  on $[-h, 0]$. It is easy to see that Lemma \ref{L:e-v-basic-1}(2) still holds for any $c \in \C \setminus U([-h, 0])$. The same above argument still applies to imply the spectral stability as already proved in \cite{Yih72}. 
\end{remark}

Finally we prove Theorem \ref{T:e-values-2} assuming $U$ has exactly one non-degenerate inflection point $x_{20}$ on $(-h, 0)$. \\

\noindent $\bullet$ {\it Proof of Theorem \ref{T:e-values-2}.}
Suppose \eqref{E:K+} holds with $U''(x_{20})=0$ and $c_0= U(x_{20})$. The same type of argument  as above based on the index and the continuation will be employed extensively below on the semi-disk domain 
\[
\Omega_6 = \big\{ c \in \C \mid \big(c_R-\tfrac 12 (U(0) +U(-h))\big)^2+c_I^2< \tfrac 14 (U(0)- U(-h))^2, \ c_I>0\}. 
\]
%and  we shall mainly outline it. 

%Finally we prove statement (7) assuming $U''$ changes the the sign exactly once in $(-h, 0)$ at $x_{20}$. 
Let us start with the case of $U'''(x_{20}) > 0$ which implies 
%$U''>0$ on $(-h, x_{20})$,  $U''<0$ on $(x_{20}, 0)$, and 
$\frac {U''}{U-c_0} >0$ on $[-h, 0]$. Lemma \ref{L:y0} yields $y_-(k, c_0, x_2) \in \R$ over $[-h, 0]$ for any $k \in \R$. Multiplying the Rayleigh equation by $y_-$, integrating by parts over $(-h, 0)$, and using $\frac {U''}{U-c_0} >0$, we obtain $y_- (k, c_0, 0) >0$ for all $k \in \R$. According to Lemma 
%\ref{L:neutral-M-1} and 
\ref{L:neutral-M-3}, there exists $k_0 >0$, unique among $k \ge 0$, such that $\BF(k_0, c_0) =0$. 
%with a positive eigenfunction $y_-(k_0, c_0, x_2)$. 
Along with the Semi-circle Theorem\footnote{From the standard proof of the Semi-circle Theorem, one could see that unstable modes can not occur on the boundary semi-circle of \eqref{E:semi-circle} when $U'>0$ on $[-h, 0]$.}, Lemmas \ref{L:e-v-basic-1}(4) and \ref{L:U(-h)}, it implies that $(\pm k_-, U(-h))$ and $(\pm k_0, c_0)$  are the only roots of $\BF$ with $c \in \p \Omega_6$. On the one hand, since $U''<0$ near $x_2= -h$, Lemma \ref{L:bifurcation} implies that there are never singular or non-singular modes with $c \in \Omega_6$  near $U(-h)$ for any $k \in \R$.
% except $(\pm k_-, U(-h))$. 
Hence the index Ind$(\BF(k, \cdot), \Omega_6)$ remains a constant for $k > k_0$. On the other hand, since $U''>0$ near $x_2=0$, Lemma \ref{L:e-v-asymp-2} implies that $(k, c^-(k))$ with $c_I^-(k) >0$ is the only unstable mode for $k\gg 1$, which implies Ind$(\BF(k, \cdot), \Omega_6)=1$ for all $k > k_0$. Again from Lemma \ref{L:bifurcation}, there exists a branch of unstable modes $(k, \CC(k))$ with $\CC(k_0)=c_0$ and $\CC_I(k) >0$ for $1\gg k- k_0 >0$, which are the only singular or non-singular modes   near $(k_0, c_0)$ with $c \in \Omega_6$. From Lemma \ref{L:continuation}, this branch $(k, \CC(k))$ can be continued for all $k\ge k_0$. Since $\BF$ has exactly one root with $c \in \Omega_6$ for $k > k_0$, we obtain that the continuations of $\CC(k)$ and $c^-(k)$ must coincide and form the only singular or non-singular modes with $c \in \overline{\Omega_6}$. The local bifurcation given by Lemma \ref{L:bifurcation} implies that Ind$(\BF(k, \cdot), \Omega_6) =0$ for $|k| < k_0$ and thus there are no unstable modes if $|k| < k_0$. 
Along with Theorem \ref{T:e-values}(2b), 
it completes the proof of this case. 

In the rest of the proof we assume $U'''(x_{20}) <0$. Lemma \ref{L:neutral-M-2} yields $k_0 > k_-$. Much as the above,  the index Ind$(\BF(k, \cdot), \Omega_6)$ remains a constant except for $k \in \{\pm k_-, \pm k_0, \pm k_1\}$ (no $k_1$ if it does not exist). Note \eqref{E:K+} and $U''' (x_{20})<0$ implies $U''<0$ on $(x_{20}, 0)$ and $U''>0$ on $(-h, 0)$. From Lemmas \ref{L:e-v-asymp-2} 
\be \label{E:ind-1}
\text{Ind}(\BF(k, \cdot), \Omega_6) = 0, \quad \text{ if } \;  |k| > k_0.
\ee 
By a similar argument to that in the proof of Theorem \ref{T:e-values-1}, 
Lemma \ref{L:bifurcation}, which gives all the unstable modes near the singular neutral modes at $c_0$ and $U(-h)$, implies that the index increases by 1 as $k$ increases through $k_-$ and decreases by 1 as $k$ increases through $k_0$, as well as through $k_1$ if $k_1$ exists. Therefore we have 
\be \label{E:ind-2}
\text{Ind}(\BF(k, \cdot), \Omega_6) =1,  \quad \text{ if } \;  |k| \in (\max\{k_-, k_1\}, k_0),  
\ee
where $\max\{k_-, k_1\}= k_-$ is understood if $k_1$ does not exist. In the case $k_1$ does not exist, we also have 
\be \label{E:ind-3}
\text{Ind}(\BF(k, \cdot), \Omega_6) = 0, \qquad |k| < k_-,
\ee 
and otherwise 
\be \label{E:ind-4}
\text{Ind}(\BF(k, \cdot), \Omega_6) = \begin{cases} 
2, \quad & |k| \in (k_-,  k_1), \\
0, \quad & |k| \in (k_1,  k_-), \\
1, & |k| < \min\{k_-, k_1\}. 
\end{cases}\ee
The linearized system \eqref{E:LEuler} is unstable when the above index is greater than zero. 
By a similar argument to that in the proof of Theorem \ref{T:e-values-1} bases on Lemma \ref{L:continuation} and the index, the branches of unstable modes bifurcating from the singular neutral modes can be continued in $k$ unless $k$ approaches $\pm k_-$, $\pm k_0$, and $\pm k_1$ or it  collides with another branch, the latter of which happens only possibly for $|k| \in (k_1, k_-)$. Let us consider the different cases (2)--(4) in the theorem separately. 

Suppose $k_1$ does not exist and thus $\pm k_0$ are the only wave numbers which make $c_0$ a singular neutral mode. From Lemma \ref{L:continuation} and \eqref{E:ind-2}, the branch of unstable modes $c_-(k)$ obtained in Theorem \ref{T:e-values}(2) (or Lemma \ref{L:bifurcation} more directly bifurcating from $(\pm k_-, U(-h))$) can be continued for all $|k| < k_0$. 
%Since Ind$(\BF(k, \cdot), \Omega_6)=1$ for all $|k| \in (k_-, k_0)$, 
Due to \eqref{E:ind-2}, this branch has to coincide with the branch $\CC(k)$ bifurcating from $(\pm k_0, c_0)$ (again by Lemma \ref{L:bifurcation})  for $|k| \le k_0$ and with $c_-(\pm k_0)=c_0$. From the Semi-circle Theorem and \eqref{E:ind-3}, 
%for any $k$ with $|k| \notin (k_-, k_0)$
$(k, c_-(k))$ is the only  singular or non-singular modes with $c \in \overline{\Omega_6}$ and statement (2) is proved. 

Suppose $k_1$ exists and $k_1 \le k_-$. By continuity, the branch $c_-(k)$ satisfies $c_-(k) \ne c_0$ for $1\gg |k| - k_- >0$. The same argument as above also yields that the branch $c_-(k)$ connects to $c_0= c_-(\pm k_0)$ and is the only singular or non-singular modes with $c \in \overline{\Omega_6}$ and $|k| \ge k_-$. Due to \eqref{E:ind-4}, there are no singular or non-singular modes with $c \in \overline{\Omega_6}$ 
and $|k| \in (k_1, k_-)$. From Lemmas \ref{L:bifurcation} and \ref{L:continuation}, both branches of unstable modes bifurcating from $(\pm k_1, c_0)$ for $|k| < k_1$ can be continued for all $|k| < k_1$ and have to coincide due to  \eqref{E:ind-4}. This branch is apparently even in $k$ and the only singular or non-singular modes with $c \in \overline{\Omega_6}$ and $|k| < k_1$. It completes the proof of statement (3). 

To prove the last case, suppose $k_1$ exists and $k_1 > k_-$. From the Semi-circle Theorem and the assumptions, it is clear 
\[
\CS = \{(k, c) \mid c \in  \overline{\Omega_6}, \ |k| \le k_0, \  \BF(k, c) =0\}, \quad \overline{\CS} = \CS.
\]
%If $\CS$ is connected, then statement (4) holds simply for $\CS_1=\CS_2=\CS$. Hence let us consider the case where $\CS$ is not connected. 

$\bullet$ {\it Claim.} For any connected component  $\tilde\CS$ of $\CS$, it holds 
\[
k^* = \sup \{ k \mid \exists c \in \overline{\Omega_6}, \ (k, c) \in \tilde \CS\}   \in \{- k_-, k_1, k_0\} \;\; \& \;\; \exists c^* \in \{c_0, U(-h)\}, 
%\Omega_6 \ \text{ s. t. }
\ (k^*, c^*) \in \tilde \CS,
\]
\[
k_* = \inf \{ k \mid \exists c \in \overline{\Omega_6}, \ (k, c) \in \tilde \CS\}   \in \{ k_-, - k_1, -k_0\} \;\; \& %\;\; \exists c_* \in 
%\overline{\Omega_6} \ \text{ s. t. } (k_*, c_*) \in \tilde \CS. 
\;\; \exists c_* \in \{c_0, U(-h)\}, 
%\Omega_6 \ \text{ s. t. }
\ (k_*, c_*) \in \tilde \CS.
\]
We shall prove the infimum case and the supremum case is similar.  
In fact, since the compactness of $\CS$ yields that of $\tilde \CS$, there exists $c_* \in \overline{\Omega_6}$ such that $(k_*, c_*) \in \tilde \CS$. If $c_* \in \Omega_6$ where $\BF$ is analytic, then $c_*$ is an isolated zero of $\BF(k_*, \cdot)$. Since, as $k$ varies locally, the set of zero points of analytic functions can always be extended continuously beyond $k_*$ (even though splitting may occur if $c_*$ is not simple), it contradicts with the definition of $k_*$. Therefore $c_* \in \p \Omega_6$ and thus $k_* \in \{ \pm k_-, \pm k_1, \pm k_0\}$. 
%Clearly $(k_0, c_0)$ is the only element in $\CS$ with $k=k_0$, but 
Lemma \ref{L:bifurcation} implies that $(k_0, c_0)$ is the end point of a branch of unstable modes with $k< k_0$, therefore $k_0$ can not be the infimum. Another two possibilities $(k_*, c_*) = (k_1, c_0)$ or $(k_*, c_*) = (-k_-, U(-h))$  can be excluded in the same way as $(k_0, c_0)$. The claim is proved. 

Let $\CS_1$ be the connected component of $\CS$ containing $(k_-, U(-h))$ and $\CS_{-1} = \{(-k, c) \mid (k, c) \in \CS_1\}$. From \eqref{E:ind-4}, 
\be \label{E:middle-p}
\exists | c^0 \in \Omega_6  \; \text{ s. t. } \; (0, c^0) \in \CS.
\ee
%which we refer to as the middle point of $\CS$. 
Applying the above claim to the supremum to $\CS_1$, there are two possibilities. 

$\#$ {\it Case A.} $k_1 = \sup_{(k, c) \in \CS_1} k$ and thus $(k_1, c_0) \in \CS_1$. In this case, let $\CS_0$ be the connected component of $\CS$ containing $(k_0, c_0)$. As $k_1$ is the supremum of $\CS_1$ and $k_1 <k_0$, we have $\CS_0 \cap \CS_1 = \emptyset$. This implies that $\inf_{(k, c) \in \CS_0} k \ne k_-$, otherwise $(k_-, U(-h)) \in \CS_0$ which would yield $\CS_0 = \CS_1$. Again the above claim implies either $(-k_1, c_0)\in \CS_0$ or $(-k_0, c_0) \in \CS_0$. From 
%\eqref{E:ind-4}, 
\eqref{E:middle-p} and the connectedness of $\CS_0$, 
%there exists $c^0 \in \Omega_6$ such that 
it holds $(0, c^0) \in \CS_0$. By the even symmetry of $\BF$ in $k$, immediately we derive the same even symmetry of $\CS_0$ about $k$ and thus $(-k_0, c_0) \in \CS_0$. 
%Let $\CS_{-1} = \{(-k, c) \mid (k, c) \in \CS_1\}$. 
Since all points $(k, c)\in \CS$ with $c \in \{c_0, U(-h)\}$ belong to $\CS_{-1} \cup \CS_0 \cup \CS_1$, it follows from the above claim that these three sets are all the connected component of $\CS$ and thus $\CS =\CS_{-1} \cup \CS_0 \cup \CS_1$.  The desired statement (4) in case (4b) is obtained. 

$\#$  {\it Case B.} $k_0 = \sup_{(k, c) \in \CS_1} k$ and thus $(k_0, c_0) \in \CS_1$.  Accordingly we let $\CS_0$ be the connected component of $\CS$ containing $(k_1, c_0)$. Again we consider the two possibilities separately.

* {Case B.1.} $(k_1, c_0) \in \CS_1$, too. In this case, we have $\CS_0 = \CS_1$. Again since all points $(k, c)\in \CS$ with $c \in \{c_0, U(-h)\}$ belong to $\CS_{-1}  \cup \CS_1$, it follows from the above claim $(0, c^0) \in \CS =\CS_{-1} \cup \CS_1$. Due to the symmetry, we obtain $(0, c^0) \in \CS_1 \cap \CS_{-1}$ and thus $\CS=\CS_{-1} = \CS_0 = \CS_1$. Hence the desired statement (4) in case (4b) holds. 

* {Case B.2.} $(k_1, c_0) \notin \CS_1$, which yields $\CS_0 \cap \CS_1 =\emptyset$. The above claim implies $\inf_{(k, c) \in \CS_0} k < k_-$ and thus $(-k_1, c_0)\in \CS_0$ or $(-k_0, c_0) \in \CS_0$. The desired statement (4) in case (4a) follows from the same argument as in the above Case A. The proof of the theorem is complete. 
\hfill $\square$\\

According to Lemma \ref{L:neutral-M-3}, if $k_C>0$ exists, i.e. $U$ is unstable for the channel flow, then $k_1>0$ exists if $g> F^0 (c_0)$ defined in \eqref{E:F0}. From Lemma \ref{L:bifurcation-P} which is  under a weaker condition, the closed branch $\CC(k)$ of unstable modes in case (3) bifurcates from $c_0$. As this branch grows, we can not exclude the possibility that it intersects with the branch $c_-(k)$ connecting $(k_-, U(-h))$ and $(k_0, c_0)$ obtained in case (2).

\subsection{The linearization of the capillary gravity water waves at monotonic shear flow revisited} \label{SS:CGWW} 

When the surface tension is also considered, the boundary condition \eqref{E:Euler-4} is replaced by 
\be \label{E:CG-BC}
p(t, x) =  \sigma \kappa (t,x), \qquad x\in S_t,
\ee
where $\sigma > 0$, $\kappa(t,x)=-\frac{\eta_{x_1x_1}}{(1+\eta_{x_1}^2)^{\frac{3}{2}}}$ is the mean curvature of $S_t$ at $x$. Shear flow $v(x) =( U(x_2), 0)^T$ under the flat surface $S_t = \{x_2=0\}$ is also a stationary solution. The linearization of the capillary gravity water waves at such a shear flow is given by a system similar to \eqref{E:LEuler} only with the linearized boundary condition \eqref{E:LEuler-5} replaced accordingly by 
\be \label{E:LCG-BC}
p(t, x) =  g \eta - \sigma \p_{x_2}^2 \eta, \qquad \; \text{ at } x_2 =0. 
\ee
Following the same procedure, 
%Fourier transform in $x_1$ and the Laplace transform in $t$, 
we obtain that $(k, c)$ corresponds to an eigenvalue $-ikc$ in the $k$-th Fourier modes of $x_1$ if and only if $\CF_\sigma (k, c)=0$, or equivalently, $\BF_\sigma (k, c)=0$, where $\CF_\sigma$ was defined \eqref{E:dispersion-CG} and $\BF_\sigma (k, c) = y_-(k, c, 0) \CF_\sigma (k, c)$.  

In \cite{LiuZ21}, we obtained results on the eigenvalue distribution and the inviscid damping of the linearized capillary gravity waves. 
In particular, the semi-circle theorem still holds for unstable modes, namely, all non-singular modes $c$ outside the circle \eqref{E:semi-circle} belong to $\R\setminus U((-h, 0))$. Moreover, there exists $k_0>0$ such that $\BF_\sigma (k, \cdot)$ has exactly two roots $c^\pm (k)$ for $|k| \ge k_0$, which are simple roots, even and analytic in $k$, and satisfy 
\[
c^+(k) > U(0), \quad c^- (k) < U(-h), \quad \lim_{|k| \to \infty}  c^\pm (k)/\sqrt{\sigma |k|} =\pm 1.
\] 
The branch $c^+(k)$ can be continued as simple roots for each $k \in \R$ with $c^+(k) > U(0)$ and even and analytic in $k$. The continuation of the other branch $c^-(k)$ may or may not reach $U(-h)$ depending on whether 
\begin{align*}
g < g_\# (\sigma) \triangleq & \max \big\{ Y\big(k, U(-h)\big) \big(U(0)-U(-h)\big)^2 - U'(0) \big(U(0)-U(-h)\big) -\sigma k^2 \mid k \in \R\big\}\\
= & \max \big\{ \CF_\sigma \big(k, U(-h)\big) +g  \mid k \in \R\big\} \ge  \CF_\sigma (0, U(-h)) +g =0.    
\end{align*}
If it does, the bifurcation at $c= U(-h)$ was analyzed under the assumption $U'' \ne 0$ on $[-h, 0]$, which was only used to ensure a.) the $C^{1, \alpha}$ regularity of $\CF_\sigma$ near $c= U(-h)$ and b.) the sign $\p_c \CF_\sigma (k_0, U(-h)) <0$ if $ \CF_\sigma (k_0, U(-h))=0$. More details can be found in Theorem 1.1 in \cite{LiuZ21}. Thanks to Corollary \ref{C:F} and Lemma \ref{L:pcF}, these two key properties of $\CF_\sigma$ also hold without assuming $U'' \ne 0$. Therefore the following proposition on the eigenvalue distribution of the linearized capillary gravity waves holds with weakened assumptions.

\begin{proposition} \label{P:CGWW}
Assume $U\in C^6$ and $U'>0$ on $[-h, 0]$, then the following hold. 
\begin{enumerate} 
%\item If $g> g_\#$ then $\CF_\sigma (k, U(-h))\ne 0$ for all $k \in \R$ and $c^-(k)$ can be continued as simple roots for each $k \in \R$ with $c^-(k) < U(-h)$ and even and analytic in $k$.
\item If $g=g_\#(\sigma)$, then there exist $k_\#>0$, unique in $[0, +\infty)$, and an even $C^{1, \alpha}$ function (for any $\alpha \in [0, 1)$) $c^- (k)$ defined for all $k \in \R$ such that 
\begin{enumerate} 
\item $\CF_\sigma (k, U(-h))=0$ iff $k = \pm k_\#$;
\item $c^- (k) < U(-h)$ is analytic in $k \ne \pm k_\#$, and $c^-(\pm k_\#) = U(-h)$;
\item $c^+(k)$ and $c^-(k)$ are simple roots of $\CF_\sigma (k, \cdot)$ which include all roots of $\CF_\sigma (k, \cdot)$ outside the disk \eqref{E:semi-circle};  and  
\item there exists $\ep, \rho>0$ that $\CF_\sigma (k, c) =0$ with $|k \pm k_\#|\le \ep$ and $|c-U(-h)| \le \rho$ iff $c = c^- (k)$. 
\end{enumerate}
\item If $g< g_\# (\sigma)$, then there exist $k_\#^+ > k_\#^->0$, $\ep, \rho>0$ and even $C^{1, \alpha}$ functions (for any $\alpha \in [0, 1)$) $c^- (k)$ defined for $|k| \ge k_\#^+ -\ep$ and $c_-(k)$ for $|k| \le k_\#^- + \ep$ satisfying the following.   
\begin{enumerate}
\item $\CF_\sigma (k, U(-h))=0$ iff $k =\pm  k_\#^\pm$.
\item $c^- (k)< U(-h)$ is analytic in $k$ if $|k| > k_\#^+$ and $c_-(k) < U(-h)$ is analytic in $k$ if $|k| < k_\#^-$. Moreover, $c^+(k)$, $c^-(k)$ with $|k| > k_\#^+$, and $c_-(k)$ with $|k| < k_\#^-$ are the only roots of $\CF_\sigma $ outside the disk \eqref{E:semi-circle}, which are all simple. 
\item $c^-(\pm k_\#^+) = U(-h) =  c_-(\pm k_\#^-)$, $\mp (c_R^-)' (\pm k_\#^+)>0$, $\pm \RP\, c_-' (\pm k_\#^-)>0$.   
\item $c_I^- (k)$ with $0\le  k_\#^+ - |k| \le \ep$ and $\IP\, c_- (k)$ with $0\le |k| - k_\#^- \le \ep$ have the same sign ($+, -$, or $0$) as $U''\big(U^{-1} (c_R^- (k))\big)$ and  $U''\big(U^{-1} (\RP\, c_- (k))\big)$, respectively, and for such $k$, 
\[
C^{-1} \big|U''(U^{-1} (c_R^-(k))) \big| (|k| - k_\#^+)^{2} \le   |c_I^- (k)| \le C\big|U''(U^{-1} (c_R^-(k))) \big| ( |k| - k_\#^+)^{2},
\]
\[
C^{-1} \big|U''(U^{-1} (\RP\, c_-(k))) \big| (|k| - k_\#^-)^{2} \le   |\IP\, c_- (k)| \le C\big|U''(U^{-1} (\RP\, c_-(k)))\big| ( |k| - k_\#^-)^{2}.
\]
\item $\CF_\sigma (k, c=c_R+ic_I) =0$ with $\big||k| - k_\#^\pm\big|\le \ep$, $|c_R-U(-h)| \le \rho$, $c_I\ge 0$ iff $c=c^- (k)$ or $c=c_-(k)$. 
\end{enumerate} \end{enumerate}
\end{proposition}

\begin{remark} 
a.) The bifurcation near inflection values of $U$ can be found in Lemma \ref{L:bifurcation}. \\
b.) If $g> g_\#(\sigma)$, it had been proved in \cite{LiuZ21} that $c^-(k)$ can be extended for all $k \in \R$ as a simple root  of $\CF_\sigma (k, c\dot)$ which is the only root other than $c^+(k)$ outside the disk \eqref{E:semi-circle}). It also satisfies $c^- (k) < U(-h)$. An explicit necessary and sufficient condition for $g_\# (\sigma)=0$ was also given there. Under the assumption $U'' \ne 0$ a complete picture of the eigenvalue distribution in the fashion of Theorem \ref{T:e-values}(5) was also obtained in \cite{LiuZ21}, where the continuations of $c^-(k)$ and $c_-(k)$ coincide. \\
c.) Following from the same proof as in statement (6) of Theorem \ref{T:e-values}, one can prove that the spectral stability of the linearized capillary gravity waves. 
\end{remark}

\begin{proof}
Due to the concavity of $\CF_\sigma$ in $K=k^2$ (Lemma \ref{L:e-v-basic-2}(2)), for $g\le g_\#(\sigma)$, the existence of $k_\#>0$ or $k_\#^\pm>0$ as the only zero points of $\CF_\sigma (\cdot, U(-h))$ follows from the definition of $g_\#(\sigma)$. Much in the proof of Theorem \ref{T:e-values}(2), the branch $c_-(k)$ starts with the simple root $c_-(0) =c_2 < U(-h)$ of $\CF_\sigma(0, \cdot)$. For $|k| \gg 1$, another branch $c^-(k) <U(-h)$ had been constructed in Theorem 1.1 in \cite{LiuZ21}. Applying Lemma \ref{L:continuation} and the semi-circle theorem, $c_-(k)$ and $c^-(k)$ can be continued in $(-\infty, U(-h))$ as the only roots of $\CF_\sigma (k, \cdot)$ outside the disk \eqref{E:semi-circle} for all $|k| < k_\#^-$ and $|k| > k_\#^+$, respectively ($k_\#^\pm =k_\#$ is understood if $g = g_\#(\sigma)$), where $\CF_\sigma (k, U(-h))=0$ for the first $k$ in the continuation. It suffices to prove $c_-(k_\#^-) = c^-(k_\#^+) = U(-h)$ and the rest of the proposition would follow from the same proof as of Theorem \ref{T:e-values}. 

Our strategy is to vary the parameter $\sigma\ge 0$ starting from $0$. From the definitions 
%\eqref{E:dispersion-CG} and 
of $\CF_\sigma$ and $g_\#(\cdot)$, if $g_\#(\sigma')>0$, then it is strictly decreasing in $\sigma'$ and there exists $\sigma_0\ge \sigma$ such that $g_\#(\sigma_0) = g$. We use the notations $k_\#^\pm( \sigma')$ for $\sigma' \in (0, \sigma_0)$ and $k_\#$ for $g = g_\# (\sigma')$ which occurs only when $\sigma'=\sigma_0$. From the concavity of $\CF_{\sigma'} (\cdot, U(-h))$ in $K=k^2$, we have  
\[
\mp \p_\sigma k_\#^\pm (\sigma') = \pm (\p_\sigma \CF_{\sigma'} / \p_k \CF_{\sigma'}) ( k_\#^\pm(\sigma'), U(-h)) >0,  \quad \sigma' \in (0, \sigma_0).
\]
They also clearly satisfy 
\[
k_\#^+(\sigma')> k_\# > k_\#^-(\sigma') >0, \;\; \lim_{\sigma' \to \sigma_0-} k_\#^\pm (\sigma') = k_\# (\sigma_0), \;\; \lim_{\sigma' \to 0+} k_\#^- (\sigma') = k_-, \;\; \lim_{\sigma' \to 0+} k_\#^+ (\sigma') = +\infty,
\] 
where $k_->0$ is given in Theorem \ref{T:e-values}(2) satisfying $F(k_-, U(-h))=0$. 

Let us use $c^-(\sigma', k)$ and $c_-(\sigma', k)$ to denote the branches of the roots in $(-\infty, U(-h))$ of $\CF_{\sigma'} (k, \cdot)$ for $\sigma' \in (0, \sigma_0]$, while we skip $\sigma'$ if there is no confusion. According to Theorem \ref{T:e-values}(2), $c_-(0, 0) < U(-h)$ can be extended as the only root $c_-(0, k)$ in $(-\infty, U(-h)]$ of $\CF_0 (k, \cdot) = F(k, \cdot)$ for all $|k| < k_-$ and $\lim_{k \to (k_-)-} c_-(0, k) = U(-h)$. 
%Since $\p_k F(k_-, U(-h))>0$ due to Lemma \ref{L:e-v-basic-2}(2), 
The continuation in the direction of $\sigma'$ through the Implicit Function Theorem applied in a neighborhood of $U(-h)$ implies that, for each $\sigma' \in [0, \sigma_0)$, $c_-(\sigma', k)$ can be extended until it reaches $U(-h)$ which has to occur at $k= k_\#^- (\sigma')$. 

Let 
\[
\Omega = \{ c\in \C \mid c_R \in (-\infty, U(-h)), \, c_I \in (-1, 1)\}.
\]
Lemma \ref{L:e-v-large}(2) implies that $\CF_{\sigma'} (k, \cdot)$ has no roots in $\Omega$ for $-c_R\gg 1$. For any $\sigma' \in (0, \sigma_0)$, the only roots of $\CF_{\sigma'}$ on $\R \times \p \Omega$ are $(k_\#^\pm (\sigma), U(-h))$ due to the semi-circle theorem. Hence from an index argument applied to  sufficiently large compact subsets of $\Omega$, the total number of roots (counting the multiplicity) of $\CF_{\sigma'} (k, \cdot)$ in $\Omega$ are constants for $k$ in $[0, k_\#^-(\sigma'))$, $(k_\#^-(\sigma'), k_\#^+(\sigma'))$, and $(k_\#^+(\sigma'), +\infty)$, respectively. Since $c_-(0, 0)$ is the only root in $\Omega$ for $(\sigma', k)=(0,0)$ and $c^-(\sigma', k)$ the only one for $\sigma'\in (0, \sigma_0)$ and $k \gg 1$, the total numbers of roots of $\CF_{\sigma'} (k, \cdot)$ in $\Omega$  is $1$ for $k$ in both $[0, k_\#^-(\sigma'))$ and $(k_\#^+(\sigma'), +\infty)$. As $c_-(\sigma', k_\#^- (\sigma')) = U(-h)$, by the bifurcation analysis based on Lemma \ref{L:bifurcation} and the signs $\p_c \CF_{\sigma'} (k_\#^- (\sigma'), U(-h))<0$ (Lemma \ref{L:pcF}) and $\p_k \CF_{\sigma'} (k_\#^- (\sigma'), U(-h))>0$ (from the concavity of $\CF_{\sigma'}$ in $K=k^2$), $\CF_{\sigma'} (k, \cdot)$ has no roots in $\Omega$  for $1\gg k-k_\#^-(\sigma') >0$. Hence $\CF_{\sigma'} (k, \cdot)$ has no roots in $\Omega$ for $k \in (k_\#^-(\sigma'), k_\#^+(\sigma'))$. For the total number of roots in $\Omega$ of $\CF_{\sigma'} (k, \cdot)$ to change from $0$ to $1$ as $k$ increases through $k_\#^+(\sigma')$, it must hold $c^-(\sigma', k_\#^+(\sigma')) = U(-h)$, otherwise $c^-(\sigma', k_\#^+(\sigma'))$ would be continued to a root in $\Omega$ for $k < k_\#^+ (\sigma')$ which is a contradiction. 

Finally, $c^-(\sigma_0, k_\#) = c_-(\sigma_0, k_\#)=U(-h)$ follows from taking the limit as $ \sigma' \to \sigma_0-$.   
\end{proof}

A direct corollary of the proposition is a sufficient condition for the linear instability.  

\begin{corollary}
If $g < g_\#$ and there exists a sequence $x_{2, n} \in (-h, 0)$ converging to $-h$ as $n \to +\infty$ such that $U''(x_{2,n}) >0$ for all $n$, then the linearized capillary gravity wave system is linearly unstable. 
\end{corollary}

\begin{remark}
For the linearized capillary gravity wave problem, a non-degenerate inflection value $c_0$ of $U$ does not necessarily leads to instability as a strong surface tension may prevent $c_0$ from becoming a neutral mode at all. 
\end{remark}

\bibliographystyle{abbrv}
\bibliography{bibliography}

\begin{thebibliography}{10}

\bibitem{Ben62}
T.~B. Benjamin.
\newblock The solitary wave on a stream with an arbitrary distribution of
  vorticity.
\newblock {\em J. Fluid Mech.}, 12:97--116, 1962.

\bibitem{BR13}
D.~Bresch and M.~Renardy.
\newblock Kelvin-{H}elmholtz instability with a free surface.
\newblock {\em Z. Angew. Math. Phys.}, 64(4):905--915, 2013.

\bibitem{BSWZ16}
O.~Buhler, J.~Shatah, S.~Walsh, and C.~Zeng.
\newblock On the wind generation of water waves.
\newblock {\em Arch. Ration. Mech. Anal.}, 222:827 -- 878, 2016.

\bibitem{Burns53}
J.~C. Burns.
\newblock Long waves in running water.
\newblock {\em Proc. Cambridge Philos. Soc.}, 49:695--706, 1953.
\newblock With an appendix by M. J. Lighthill.

\bibitem{DR04}
P.~Drazin and W.~Reid.
\newblock Hydrodynamic stability.
\newblock {\em Cambridge: Cambridge University Press.}, 2004.

\bibitem{Fj50}
R.~Fj\o~rtoft.
\newblock Application of integral theorems in deriving criteria of stability
  for laminar flows and for the baroclinic circular vortex.
\newblock {\em Geofys. Publ. Norske Vid.-Akad. Oslo}, 17(6):52, 1950.

\bibitem{FH98}
S.~Friedlander and L.~Howard.
\newblock Instability in parallel flows revisited.
\newblock {\em Stud. Appl. Math.}, 101(1):1--21, 1998.

\bibitem{How61}
L.~Howard.
\newblock Note on a paper of john w. miles.
\newblock {\em J. Fluid. Mech.}, 10:509--512, 1961.

\bibitem{Hunt55}
J.~N. Hunt.
\newblock Gravity waves in flowing water.
\newblock {\em Proc. Roy. Soc. London Ser. A}, 231:496--504, 1955.

\bibitem{HL08}
V.~Hur and Z.~Lin.
\newblock Unstable surface waves in running water.
\newblock {\em Comm. Math Phys.}, 282:733 -- 796, 2008.

\bibitem{HL13}
V.~M. Hur and Z.~Lin.
\newblock Erratum to: {U}nstable surface waves in running water [mr2426143].
\newblock {\em Comm. Math. Phys.}, 318(3):857--861, 2013.

\bibitem{Jan04}
P.~Janssen.
\newblock {\em The interaction of ocean waves and wind}.
\newblock Cambridge University Press, 2004.

\bibitem{KR11}
A.~Kaffel and M.~Renardy.
\newblock Surface modes in inviscid free surface shear flows.
\newblock {\em ZAMM Z. Angew. Math. Mech.}, 91(8):649--652, 2011.

\bibitem{LinC55}
C.~C. Lin.
\newblock {\em The theory of hydrodynamic stability}.
\newblock Cambridge, at the University Press, 1955.

\bibitem{Lin03}
Z.~Lin.
\newblock Instability of some ideal plane flows.
\newblock {\em SIAM J. Math. Anal.}, 35(2):318--356, 2003.

\bibitem{Liu22}
X.~Liu.
\newblock Instability and spectrum of the linearized two-phase fluids interface
  problem at shear flows.
\newblock {\em arXiv:2208.11159}, 2022.

\bibitem{LiuZ21}
X.~Liu and C.~Zeng.
\newblock Capillary gravity water waves linearized at monotone shear flows:
  eigenvalues and inviscid damping.
\newblock {\em arXiv:2110.12604}, 2021.

\bibitem{LH98}
M.~S. Longuet-Higgins.
\newblock Instabilities of a horizontal shear flow with a free surface.
\newblock {\em J. Fluid Mech.}, 364:147--162, 1998.

\bibitem{MP94}
C.~Marchioro and M.~Pulvirenti.
\newblock {\em Mathematical theory of incompressible nonviscous fluids},
  volume~96 of {\em Applied Mathematical Sciences}.
\newblock Springer-Verlag, New York, 1994.

\bibitem{Mi57}
J.~Miles.
\newblock On the generation of surface waves by shear flows.
\newblock {\em J. Fluid Mech.}, 3:185--204, 1957.

\bibitem{Mi59a}
J.~W. Miles.
\newblock On the generation of surface waves by shear flows. {II}.
\newblock {\em J. Fluid Mech.}, 6:568--582, 1959.

\bibitem{Mi59b}
J.~W. Miles.
\newblock On the generation of surface waves by shear flows. {III}.
  {K}elvin-{H}elmholtz instability.
\newblock {\em J. Fluid Mech.}, 6:583--598. (1 plate), 1959.

\bibitem{Ray1880}
L.~Rayleigh.
\newblock On the stability or instability of certain fluid motions.
\newblock {\em Pro. London Math. Soc.}, 9:57--70, 1880.

\bibitem{RR13}
M.~Renardy and Y.~Renardy.
\newblock On the stability of inviscid parallel shear flows with a free
  surface.
\newblock {\em Math. Fluid Mech.}, 15:129--137, 2013.

\bibitem{Shr93}
V.~I. Shrira.
\newblock Surface waves on shear currents: solution of the boundary value
  problem.
\newblock {\em J. Fluid Mech.}, 252:565--584, 1993.

\bibitem{kelvin1871hydro}
W.~L.~K. Thomson.
\newblock Hydrokinetic solutions and observations.
\newblock {\em Philosophical Magazine Series 4}, (42):362--377, 1871.

\bibitem{To35}
W.~Tollmien.
\newblock Ein allgemeines kriterium der instabititat laminarer
  geschwindigkeitsverteilungen.
\newblock {\em Nachr. Ges. Wiss. Gottingen Math. Phys.}, 50:79--114, 1935.

\bibitem{Yih72}
C.-S. Yih.
\newblock Surface waves in flowing water.
\newblock {\em J. Fluid. Mech.}, 51:209--220, 1972.

\end{thebibliography}
\endgroup
\end{document}